\documentclass[11pt,reqno]{article}
\usepackage{amsmath,amssymb,amsthm}
\usepackage{color,url}
\newtheorem{theorem}{Theorem}[section]
\newtheorem{proposition}[theorem]{Proposition}
\newtheorem{corollary}[theorem]{Corollary}
\newtheorem{definition}[theorem]{Definition}
\newtheorem{lemma}[theorem]{Lemma}
\newtheorem{example}[theorem]{Example}
\newtheorem{remark}[theorem]{Remark}

\numberwithin{equation}{section}

\usepackage[dvipdfmx]{hyperref}

\def\bR{\mathbb{R}}
\def\bC{\mathbb{C}}
\def\bN{\mathbb{N}}
\def\cH{\mathcal{H}}
\def\cC{\mathcal{C}}
\def\cK{\mathcal{K}}
\def\cP{\mathcal{P}}

\def\cD{\mathcal{D}}
\def\cB{\mathcal{B}}
\def\sa{\mathrm{sa}}
\def\OC{\mathrm{OC}}

\def\Tr{\mathrm{Tr}}

\def\eps{\varepsilon}
\def\<{\langle}
\def\>{\rangle}
\def\ffi{\varphi}
\def\lb{\mathrm{lb}}
\def\ran{\mathrm{ran}}
\def\half{\textstyle{1\over2}}

\begin{document}

\centerline{\LARGE Pusz--Woronowicz functional calculus and}
\medskip
\centerline{\LARGE extended operator convex perspectives}

\bigskip
\bigskip
\centerline{\large
Fumio Hiai\footnote{{\it E-mail:} hiai.fumio@gmail.com},
Yoshimichi Ueda\footnote{{\it E-mail:} ueda@math.nagoya-u.ac.jp}
and Shuhei Wada\footnote{{\it E-mail:} wada@j.kisarazu.ac.jp}}

\medskip
\begin{center}
$^1$\,Graduate School of Information Sciences, Tohoku University, \\
Aoba-ku, Sendai 980-8579, Japan
\end{center}

\begin{center}
$^2$\,Graduate School of Mathematics, Nagoya University, \\
Furocho, Chikusaku, Nagoya 464-8602, Japan
\end{center}

\begin{center}
$^3$\,Department of Information and Computer Engineering, \\
National Institute of Technology (KOSEN), Kisarazu College, \\
Kisarazu, Chiba 292-0041, Japan
\end{center}

\medskip

\begin{abstract}

In this article, we first study, in the framework of operator theory, Pusz and Woronowicz's
functional calculus for pairs of bounded positive operators on Hilbert spaces associated with a
homogeneous two-variable function on $[0,\infty)^2$. Our construction has special features that
functions on $[0,\infty)^2$ are assumed only locally bounded from below and that the functional
calculus is allowed to take extended semibounded self-adjoint operators. To analyze convexity
properties of the functional calculus, we extend the notion of operator convexity for real functions
to that for functions with values in $(-\infty,\infty]$. Based on the first part, we generalize the concept
of operator convex perspectives to pairs of (not necessarily invertible) bounded positive operators
associated with any operator convex function on $(0,\infty)$. We then develop theory of such
operator convex perspectives, regarded as an operator convex counterpart of Kubo and Ando's theory
of operator means. Among other results, integral expressions and axiomatization are discussed for
our operator perspectives.



\end{abstract}

{\baselineskip10pt
\tableofcontents
}

\section{Introduction}\label{Sec-1}

The concept of operator (convex) perspectives has not been studied so far in full generality.
Namely, when given positive operators are not invertible, the existing definition of operator
perspectives does not work for many interesting operator convex functions such as $t\log t$ and
$t^\alpha$ ($1 < \alpha \le 2$), because the corresponding (scalar-valued) perspective functions
are not locally bounded and moreover take $\infty$ on $[0,\infty)^2$. This difficulty appears even
in the finite-dimensional setting, but never does in theory of operator connections/means, viewed
as operator perspectives corresponding to operator monotone functions on $[0,\infty)$. The present
work attempts to overcome this drawback of the current operator perspective theory, by allowing
an operator perspective in question to be unbounded.

Theory of operator connections/means has grown into a significant subject of operator theory.
The theory has its origin in a study of parallel sum motivated by electrical networks \cite{AD}.
Besides parallel sum (a half of harmonic mean), the most interesting and the most studied operator
mean is geometric mean, which was introduced in 1975 by Pusz and Woronowicz \cite{PW1} and
further discussed by Ando \cite{An0,An1}. In \cite{PW1} the authors developed a certain type of
functional calculus for positive sesquilinear forms on a complex vector space, which we call the
Pusz--Woronowicz (PW for short) functional calculus. For two positive sesquilinear forms
$\alpha,\beta$ on a complex vector space, the PW-functional calculus determines a new
sesquilinear form $\phi(\alpha,\beta)$ on the vector space in a canonical way, associated with a
given Borel function $\phi:[0,\infty)^2\to\bR$ that is locally bounded and homogeneous
(i.e., $\phi(\lambda x,\lambda y)=\lambda\phi(x,y)$ for $x,y,\lambda\ge0$). In particular,
the geometric mean $\sqrt{\alpha\beta}$ of $\alpha,\beta$ was defined in \cite{PW1} as
$\phi(\alpha,\beta)$ for the function $\phi(x,y)=(xy)^{1/2}$, $x,y\in[0,\infty)$. Furthermore, in
\cite{PW2}, Pusz and Woronowicz considered the PW-functional calculus $\phi(\alpha,\beta)$ for
more general homogeneous functions $\phi$ on $[0,\infty)^2$ and characterized joint convexity of
$\phi(\alpha,\beta)$ in terms of operator convexity properties of $\phi(x,y)$.

In 1980, Kubo and Ando \cite{KA} proposed an axiomatic approach (see also the beginning of
\S\ref{Sec-10} of this paper) to a general theory of operator means (and connections) in a close
relation to L\"{o}wner's theory \cite{Lo} on operator monotone functions. In fact, Kubo and Ando's
operator connections $\sigma$ correspond one-to-one to non-negative operator monotone
functions $h$ on $[0,\infty)$ in such a way that
\begin{align}\label{F-1.1}
A\sigma B=A^{1/2}h(A^{-1/2}BA^{-1/2})A^{1/2}
\end{align}
for $A,B\in B(\cH)_+$, the bounded positive operators on a Hilbert space $\cH$, with $A$ invertible,
which is further extended to general $A,B\in B(\cH)_+$ as
\begin{align}\label{F-1.2}
A\sigma B=\lim_{\eps\searrow0}(A+\eps I)\sigma(B+\eps I)
\end{align}
in the strong operator topology.

More recently, in \cite{Ef,ENG}, the notion of operator perspectives associated with real continuous
functions $f$ on $(0,\infty)$ was introduced, similarly to \eqref{F-1.1}, as
\begin{align}\label{F-1.3}
P_f(A,B)=B^{1/2}f(B^{-1/2}AB^{-1/2})B^{1/2}
\end{align}
for $A,B\in B(\cH)_{++}$, the invertible operators in $B(\cH)_+$, though the roles of $A,B$ are
interchanged. It was proved that $P_f(A,B)$ is jointly operator convex on
$B(\cH)_{++}\times B(\cH)_{++}$ if and only if $f$ is operator convex on $(0,\infty)$. The extension
problem of operator perspectives $P_f(A,B)$ to general $A,B\in B(\cH)_+$ is rather complicated
and has not been well studied so far, while a few discussions are in \cite[\S2]{FS} and
\cite[\S6]{HSW}. In fact, when $f$ is operator convex on $(0,\infty)$, the limit of
$P_f(A+\eps I,B+\eps I)$ as $\eps\searrow0$ for $A,B\in B(\cH)_+$ does not always exist as a
bounded operator, unlike \eqref{F-1.2} for operator connections $A\sigma B$.

Recently in \cite{HU}, Hatano and the second-named author of this paper considered the
PW-functional calculus in the framework of operator theory. When positive
sesquilinear forms $\alpha,\beta$ on $\cH$ are defined, for given operators $A,B\in B(\cH)_+$, by
$\alpha(\xi,\eta):=\<A\xi,\eta\>$ and $\beta(\xi,\eta):=\<B\xi,\eta\>$ for $\xi,\eta\in\cH$, the
PW-functional calculus $\phi(\alpha,\beta)$ is described in terms of $A,B$ as follows. With a
bounded operator $T_{A,B}:\cH\to\cH_{A,B}:=\overline{\ran}(A+B)$ defined by
$T_{A,B}\xi:=(A+B)^{1/2}\xi$, $\xi\in\cH$, we have $R_{A,B},S_{A,B}\in B(\cH_{A,B})_+$
such that $R_{A,B}+S_{A,B}=I_{\cH_{A,B}}$, $A=T_{A,B}^*R_{A,B}T_{A,B}$ and
$B=T_{A,B}^*S_{A,B}T_{A,B}$. Then $(T_{A,B}:\cH\to\cH_{A,B},R_{A,B},S_{A,B})$ is a
compatible representation of $(\alpha,\beta)$ in the sense of \cite{PW1}, and $\phi(\alpha,\beta)$
clearly coincides with
\begin{align}\label{F-1.4}
\phi(A,B):=T_{A,B}^*\phi(R_{A,B},S_{A,B})T_{A,B}
\end{align}
if $\phi$ is a locally bounded and homogeneous Borel function on $[0,\infty)^2$, where
$\phi(R_{A,B},S_{A,B})$ denotes the usual Borel functional calculus of commuting
$R_{A,B},S_{A,B}$. As clarified in \cite{HU}, Kubo and Ando's operator connections $A\sigma B$
are captured by the PW-functional calculus, that is, $A\sigma B=\phi(B,A)$ if $\phi$ is the
two-variable extension (i.e., the perspective function) of the representing function $h$ in
\eqref{F-1.1}. Similarly, operator perspectives $P_f(A,B)$ for $A,B\in B(\cH)_{++}$ are realized as
$\phi(A,B)$ with the perspective function $\phi$ of $f$. The convexity criteria given in \cite{PW2}
were also examined in \cite[Theorem 9]{HU}, so that the joint convexity assertion in \cite{Ef,ENG}
may be considered as a specialized version of the result of \cite{PW2}. Moreover, it was observed
in \cite{HU} that operator homogeneity (see Definition \ref{D-4.1}(2) of this paper) generally holds
for the PW-functional calculus, while this property was formerly shown in \cite{Fu2} for operator
means.

Now we explain our aims of the present paper in the following two items:

(1)\enspace
Assume, for instance, that $\phi$ is the perspective function of a real function $f$ on $(0,\infty)$
and $\phi$ is extended to $[0,\infty)^2$ by continuity, that is, $\phi(x,y):=yf(x/y)$ for $x,y>0$,
$\phi(x,0):=\alpha x$ for $x\ge0$ and $\phi(0,y):=\beta y$ for $y\ge0$, where
$\alpha:=\lim_{t\to\infty}f(t)/t$ and $\beta:=\lim_{t\searrow0}f(t)$ (whose limits in $(-\infty,\infty]$
are here assumed to exist). In this situation, $\phi$ is not necessarily $\bR$-valued on the
boundary of $[0,\infty)^2$, i.e., on $\{0\}\times(0,\infty)$ and $(0,\infty)\times\{0\}$. We are then
motivated to extend the PW-functional calculus $\phi(A,B)$ discussed in \cite{HU} to locally
lower bounded and homogenous functions $\phi$ on $[0,\infty)^2$ having extended values in
$(-\infty,\infty]$. Here we are forced to allow $\phi(A,B)$ to be unbounded. Thus we will formulate 
the PW-functional calculus associated with $\phi$ as a two-variable mapping from
$B(\cH)_+\times B(\cH)_+$ to the extended lower semibounded self-adjoint part
$\widehat{B(\cH)}_\lb$ of $B(\cH)$, which is slightly bigger than the extended positive part
$\widehat{B(\cH)}_+$ of $B(\cH)$ in the sense of \cite{Ha}. Our primary aim is to carry out the
account of the PW-functional calculus in \cite{HU} in this extended setting. The main target here
is to obtain a complete set of convexity criteria of $\phi(A,B)$ generalizing those in \cite{HU}. This
will consequently justify the use of $\widehat{B(\cH)}_\lb$ in our formulation.

(2)\enspace
For an operator convex function $f$ on $(0,\infty)$, we can define the operator perspective
$\phi_f(A,B)$ (as an element of $\widehat{B(\cH)}_\lb$) for all $A,B\in B(\cH)_+$ based on the
PW-functional calculus constructed in (1). That is, $\phi_f(A,B)$ is defined to be the
PW-functional calculus $\phi(A,B)$ associated with the perspective function $\phi$ of $f$
(extended to $[0,\infty)^2$ by continuity as mentioned in (1)). Then $\phi_f(A,B)$ coincides with
$P_f(A,B)$ in \eqref{F-1.3} if $A,B$ are invertible, so the extension problem mentioned after
\eqref{F-1.3} is indirectly settled because $\phi_f(A,B)$ is already defined for all $A,B\in B(\cH)_+$,
though the values of $\phi_f(A,B)$ are not necessarily bounded operators. We, instead, have to
consider the problem of characterizing when $\phi_f(A,B)$ is bounded. Our second aim is to
develop theory of extended operator convex perspectives along the lines of Kubo and Ando's
theory of operator connections. We will consider, for example, their integral expressions and
axiomatization of Kubo and Ando's type.

We end the introduction with a brief summary of contents of the paper. Section \ref{Sec-2}
is a preliminary on the extended lower semibounded self-adjoint part $\widehat{B(\cH)}_\lb$,
and Section \ref{Sec-3} gives basics of extended $(-\infty,\infty]$-valued operator convex
functions for later use. In Section \ref{Sec-4}, extending discussions in \cite{HU}, we introduce 
and study the PW-functional calculus $\phi(A,B)$ of $A,B\in B(\cH)_+$ associated with 
such a function $\phi$ on $[0,\infty)^2$ as stated in (1) above. The definition of $\phi(A,B)$ is 
given in an axiomatic fashion with two postulates (see Definition \ref{D-4.1}), while an explicit 
definition like \eqref{F-1.4} is also possible. The main result (Theorem \ref{T-4.9}) gives 
characterizations for $\phi(A,B)$ to be jointly convex in $(A,B)$. In Section \ref{Sec-5} 
we consider the PW-functional calculus $\phi(A,B)$ associated with a function $\phi$ on
$[0,\infty)^2\setminus(\{0\}\times(0,\infty))$ with a restricted domain of
$(A,B)\in B(\cH)_+\times B(\cH)_+$ such that $A\ge\alpha B$ for some $\alpha>0$.
Section \ref{Sec-6} establishes the continuity of $\phi(A_n,B_n)\to\phi(A,B)$ in the
strong operator topology for decreasing $A_n\searrow A$, $B_n\searrow B$ in $B(\cH)_+$
when $\phi$ is $\bR$-valued and continuous on $[0,\infty)^2$. We believe that such a general
continuity property for the PW-functional calculus in its original form has not 
been examined so far.

In the second part, we study operator perspectives $\phi_f(A,B)$ associated with an operator
convex function $f$ on $(0,\infty)$, that is, the PW-functional calculus associated with the
perspective function of $f$. In Section \ref{Sec-7} we discuss (semi-)continuity properties of
$\phi_f(A,B)$. The main result (Theorem \ref{T-7.7}), in particular, says that the approach taking
limit as in \eqref{F-1.2} is also available for $\phi_f(A,B)$. It is also shown (Proposition \ref{P-7.13})
that when $A,B$ are positive trace-class operators, $\Tr\,\phi_f(A,B)$ is well defined and coincides
with the maximal $f$-divergence \cite{Hi2} of $A,B$. In Section \ref{Sec-8} we examine the cases
when $\phi_f(A,B)$ is bounded and when $\phi_f(A,B)$ has a dense domain, i.e., $\phi_f(A,B)$ is
a densely-defined self-adjoint operator on $\cH$. Interestingly, this problem in the case $f(t)=t^2$
is strongly related to absolute continuity between positive operators \cite{An2}. Furthermore,
Section \ref{Sec-9} treats integral expressions and variational expressions of $\phi_f(A,B)$ based
on integral expressions of $f$. Finally, Section \ref{Sec-10} gives some axiomatization results,
including a new axiomatization of operator connections different from the familiar one in \cite{KA}.

\section{Preliminary on the extended lower semibounded self-adjoint part}\label{Sec-2}

Throughout this article, let $\cH$ be a Hilbert space and $B(\cH)$ be the set of all bounded
operators on $\cH$. We use the notations $B(\cH)_\sa$, $B(\cH)_+$, and $B(\cH)_{++}$ for the
sets of self-adjoint operators, positive operators, and positive invertible operators in $B(\cH)$,
respectively. In this preliminary section, we briefly describe unbounded objects extending
self-adjoint operators for later use.

Note that $B(\cH)$ is a von Neumann algebra with the predual $B(\cH)_*\cong\cC_1(\cH)$, the
space of trace-class operators on $\cH$ with trace-norm. Here we identify $\rho\in\cC_1(\cH)$
with a normal functional $\rho(X)=\Tr\,X\rho$ for $X\in B(\cH)$, where $\Tr$ is the usual trace
on $B(\cH)$. The positivity $\rho\ge0$ in the operator sense is equivalent to the positivity of
$\rho$ in the functional sense, so we can further identify $B(\cH)_*^+=\cC_1(\cH)_+$, where
$B(\cH)_*^+$ and $\cC_1(\cH)_+$ are the positive parts of $B(\cH)_*$ and $\cC_1(\cH)$,
respectively.

The \emph{extended positive part} $\widehat{B(\cH)}_+$ of $B(\cH)$ (in the sense of Haagerup
\cite{Ha}) is the set of mappings $m:B(\cH)_*^+\to[0,\infty]$ that satisfy the following:
\begin{itemize}
\item[(1)] $m(\alpha\rho)=\alpha m(\rho)$ for all $\alpha\ge0$ and $\rho\in B(\cH)_*^+$ (with
usual convention $0\cdot\infty=0$),
\item[(2)] $m(\rho_1+\rho_2)=m(\rho_1)+m(\rho_2)$ for all $\rho_1,\rho_2\in B(\cH)_*^+$,
\item[(3)] $m$ is lower semicontinuous on $B(\cH)_*^+$.
\end{itemize}

This notion was originally introduced in \cite{Ha} in a more general setting to study operator
valued weights in theory of von Neumann algebras. Recently in \cite{Ko4}, the extended
positive part $\widehat{B(\cH)}_+$ was effectively used in a study of parallel sum of unbounded
positive operators. For our purpose it is convenient to slightly generalize $\widehat{B(\cH)}_+$
as follows:

\begin{definition}\label{D-2.1}\rm
We define \emph{the extended lower semibounded self-adjoint part} $\widehat{B(\cH)}_\lb$
of $B(\cH)$ to be the set of mappings $m:B(\cH)_*^+\to(-\infty,\infty]$ that satisfies, in addition to
the above (1)--(3), the following:
\begin{itemize}
\item[(4)] there exists an $\ell\geq 0$ such that $m(\rho)+\ell\rho(I)\ge0$ for all
$\rho\in B(\cH)_*^+$, where $I$ is the identity operator on $\cH$.
\end{itemize}
\end{definition}

The conic and the order structures of $\widehat{B(\cH)}_\lb$ are simply defined as follows.
Let $m,m_1,m_2\in\widehat{B(\cH)}_\lb$. Define $\alpha m$ ($\alpha\ge0$),
$m_1+m_2\in\widehat{B(\cH)}_\lb$ by $(\alpha m)(\rho):=\alpha m(\rho)$,
$(m_1+m_2)(\rho):=m_1(\rho)+m_2(\rho)$, and define $m_1\le m_2$ if $m_1(\rho)\le m_2(\rho)$
for all $\rho\in B(\cH)_*^+$. Clearly, $\widehat{B(\cH)}_+$ is included in
$\widehat{B(\cH)}_\lb$ as a sub-cone (since $\ell=0$ is available for condition (4) if
$m\in\widehat{B(\cH)}_+$). Each $A\in B(\cH)_\sa$ (resp., $A\in B(\cH)_+$) is regarded as an
element of $\widehat{B(\cH)}_\lb$ (resp., $\widehat{B(\cH)}_+$) in a natural way that
$A(\rho):=\rho(A)$ for $\rho\in B(\cH)_*^+$.

The next proposition is a slight modification of \cite[Theorem 1.5]{Ha}.

\begin{proposition}\label{P-2.2}
For every $m\in\widehat{B(\cH)}_\lb$ there exists a spectral resolution $(E_t)_{t\in\bR}$ on a
closed subspace $\cH_0$ of $\cH$, i.e., a one-parameter family of non-decreasing and
right-continuous orthogonal projections $E_t$ ($t\in\bR$) with $E_t\nearrow P_{\cH_0}$
as $t\to\infty$, such that $E_t=0$ ($t<\ell$) for some $\ell\in\bR$ and
\begin{align}\label{F-2.1}
m(\rho)=\int_{-\infty}^\infty t\,d\rho(E_t)+\infty\cdot\rho(P_{\cH_0^\perp}),
\qquad\rho\in B(\cH)_*^+,
\end{align}
where $P_{\cH_0}$ and $P_{\cH_0^\perp}$ are the projections onto $\cH_0$ and $\cH_0^\perp$,
respectively. Furthermore, $\cH_0$ and $(E_t)_{t\in\bR}$ are uniquely determined by $m$.
\end{proposition}

\begin{proof}
Let $m\in\widehat{B(\cH)}_\lb$ with $\ell\in\bR$ as in (4); then it is obvious that
$m-\ell I\in\widehat{B(\cH)}_+$. Hence by \cite[Theorem 1.5]{Ha} there is a spectral
resolution $(F_t)_{t\ge0}$ on a closed subspace $\cH_0$ of $\cH$ with $F_t\nearrow P_{\cH_0}$
as $t\to\infty$ such that
\begin{align}\label{F-2.2}
m(\rho)-\ell\rho(I)=\int_0^\infty t\,d\rho(F_t)+\infty\cdot\rho(P_{\cH_0^\perp}),
\qquad\rho\in B(\cH)_*^+.
\end{align}
Define a spectral resolution $(E_t)_{t\in\bR}$ on $\cH_0$ by
\[
E_t:=\begin{cases}0 & \text{for $t<\ell$}, \\
F_{t-\ell} & \text{for $t\ge\ell$}.\end{cases}
\]
Then for every $\rho\in B(\cH)_*^+$ we have
\begin{align*}
\int_{-\infty}^\infty t\,d\rho(E_t)+\infty\cdot\rho(P_{\cH_0^\perp})
&=\int_{[0,\infty)}(t+\ell)\,d\rho(F_t)+\infty\cdot\rho(P_{\cH_0^\perp}) \\
&=\int_{[0,\infty)}t\,d\rho(F_t)+\ell\rho(P_{\cH_0})+\infty\cdot\rho(P_{\cH_0^\perp}) \\
&=\int_{[0,\infty)}t\,d\rho(F_t)+\ell\rho(I)+\infty\cdot\rho(P_{\cH_0^\perp})=m(\rho)
\end{align*}
thanks to \eqref{F-2.2}. Hence \eqref{F-2.1} holds.

Next, let us show the uniqueness of $\cH_0$ and $(E_t)_{t\in\bR}$. By the proof of
\cite[Lemma 1.4]{Ha}, $\cH_0$ is the closure of
\[
\{\xi\in\cH:m(\omega_\xi)-\ell\|\xi\|^2<\infty\}
=\{\xi\in\cH:m(\omega_\xi)<\infty\},
\]
where $\omega_\xi:=\<\cdot\,\xi,\xi\>\in B(\cH)_*^+$, a vector functional. Hence $\cH_0$ is
uniquely determined by $m$. Define a lower semibounded self-adjoint operator $T$ on
$\cH_0$ by
\begin{align}\label{F-2.3}
T:=\int_{-\infty}^\infty t\,dE_t=\int_{[\ell,\infty)}t\,dE_t.
\end{align}
For every $\xi\in\cH$, by \eqref{F-2.1} we have
\[
m(\omega_\xi)=\int_{-\infty}^\infty t\,d\omega_\xi(E_t)+\infty\cdot\|P_{\cH_0}^\perp\xi\|^2
=\int_{[\ell,\infty)}t\,d\<E_t\xi,\xi\>+\infty\cdot\|P_{\cH_0}^\perp\xi\|^2,
\]
that is,
\begin{align}\label{F-2.4}
m(\omega_\xi)=\begin{cases}
\|(T-\ell I_{\cH_0})^{1/2}\xi\|^2+\ell\|\xi\|^2 &
\text{if $\xi\in\cD((T-\ell I_{\cH_0})^{1/2})$}, \\
\infty & \text{otherwise}.\end{cases}
\end{align}
This is a lower semicontinuous and lower semibounded quadratic form on $\cH$.
Here, recall that the operator $T$ is uniquely determined by the quadratic form \eqref{F-2.4};
see, e.g., \cite[\S VI.2]{Ka} and \cite[Chap.~10]{Sch}. Hence \eqref{F-2.4} determines $T$ so that
$(E_t)_{t\in\bR}$ is unique as the spectral resolution of $T$.
\end{proof}

We call $(E_t)_{t\in\bR}$  in Proposition \ref{P-2.2} the \emph{spectral resolution} of $m$,
$\cH_0$ the \emph{essential part} of $m$ and $\cH_0^\perp$ the \emph{$\infty$-part} of $m$.

Proposition \ref{P-2.2} and its proof show that each $m\in\widehat{B(\cH)}_\lb$ is associated
with a lower semibounded self-adjoint operator $T$ on a closed subspace of $\cH$. Conversely,
let $T$ be a lower semibounded self-adjoint operator on a closed subspace $\cH_0$ of $\cH$,
and let $T=\int_{-\infty}^\infty t\,dE_t$ be the spectral decomposition of $T$. Then it is easy to
see that $m:B(\cH)_*^+\to(-\infty,\infty]$ defined by \eqref{F-2.1} is indeed an element of
$\widehat{B(\cH)}_\lb$. Furthermore, a lower semibounded quadratic form $q_T$ on $\cH_0$
corresponding to $T$ is defined as in \eqref{F-2.4} by
\begin{align}\label{F-2.5}
q_T(\xi):=\begin{cases}
\|(T-\ell I_{\cH_0})^{1/2}\xi\|^2+\ell\|\xi\|^2 &
\text{if $\xi\in\cD((T-\ell I_{\cH_0})^{1/2})$}, \\
\infty & \text{otherwise},\end{cases}
\end{align}
with any choice of $\ell\in(-\infty,\min\sigma(T)]$, where $\sigma(T)$ is the spectrum of $T$.
Note (see, e.g., \cite[Chap.~10]{Sch}) that $q_T|_{\cD(q_T)}$ is a closed quadratic form, where
$\cD(q_T):=\{\xi\in\cH:q_T(\xi)<\infty\}$, the domain of $q_T$, or equivalently, $q_T$ is lower
semicontinuous on $\cH$. Since any $\rho\in B(\cH)_*^+$ is written as
$\rho=\sum_n\omega_{\xi_n}$ for some $\{\xi_n\}$ in $\cH$ with $\sum_n\|\xi_n\|^2<\infty$,
an $m\in\widehat{B(\cH)}_\lb$ is uniquely determined by expression \eqref{F-2.4}, i.e.,
$m(\omega_\xi)=q_T(\xi)$, $\xi\in\cH$.

Let $T_1,T_2$ be lower semibounded self-adjoint operators on closed subspaces $\cH_1,\cH_2$
of $\cH$, respectively. The order $T_1\le T_2$ (in the form sense) is defined if
$\cD(q_{T_2})\subseteq\cD(q_{T_1})$ and $q_{T_1}(\xi)\le q_{T_2}(\xi)$ for all $\xi\in\cD(q_{T_2})$.
It is known that $T_1\le T_2$ holds if and only if $(T_2-\lambda I)^{-1}\le(T_1-\lambda I)^{-1}$
for some (equivalently, for any) $\lambda\in\bR$ with $\lambda<\min\sigma(T_1)$ and
$\lambda<\min\sigma(T_2)$, where $(T_i-\lambda I)^{-1}$ is understood to be zero on
$\cH_i^\perp=\cD(q_{T_i})^\perp$ ($i=1,2$); see \cite[Corollary 10.13]{Sch}. Furthermore, the
\emph{form sum} $T:=T_1\,\dot+\,T_2$ is defined in such a way that
$\cD(q_T)=\cD(q_{T_1})\cap\cD(q_{T_2})$ and $q_T(\xi):=q_{T_1}(\xi)+q_{T_2}(\xi)$ for every
$\xi\in\cD(q_T)$; see \cite[\S IV.1.6]{Ka} and \cite[Proposition 10.22]{Sch}.

Summing up the discussions so far, we conclude that there are bijective
correspondences $m\leftrightarrow T\leftrightarrow q$ between the following three objects:
\begin{itemize}
\item elements $m$ of the extended lower semibounded self-adjoint part $\widehat{B(\cH)}_\lb$
(Definition \ref{D-2.1}),
\item lower semibounded self-adjoint operators $T$ on closed subspaces of $\cH$,
\item lower semicontinuous, lower semibounded quadratic forms $q$ on $\cH$ (with not
necessarily dense domains).
\end{itemize}
The correspondence $m\leftrightarrow T$ is determined by \eqref{F-2.1} and \eqref{F-2.3}
(also \eqref{F-2.4}), and $T\leftrightarrow q=q_T$ is given by \eqref{F-2.5}. These correspondences
preserve order and sum (described before Proposition \ref{P-2.2} and in the last paragraph). Below
we use the symbol $T$ to denote elements of $\widehat{B(\cH)}_\lb$ with identification between
$m\leftrightarrow T$.

The following wording will be convenient in \S7.

\begin{definition}\label{D-2.3}\rm
Let $T\in\widehat{B(\cH)}_\lb$. We briefly say that $T$ \emph{has a dense domain} if the
$\infty$-part of $T$ is trivial, that is, $T$ is densely defined on $\cH$ as a lower semibounded
self-adjoint operator, or equivalently $\cD(q_T)$ is dense in $\cH$. Also, we say that $T$ is
\emph{bounded} if it is a bounded self-adjoint operator on $\cH$.
\end{definition}

In particular, when $\cH$ is finite-dimensional with $n=\dim\cH$, each $T\in\widehat{B(\cH)}_\lb$
is represented in the form of diagonalization $T=\sum_{i=1}^n\lambda_iP_{\xi_i}$ with an
orthonormal basis $(\xi_i)_{i=1}^n$ of $\cH$ and
$-\infty<\lambda_1\le\lambda_2\le\dots\le\lambda_n\le\infty$ (with value $\infty$ allowed),
where $P_{\xi_i}$ is the rank-one projection onto $\bC\xi_i$. Therefore, considering
$\widehat{B(\cH)}_\lb$ is already non-trivial even in the finite-dimensional setting.

We end the section with two basic lemmas, which will be used in subsequent sections.

\begin{lemma}\label{L-2.4}
Let $T\in\widehat{B(\cH)}_\lb$ and $C$ be a bounded operator from another Hilbert space $\cK$
to $\cH$. Then the mapping
\[
\rho\in B(\mathcal{K})_*^+\,\longmapsto\,T(C\rho\,C^*) \in (-\infty,\infty]
\]
defines an element $C^*TC \in \widehat{B(\cK)}_\lb$, where we use the standard notation
$(C\rho\,C^*)(X) := \rho(C^* X C)$ for $X \in B(\cH)$.
\end{lemma}

\begin{proof}
Remark that $C\rho\,C^*$ falls into $B(\cH)_*^+$. Thus, $T(C\rho\,C^*)\in(-\infty,\infty]$ is well
defined. For any $\rho,\rho_1,\rho_2 \in B(\mathcal{K})_*^+$ and $\alpha \geq 0$, we have
\begin{align*}
T(C(\alpha\rho)C^*) &= T(\alpha(C\rho\,C^*)) = \alpha T(C\rho\,C^*), \\ 
T(C(\rho_1+\rho_2)C^*) &= T((C\rho_1 C^*)+(C\rho_2 C^*)) = T(C\rho_1 C^*)+T(C\rho_2 C^*),
\end{align*}
since $C(\alpha\rho)C^*=\alpha(C\rho\,C^*)$ and
$C(\rho_1+\rho_2)C^*=(C\rho_1 C^*)+(C\rho_2 C^*)$ obviously hold. Moreover, if
$\|\rho_n-\rho\|\to0$ in $B(\mathcal{K})_*^+$, then $\|C\rho_nC^*-C\rho\,C^*\|\to0$ in $B(\cH)_*^+$
and hence $T(C\rho\,C^*) \leq \liminf_n T(C\rho_nC^*)$. Finally, choosing an $\ell > 0$ such that
$T(\rho') + \ell\,\rho'(I_\cH)\ge0$ for all $\rho' \in B(\cH)_*^+$, we have
\[
T(C\rho\,C^*) + \ell\|C\|^2\rho(I_\cK) \ge T(C\rho\,C^*)+\ell(C\rho\,C^*)(I_\cH)\ge0
\] 
for all $\rho \in B(\mathcal{K})_*^+$, where $\|C\|$ denotes the operator norm of $C$.
\end{proof}

\begin{lemma}\label{L-2.5}
Let $E$ be a spectral measure in $\cH$ on a measurable space $\Omega$. Let
$f : \Omega \to (-\infty,\infty]$ be a measurable function, and assume that $f$ is bounded from
below on $\Omega$. Then the mapping
\[
\rho \in B(\cH)_*^+\,\longmapsto\,\int_\Omega f(\omega)\,d\rho(E(\omega))
\in (-\infty,\infty]
\]
defines an element of $\widehat{B(\cH)}_\lb$.
\end{lemma}

\begin{proof}
For each $n\in\bN$, $T_n := \int_{f^{-1}((-\infty,n])} f(\omega)\,dE(\omega)$ defines an element of
$B(\cH)_\sa$, since $f$ is bounded from below. Moreover, it is clear by definition that 
\[
T_{n+1} - T_n = \int_{f^{-1}((n,n+1])} f(\omega)\,dE(\omega)
\geq n E(f^{-1}((n,n+1])) \geq 0. 
\]
For each $\rho \in B(\cH)_*^+$ we have 
\[
\int_\Omega f(\omega)\,d\rho(E(\omega))
= \sup_{n\geq 1} \int_{f^{-1}((-\infty,n])} f(\omega)\,d\rho(E(\omega))
= \sup_{n\geq1} \rho(T_n) 
\]
by the monotone convergence theorem (which is applicable because $f$ is bounded from below).
The desired assertion immediately follows from the above expression. 
\end{proof}

We write $\int_\Omega f\,dE$ for the element of $\widehat{B(\cH)}_\lb$ given in Lemma {\ref{L-2.5}.
When $f$ is an $\bR$-valued measurable function on $\Omega$ bounded from below, it is
immediate to see that $\int_\Omega f\,dE$ is a lower semibounded self-adjoint operator defined
by the usual spectral integral. In particular, $\int_\Omega f\,dE\in B(\cH)_\sa$ if $f$ is bounded
on $\Omega$.

\section{Extended real-valued operator convex functions}\label{Sec-3}

A notion of operator convex functions with extended real values in $(-\infty,\infty]$ will be
essential in our later discussions. In this section we present a brief exposition of such extended
operator convex functions, since the subject has been nowhere discussed so far.

Let $J$ be an arbitrary interval in $\bR$, either finite or infinite and either closed or open.
In this section we consider a Borel function $f:J\to(-\infty,\infty]$, and assume throughout that
$f$ is locally bounded from below, i.e., bounded from below on any compact subset of $J$.
We write $B(\cH)_J$ for the set of $A\in B(\cH)_\sa$ whose spectrum $\sigma(A)$ is included
in $J$. It is clear that $B(\cH)_J$ is a convex subset of $B(\cH)_\sa$.

Let $A\in B(\cH)_J$ and $A=\int_{\sigma(A)}t\,dE_A(t)$ be the spectral decomposition of $A$
with the spectral measure $E_A$ of $A$ supported on $\sigma(A)$. By Lemma \ref{L-2.5} we
can define $f(A)\in\widehat{B(\cH)}_\lb$ by $f(A):=\int_{\sigma(A)}f(t)\,dE_A(t)$, i.e.,
\begin{align}\label{F-3.1}
f(A)(\rho):=\int_{\sigma(A)}f(t)\,d\rho(E_A(t))
=\int_Jf(t)\,d\rho(E_A(t))\in(-\infty,\infty],\quad\rho\in B(\cH)_*^+.
\end{align}
When $f$ is a continuous $\bR$-valued function on $J$, it is clear that $f(A)\in B(\cH)_\sa$ is
the usual continuous functional calculus of $A$.

Recall that the usual topology on $\bR$ is extended to $(-\infty,\infty]$ as is generated by the
intervals $(a,b)$, $(a,\infty]$. The continuity of a function $\psi$ from a metric space
$\mathcal{X}$ to $(-\infty,\infty]$ is considered against this topology. Namely, if $x_n\to x$ in
$\mathcal{X}$, then $\psi(x_n)\to\psi(x)$ holds even when $\psi(x)=\infty$.  

\begin{lemma}\label{L-3.1}
Assume that $f$ is continuous on $J$ as a function to $(-\infty,\infty]$. Then for every
$\rho\in B(\cH)_*^+$, the mapping $A\in B(\cH)_J\mapsto f(A)(\rho)\in(-\infty,\infty]$ is lower
semicontinuous in the operator norm. Furthermore, the same holds even in the strong
operator topology whenever $\inf_{t\in J}f(t)/(1+|t|)>-\infty$.
\end{lemma}

\begin{proof}
For each $n\in\bN$ set $f_n:=f\wedge n$, which is a continuous $\bR$-valued function on $J$.
For every $A\in B(\cH)_J$ and $\rho\in B(\cH)_*^+$, it follows from the monotone convergence
theorem that
\[
f(A)(\rho)=\sup_n\int_Jf_n(t)\,d\rho(E_A(t))=\sup_n\rho(f_n(A)),
\]
where $f_n(A)$ is the usual continuous functional calculus of $A$. Since $A\mapsto f_n(A)$ is
continuous on $B(\cH)_J$ in the operator norm, the first assertion follows.
Next assume that $\inf_{t\in J}f(t)/(1+|t|)>-\infty$ and hence $\sup_{t\in J}|f_n(t)|/(1+|t|)<\infty$.
It follows (see e.g., \cite[Appendix A.2]{St}) that $A\mapsto f_n(A)$ is continuous on $B(\cH)_J$
in the strong operator topology for each $n\in\bN$. Hence the latter assertion holds as well.
\end{proof}

\begin{definition}\label{D-3.2}\rm
\begin{itemize}
\item[(1)] We say that $f$ is \emph{operator convex} if
\begin{align}\label{F-3.2}
f((1-\lambda)A+\lambda B)\le(1-\lambda)f(A)+\lambda f(B)\quad
\mbox{in $\widehat{B(\cH)}_\lb$}
\end{align}
holds for every $A,B\in B(\cH)_J$ with an arbitrary Hilbert space $\cH$ and for any
$\lambda\in(0,1)$.
\item[(2)] We say that $f$ is \emph{operator monotone} (resp., \emph{operator monotone
decreasing}) if $A\le B$ implies $f(A)\le f(B)$ (resp., $f(A)\ge f(B)$) in $\widehat{B(\cH)}_\lb$
for every $A,B\in B(\cH)_J$ with any $\cH$.
\end{itemize}
If $f$ is an $\bR$-valued function on $J$, then the above definitions are obviously the same as
the usual operator convexity and the operator monotonicity of $f$; see \cite{Bh,Hi}.
\end{definition}

Some basic equivalent conditions for $f$ being operator convex are given in the next
proposition, extending the $\bR$-valued case in \cite[Theorem 2.5.7]{Hi}. Condition (ii) will
particularly be useful in later discussions.

\begin{proposition}\label{P-3.3}
Let $f$ be as stated above. Then the following conditions are equivalent, where Hilbert spaces
$\cH,\cH_i,\cK$ are arbitrary and not fixed:
\begin{itemize}
\item[\rm(i)] $f$ is operator convex;
\item[\rm(ii)] for every $A\in B(\cH)_J$ and every isometry $V:\cK\to\cH$,
\[
f(V^*AV)\le V^*f(A)V\quad\mbox{in $\widehat{B(\cK)}_\lb$};
\]
\item[\rm(iii)] for every $A_i\in B(\cH_i)_J$ and every bounded operator $V_i:\cK\to\cH_i$
for $1\le i\le m$ with any $m\in\bN$ such that $\sum_{i=1}^mV_i^*V_i=I_\cK$,
\[
f\Biggl(\sum_{i=1}^mV_i^*A_iV_i\Biggr)\le\sum_{i=1}^mV_i^*f(A_i)V_i
\quad\mbox{in $\widehat{B(\cK)}_\lb$};
\]
\item[\rm(iv)] for every $A,B\in B(\cH)_J$ and every projection $P\in B(\cH)$,
\[
f(PAP+(I-P)B(I-P))\le Pf(A)P+(I-P)f(B)(I-P)\quad\mbox{in $\widehat{B(\cH)}_\lb$}.
\]
\end{itemize}
\end{proposition}

Here note that $V^*AV$ in (ii) and $\sum_{i=1}^nV_i^*A_iV_i$ in (iii) are in $B(\cK)_J$
automatically. The proof of the proposition is left to the reader. Indeed, the proof is essentially
the same as that of \cite[Theorem 2.5.7]{Hi} by taking account of the following basic facts which
are immediate from definition \eqref{F-3.1}:
\begin{itemize}
\item[(a)] $f(U^*AU)=U^*f(A)U$ in $\widehat{B(\cH)}_\lb$ for every $A\in B(\cH)_J$ and any
unitary $U$ on $\cH$,
\item[(b)] for $T\in\widehat{B(\cK)}_\lb$ and $S\in\widehat{B(\cH)}_\lb$,
$T\oplus S\in\widehat{B(\cK\oplus\cH)}_\lb$ is defined by
$(T\oplus S)(\rho):=T(\rho_1)+S(\rho_2)$ for every
$\rho=\begin{bmatrix}\rho_1&\rho_{12}\\\rho_{12}^*&\rho_2\end{bmatrix}\in B(\cK\oplus\cH)_*^+$,
\item[(c)] $f(A\oplus B)=f(A)\oplus f(B)$ in $\widehat{B(\cK\oplus\cH)}_\lb$ for every
$A\in B(\cK)_J$ and $B\in B(\cH)_J$.
\end{itemize}

\begin{remark}\label{R-3.4}\rm
In Definition \ref{D-3.2} we can fix an infinite-dimensional separable Hilbert space $\cH$.
Indeed, let $\cH_0$ be a such Hilbert space. For every $A,B\in B(\cH)_\sa$ with any Hilbert
space $\cH$, one can decompose $A,B$ into direct sums $A=\bigoplus_iA_i$ and
$B=\bigoplus_iB_i$ under a direct sum decomposition $\cH=\bigoplus_i\cH_i$ into separable
Hilbert spaces $\cH_i$. (This is because the unital $C^*$-algebra generated by $A,B$ is
separable and any non-degenerate representation of a $C^*$-algebra is the direct sum of cyclic
representations.) Since each $\cH_i$ is isomorphic to a subspace of $\cH_0$, it is not
difficult to see (from properties like (a)--(c) above) that inequality \eqref{F-3.2} follows from
that when $\cH=\cH_0$, while the details are omitted here.
\end{remark}

\begin{remark}\label{R-3.5}\rm
It is clear, from the definition in Lemma \ref{L-2.5}, that $f(tI)=f(t)I$ holds for any $t\in J$, where
$(\infty\cdot I)(\rho)=\infty\cdot\rho(I)$ (for $f(t)=\infty$). This shows that if $f$ is operator convex
(resp., operator monotone), then it is convex (resp., monotone increasing) on $J$ as a numerical
function with values in $(-\infty,\infty]$.
\end{remark}

\begin{example}\label{E-3.6}\rm
Here we pick out two exceptional examples of operator convex functions $f:J\to(-\infty,\infty]$.
A trivial example is $f\equiv\infty$. Another particular one is the case when $f(t_0)<\infty$ for
some $t_0\in J$ and $f(t)=\infty$ for all $t\in J\setminus\{t_0\}$. This case is confirmed as follows.
Let $A,B\in B(\cH)_J$, $0<\lambda<1$ and $\xi\in\cH$. If
$((1-\lambda)f(A)+\lambda f(B))(\omega_\xi)<\infty$, then $f(A)(\omega_\xi)<\infty$ and
$f(B)(\omega_\xi)<\infty$. Since
\[
f(A)(\omega_\xi)=f(t_0)\<E_A(\{t_0\})\xi,\xi\>+\infty\<E_A(J\setminus\{t_0\})\xi,\xi\><\infty,
\]
one has $E_A(J\setminus\{t_0\})\xi=0$ so that $A\xi=t_0\xi$, and similarly $B\xi=t_0\xi$. Hence
it follows that $((1-\lambda)A+\lambda B)\xi=t_0\xi$ so that
\[
f((1-\lambda)A+\lambda B)(\omega_\xi)=f(t_0)\<\xi,\xi\>
=(1-\lambda)f(A)(\omega_\xi)+\lambda f(B)(\omega_\xi),
\]
showing that $f$ is operator convex.
\end{example}

The next theorem says that $f:J\to(-\infty,\infty]$ is operator convex, except for the particular
cases in Example \ref{E-3.6}, only when $f$ is $\bR$-valued and operator convex on the interior
$J^\circ$ (so that $f$ can take value $\infty$ only at the boundary of $J$). The situation is similar
for operator monotone functions. In the following, let $a:=\inf J$, $b:=\sup J$, and denote
$f(a^+):=\lim_{t\searrow a}f(t)$, $f(b^-):=\lim_{t\nearrow b}f(t)$ (if the limits exist in $(-\infty,\infty]$).

\begin{theorem}\label{T-3.7}
Assume that $f(t)<\infty$ at more than one point in $J$. 
\begin{itemize}
\item[\rm(1)] The following conditions are equivalent:
\begin{itemize}
\item[\rm(i)] $f$ is operator convex on $J$;
\item[\rm(ii)] $f$ is $\bR$-valued and operator convex on $J^\circ$, $f(a)\ge f(a^+)$ if $a\in J$,
and $f(b)\ge f(b^-)$ if $b\in J$.
\end{itemize}
\item[\rm(2)] The following conditions are equivalent:
\begin{itemize}
\item[\rm(i$'$)] $f$ is operator monotone on $J$;
\item[\rm(ii$'$)] $f$ is $\bR$-valued and operator monotone on $J^\circ$, $f(a)\le f(a^+)$ if
$a\in J$, and $f(b)\ge f(b^-)$ if $b\in J$.
\end{itemize}
\end{itemize}
\end{theorem}

To prove the theorem, we first give a lemma.

\begin{lemma}\label{L-3.8}
Let $-\infty<\alpha<\gamma<\beta<\infty$.
\begin{itemize}
\item[\rm(1)] If a function $f:(\alpha,\beta)\to(-\infty,\infty]$
satisfies either
\begin{align}\label{F-3.3}
f(x)<\infty\ \ (\alpha<x<\gamma),\quad f(x)=\infty\ \ (\gamma<x<\beta),
\end{align}
or
\begin{align}\label{F-3.4}
f(x)=\infty\ \ (\alpha<x<\gamma),\quad f(x)<\infty\ \ (\gamma<x<\beta),
\end{align}
then $f$ is not operator convex on $(\alpha,\beta)$.
\item[\rm(2)] If $f$ satisfies \eqref{F-3.3} (resp., \eqref{F-3.4}), then $f$ is not operator
monotone (resp., not operator monotone decreasing) on $(\alpha,\beta)$.
\end{itemize}
\end{lemma}

\begin{proof}
By transforming by a linear function, it suffices to show the first case with
$\alpha=0<\gamma=2<\beta$ in each of (1) and (2). For $0<\delta<\min\{1,\beta-2\}$
consider $A,B\in B(\bC^2)_{(\alpha,\beta)}$ and $\xi\in\bC^2$ defined to be
\[
A:=\begin{bmatrix}2-\delta^2&0\\0&2+\delta\end{bmatrix},\quad
B:=\begin{bmatrix}1&1-\delta^2\\1-\delta^2&1\end{bmatrix},\quad
\xi:=\begin{bmatrix}1\\0\end{bmatrix}.
\]
Then one has $f(A)(\omega_\xi)=\<f(A)\xi,\xi\>=f(2-\delta^2)$. Since the eigenvalues of
$B$ are $\delta^2$ and $2-\delta^2$, one has
$f(B)(\omega_\xi)=f(\delta^2)|\<v_1,\xi\>|^2+f(2-\delta^2)|\<v_2,\xi\>|^2$, where $v_1,v_2$
are the unit eigenvectors of $B$ for $\delta^2,2-\delta^2$ respectively. Therefore,
\[
\half(f(A)+f(B))(\omega_\xi)=\half(f(A)(\omega_\xi)+f(B)(\omega_\xi))<\infty.
\]
On the other hand, a direct calculation shows that two eigenvalues of $\half(A+B)$ are
$\lambda_1(\delta):=1+\delta/4+o(\delta)$ and $\lambda_2(\delta) := 2+\delta/4+o(\delta)$
as $\delta \searrow 0$. With the unit eigenvectors $u_1(\delta),u_2(\delta)$ corresponding to
$\lambda_1(\delta),\lambda_2(\delta)$ respectively, we have
\[
f\bigl(\half(A+B)\bigr)(\omega_\xi)
=f(\lambda_1(\delta))|\<u_1(\delta),\xi\>|^2
+f(\lambda_2(\delta))|\<u_2(\delta),\xi\>|^2=\infty,
\]
for all sufficiently small $\delta>0$, since $\<u_2(\delta),\xi\>\ne0$ obviously. Therefore,
$f\bigl(\half(A+B)\bigr)\le\half(f(A)+f(B))$ does not hold, and (1) has been shown.

Next we show (2). Since $A\ge B$ as immediately verified, one has $(A+B)/2\le A$ but
$f\bigl(\half(A+B)\bigr)\le f(A)$ does not hold. Hence $f$ is not operator monotone.
\end{proof}

\begin{proof}[Proof of Theorem \ref{T-3.7}]
(1)\enspace
(i)$\implies$(ii).\enspace
Assume item (i); then by Remark \ref{R-3.5}, $f$ is numerically convex on $J$. By assumption
on $f$ (having finite values at more than one point) there are $a_0,b_0\in[a,b]$ with $a_0<b_0$
such that $f(t)=\infty$ for all $t\in J\setminus[a_0,b_0]$, $f|_{(a_0,b_0)}$ is an $\bR$-valued
convex function, $f(a_0)\ge f(a_0^+)$ if $a_0\in J$, and $f(b_0)\ge f(b_0^-)$ if $b_0\in J$,
where the limits $f(a_0^+)$ and $f(b_0^-)$ exist in $(-\infty,\infty]$ thanks to the numerical
convexity of $f|_{(a_0,b_0)}$. If $a<a_0$ (resp., $b_0<b$), then we can apply Lemma
\ref{L-3.8}(1) with $a\le\alpha<\gamma=a_0<\beta\le b_0$ (resp.,
$a_0\le\alpha<\gamma=b_0<\beta\le b$) to find a contradiction to (i). Hence $a_0=a$ and
$b_0=b$. Moreover, by applying property (i) to $A,B\in B(\cH)_{(a,b)}$ we see that $f$ is
$\bR$-valued and operator convex on $J^\circ=(a,b)$.

(ii)$\implies$(i).\enspace
When $J=(a,b)$, this is obvious. In the following, we will prove (ii)$\implies$(i) when $J=[a,b]$
(so $-\infty<a<b<\infty$). The proof is similar when $J=[a,b)$ or $J=(a,b]$. Now assume item (ii).
First, assume further that $f$ is continuous at $a,b$ (hence on the whole $[a,b]$) as a function
to $(-\infty,\infty]$. We show property (ii) of Proposition \ref{P-3.3}. Let $A\in B(\cH)_{[a,b]}$,
$V:\cK\to\cH$ be an isometry, and $\rho\in B(\cK)_*^+$. Choose a sequence
$\delta_n\in(0,(b-a)/2)$ with $\delta_n\searrow0$, and define
$r_n(t):=(t\vee(a+\delta_n))\wedge(b-\delta_n)$ for $t\in[a,b]$. Since $r_n(A)\in B(\cH)_{(a,b)}$,
one has $f(V^*r_n(A)V)\le V^*f(r_n(A))V$ in $\widehat{B(\cK)}_\lb$ (thanks to Proposition
\ref{P-3.3} for $f|_{(a,b)}$). Since
$\|V^*r_n(A)V-V^*AV\|\to0$, by Lemma \ref{L-3.1} one has
\begin{align}\label{F-3.5}
f(V^*AV)(\rho)\le\liminf_{n\to\infty}f(V^*r_n(A)V)(\rho)
\le\liminf_{n\to\infty}(V^*f(r_n(A))V)(\rho).
\end{align}
Furthermore, note that
\begin{equation}\label{F-3.6}
\begin{aligned}
(V^*f(r_n(A))V)(\rho)&=\int_{[a,b]}f(t)\,d\rho(V^*E_{r_n(A)}(t)V) \\
&=\int_{[a,b]}f(r_n(t))\,d\rho(V^*E_A(t)V).
\end{aligned}
\end{equation}
When $f(a)<\infty$ and $f(b)<\infty$ (hence $f(t)<\infty$ for all $t\in[a,b]$), it is clear that
$f(r_n(t))\to f(t)$ uniformly on $[a,b]$. Taking the (numerical) convexity of $f$ into consideration,
we observe the following: When $f(a)<\infty$ and $f(b)=\infty$, there are an $n_0$
and some $c\in(a,b)$ such that $f(r_n(t))\to f(t)$ uniformly on $[a,c]$ and $f(r_n(t))\nearrow f(t)$
for all $t\in(c,b]$ as $n_0\le n\to\infty$. When $f(a)=\infty$ and $f(b)<\infty$, the situation is
similar. When $f(a)=f(b)=\infty$, there is an $n_0$ such that $f(r_n(t))\nearrow f(t)$ for all
$t\in[a,b]$ as $n_0\le n\to\infty$. Hence, using the bounded and the monotone convergence
theorems (after dividing the integration over $[a,b]$ into those over $[a,c]$ and $(c,b]$ if
necessary), we have
\begin{align}\label{F-3.7}
\lim_{n\to\infty}\int_{[a,b]}f(r_n(t))\,d\rho(V^*E_A(t)V)=\int_{[a,b]}f(t)\,d\rho(V^*E_A(t)V)
=(V^*f(A)V)(\rho).
\end{align}
Combining \eqref{F-3.5}--\eqref{F-3.7} gives $f(V^*AV)(\rho)\le(V^*f(A)V)(\rho)$, showing that
$f$ is operator convex when $f$ is continuous on the whole $[a,b]$.

To show (i) without the continuity assumption at $a,b$, set $f_0(t):=f(t)$ for $t\in(a,b)$,
$f_0(a):=f(a^+)$ and $f_0(b):=f(b^-)$. Then $f_0$ is operator convex as shown above. Furthermore,
set $\chi_a(a):=1$, $\chi_a(t):=0$ for $t\in(a,b]$, and similarly for $\chi_b$. One can choose
increasing $\alpha_n,\beta_n\ge0$ such that $f_0+\alpha_n\chi_a+\beta_n\chi_b\nearrow f$
and hence, by the monotone convergence theorem,
\begin{align*}
f(A)(\rho)&=\lim_n(f_0+\alpha_n\chi_a+\beta_n\chi_b)(A)(\rho) \\
&=\lim_n(f_0(A)(\rho)+\alpha_n\chi_a(A)(\rho)+\beta_n\chi_b(A)(\rho))
\end{align*}
for all $A\in B(\cH)_{[a,b]}$ and $\rho\in B(\cH)_*^+$ (and any $\cH$). Therefore, it remains to
prove that $\chi_a$ and $\chi_b$ are operator convex. For this, let $A\in B(\cH)_{[a,b]}$ and
$V:\cK\to\cH$ be as above. Note that $\chi_a(A)$ and $\chi_a(V^*AV)$ are the projections
onto $\ker(A-aI_\cH)$ and $\ker(V^*AV-aI_\cK)$, respectively. Assume that $\chi_a(V^*AV)\xi=\xi$,
i.e., $\xi\in\ker(V^*AV-aI_\cK)=\ker V^*(A-aI_\cH)V$. Hence we have $V\xi\in\ker(A-aI_\cH)$ so that
$\<V^*\chi_a(A)V\xi,\xi\>=\<V\xi,V\xi\>=\<\xi,\xi\>$, which implies that $V^*\chi_a(A)V\xi=\xi$.
Therefore, $\chi_a(V^*AV)\le V^*\chi_a(A)V$, as desired, and similarly for $\chi_b$.

(2)\enspace
(i$'$)$\implies$(ii$'$).\enspace
The proof is similar to that of (i)$\implies$(ii), so we omit the details.

(ii$'$)$\implies$(i$'$).\enspace
As in the proof of (ii)$\implies$(i), we prove the case $J=[a,b]$, and first assume that $f$ is
continuous at $a,b$. Since $f(a)$ must be finite, it is clear that $f$ is operator monotone on $[a,b)$.
For the remaining, by transforming $[a,b]$ to $[0,1]$ by a linear function, we may assume that
$[a,b]=[0,1]$. Let $A,B\in B(\cH)_{[0,1]}$ with $A\le B$. For any $r\in(0,1)$, since
$rA,rB\in B(\cH)_{[0,1)}$ and $rA\le rB$, one has $f(rA)\le f(rB)$. For every $\rho\in B(\cH)_*^+$
note that
\begin{align*}
f(rA)(\rho)=\int_{[0,1]}f(t)\,d\rho(E_{rA}(t))&=\int_{[0,1]}f(rt)\,d\rho(E_A(t)) \\
&\nearrow\int_{[0,1]}f(t)\,d\rho(E_A(t))=f(A)(\rho)
\end{align*}
as $r\nearrow1$ by the monotone convergence theorem. The same holds for $f(rB)(\rho)$,
so that $f(A)(\rho)\le f(B)(\rho)$.

As in the proof of (ii)$\implies$(i), it remains to show that $\chi_0$ is operator monotone
decreasing and $\chi_1$ is operator monotone on $B(\cH)_{[0,1]}$. Let $A,B\in B(\cH)_{[0,1]}$
with $A\le B$. Then $\ker A\supseteq\ker B$ and hence $\chi_0(A)\ge\chi_0(B)$. Since $I-A\ge I-B$,
one has $\ker(I-A)\subseteq\ker(I-B)$, implying $\chi_1(A)\le\chi_1(B)$.
\end{proof}

\begin{remark}\label{R-3.9}\rm
The operator convexity of $\chi_a$ (in the last part of the above proof of (i)$\implies$(ii)) is also
easy by showing \eqref{F-3.2} directly. Indeed, let $A,B\in B(\cH)_{[a,b]}$ and $0<\lambda<1$.
Since $A-aI\ge0$ and $B-aI\ge0$, one has
\begin{align*}
\ker((1-\lambda)A+\lambda B-aI)&=\ker((1-\lambda)(A-aI)+\lambda(B-aI)) \\
&=\ker(A-aI)\cap\ker(B-aI).
\end{align*}
Hence $\chi_a((1-\lambda)A+\lambda B)=\chi_a(A)\wedge\chi_a(B) \le(1-\lambda)\chi_a(A)+
\lambda\chi_a(B)$.
\end{remark}

\section{Pusz--Woronowicz functional calculus}\label{Sec-4}

Throughout this section, we fix an arbitrary Borel function $\phi : [0,\infty)^2 \to (-\infty,\infty]$
which is \emph{homogeneous} and \emph{locally bounded from below} (i.e., bounded from below
on any compact subset of $[0,\infty)^2$). Here, $\phi$ is homogeneous if
$\phi(\lambda x,\lambda y) = \lambda \phi(x,y)$ holds for every $\lambda, x, y \geq 0$. Hence
$\phi(0,0)=0$ necessarily holds. Also we remark that $\phi$ is locally bounded from below if and
only if $\phi(t,1-t)$ on $[0,1]$ is bounded from below. Typical and important examples of such
functions are $\psi(x,y):= x\log(x/y)$ and $\phi_\alpha(x,y):= y(x/y)^\alpha=x^\alpha y^{1-\alpha}$
(where $\alpha\ge0$), with conventions $\psi(0,0)=\phi_\alpha(0,0)=0$ and
$\psi(x,0) = \phi_\alpha(x,0)=\infty$ for $x>0$ and $\alpha>1$. Both of them play an important role,
for instance, in quantum information theory.

Although the original formalism of Pusz and Woronowicz in \cite{PW1,PW2} based on positive
sesquilinear forms is available even for unbounded functions as above, we will do reformulate
their functional calculus in terms of unbounded objects extending self-adjoint operators discussed
in \S\ref{Sec-2}. The content of this section is somewhat expository and also may be regarded as
an upgrade of the discussions by Hatano and the second-named author in \cite[\S4]{HU} to Borel
functions locally bounded from below. The approach here is axiomatic unlike \cite{HU}, though the
theory is about a functional calculus still depending on the distinguished function $\phi$
unlike Kubo and Ando's axiomatization of operator connections \cite{KA}. (Axiomatization of
Kubo and Ando's type will be discussed in \S\ref{Sec-10}.)

Associated with the function $\phi$ above, we introduce the Pusz--Woronowicz functional calculus
$\phi(A,B)$ of pairs $(A,B)$ in $B(\cH)_+$, whose values are elements of the extended lower
semibounded self-adjoint part $\widehat{B(\cH)}_\lb$ introduced in \S\ref{Sec-2}. The
definition is given in an axiomatic fashion with two postulates as follows:

\begin{definition}\label{D-4.1}\rm
The \emph{Pusz--Woronowicz functional calculus} (or \emph{PW-functional calculus}
for short) associated with $\phi$ is an operation giving a two-variable mapping   
\[
(A,B) \in B(\cH)_+\times B(\cH)_+ \,\longmapsto\, \phi(A,B) \in \widehat{B(\cH)}_\lb 
\]
for each Hilbert space $\cH$ such that the following properties hold{\rm:}
\begin{itemize} 
\item[(1)] (\emph{Extending the usual functional calculus})\enspace
Whenever $(A,B)$ is a commuting pair, $\phi(A,B)$ is given by the usual functional calculus, that is, 
\[
\phi(A,B)(\rho) = \int_{\sigma(A)\times\sigma(B)} \phi(x,y)\,d\rho(E_{(A,B)}(x,y)),
\qquad \rho \in B(\cH)_*^+, 
\]
where $E_{(A,B)}$ denotes the joint spectral measure of the pair $(A,B)$.
\item[(2)] (\emph{Operator homogeneity})\enspace
For any bounded operator $C$ from another Hilbert space $\mathcal{K}$ to $\cH$ with
$\overline\ran(A+B) \subseteq \overline\ran\,C$, the closures of the ranges of $A+B$, $C$
respectively, we have 
\[
\phi(C^* A C, C^* B C) = C^*\phi(A,B) C,  
\]
where the right-hand side is in the sense of Lemma \ref{L-2.4}.
\end{itemize}
\end{definition}

Both items (1) and (2) are quite natural as a functional calculus associated with $\phi$. The former
makes sense thanks to Lemma \ref{L-2.5}. The latter reflects the homogeneity of $\phi$ and also
plays a role of intertwining property built in the recent notion of non-commutative functions;
see, e.g., \cite{KV} and \cite[Part Two]{AMY}.

As item (2) indicates, we will discuss operators on different Hilbert spaces at the same time, and
hence we will sometimes denote by $I_\cH$ the identity operator on a Hilbert space $\cH$ to avoid
any confusion.

The next theorem justifies the above definition.

\begin{theorem}\label{T-4.2}
The PW-functional calculus associated with $\phi$ exists and is uniquely determined. 
\end{theorem}

We will prove this based on the description of PW-functional calculus discussed in \cite{HU}.
In what follows, we will use the following notations: For a pair $(A,B)$ in $B(\cH)_+$, we write
\begin{align}\label{F-4.1}
\cH_{A,B} := \overline\ran(A+B) = \ker(A+B)^\perp,
\end{align}
and define a bounded operator $T_{A,B}:\cH\to\cH_{A,B}$ by
\begin{align}\label{F-4.2}
T_{A,B}\xi:=(A+B)^{1/2}\xi \in \cH_{A,B},\qquad \xi\in\cH.
\end{align}
By construction, $T_{A,B}$ has a dense range. Remark that $T_{A,B}$
was given in \cite[\S\S3.1]{HU} by the block matrix
\[
T_{A,B} = \begin{bmatrix} 0 & ([A]_{A,B}+[B]_{A,B})^{1/2}\end{bmatrix}
\quad \text{along} \quad \cH = \ker(A+B)\oplus\cH_{A,B},
\] 
where $[A]_{A,B}, [B]_{A,B}$ are the restrictions to $\cH_{A,B}$ of $A, B$ respectively.

The next lemma is a collection of technical facts, which will be used successively.
Items (1) and (2) are just extracted from \cite{HU}. 

\begin{lemma}\label{L-4.3}
With $\cH_{A,B}$ and $T_{A,B}$ defined above, we have the following facts:
\begin{itemize} 
\item[\rm(1)] For any pair $(A,B)\in B(\cH)_+\times B(\cH)_+$, there is a unique pair
$(R_{A,B}, S_{A,B})$ of positive contractions on $\cH_{A,B}$ such that 
\begin{gather*}
R_{A,B}+S_{A,B} = I_{\cH_{A,B}}, \\ 
(A,B) = (T_{A,B}^* R_{A,B}T_{A,B}, T_{A,B}^* S_{A,B}T_{A,B}).
\end{gather*} 
\item[\rm(2)] For any pair $(A,B)\in B(\cH)_+\times B(\cH)_+$ and any bounded operator
$C : \mathcal{K} \to \cH$ with another Hilbert space $\cK$, there is an isometry
$\widehat{C} : \cK_{C^*AC,C^*BC} \to \cH_{A,B}$ such that 
\begin{gather*}
\widehat{C}\,T_{C^*AC,C^*BC} = T_{A,B}\,C, \\
\big(R_{C^*AC,C^*BC},S_{C^*AC,C^*BC}\big)
= \big(\widehat{C}^* R_{A,B}\widehat{C},\widehat{C}^* S_{A,B}\widehat{C}\big).
\end{gather*} 
Moreover, if $\overline\ran(A+B)\subseteq\overline\ran\,C$, then $\widehat{C}$ must be a unitary
transform. 
\item[\rm(3)] Let $R,S$ be a pair of positive contractions on $\cH$ such that $R+S=I_\cH$. Then
we have
\[
\phi(R,S) = \phi(R,1-R) = \phi(1-S,S),
\]
where $\phi(R,1-R), \phi(1-S,S) \in \widehat{B(\cH)}_\lb$ are defined by 
\begin{align}
\phi(R,1-R)(\rho) &:= \int_{[0,1]}\phi(r,1-r)\,d\rho(E_R(r)), \label{F-4.3}\\
\phi(1-S,S)(\rho) &:= \int_{[0,1]}\phi(1-s,s)\,d\rho(E_S(s)) \label{F-4.4}
\end{align}
for every $\rho \in B(\cH)_*^+$ with the spectral measures $E_R, E_S$ of $R, S$ respectively.
\end{itemize} 
\end{lemma}
\begin{proof} 
Item (1) is in \cite[\S\S3.1]{HU}. The first part of (2) is found in the proof of \cite[Theorem 9]{HU},
and the latter part is obvious; see \cite[Remark 10]{HU}. 

\medskip
(3)\enspace
Note that $f(R,1-R), f(1-S,S)$ are well defined by Lemma \ref{L-2.5}. Let $E$ be the joint
spectral measure of $(R,S)$. Observe that for any Borel sets  $\Phi,\Psi\subseteq[0,1]$,
\begin{align*}
E(\Phi\times\Psi)
&= \chi_\Phi(R)\chi_\Psi(S) = \chi_\Phi(R)\chi_\Psi(I_\cH-R) \\
&= \int_{[0,1]} \chi_\Phi(r)\chi_\Psi(1-r)\,dE_R(r)
= \int_{[0,1]} \chi_{\Phi\times\Psi}(r,1-r)\,dE_R(r).
\end{align*}
For each $\rho \in B(\cH)_*^+$, the monotone class theorem yields
\[
\rho(E(\Lambda)) = \int_{[0,1]} \chi_\Lambda(r,1-r)\,d\rho(E_R(r))
\]
for any Borel subset $\Lambda \subseteq [0,1]^2$. This immediately implies, by Definition
\ref{D-4.1}(1) and \eqref{F-4.3}, that 
\begin{align*}
\phi(R,S)(\rho) &= \int_{[0,1]^2} \phi(r,s)\,d\rho(E(r,s)) \\
&= \int_{[0,1]} \phi(r,1-r)\,d\rho(E_R(r)) = \phi(R,1-R)(\rho). 
\end{align*}
The other expression is shown in exactly the same manner. 
\end{proof}  

Now we are ready to prove Theorem \ref{T-4.2}. 

\begin{proof}[Proof of Theorem \ref{T-4.2}]
Let $(A, B)\in B(\cH)_+\times B(\cH)_+$ be an arbitrary pair. Since
$R_{A,B}+S_{A,B}=I_{\cH_{A,B}}$ by Lemma \ref{L-4.3}(1), we have the joint spectral measure
$E$ of $(R_{A,B},S_{A,B})$. Then we consider  $T_{A,B}^* \phi(R_{A,B},S_{A,B}) T_{A,B}$ in the
sense of Lemma \ref{L-2.4} with
\[
\phi(R_{A,B},S_{A,B})(\rho) := \int_{[0,1]^2} \phi(r,s)\,d\rho(E(r,s)),
\qquad \rho \in B(\cH_{A,B})_*^+. 
\]
This procedure is the same as \cite[Eq.~(1)]{HU}. Whenever the desired $\phi(A,B)$ exists, 
we must have
\[
\phi(A,B)=\phi(T_{A,B}^* R_{A,B}T_{A,B},T_{A,B}^* S_{A,B}T_{A,B})=
T_{A,B}^*\phi(R_{A,B},S_{A,B}) T_{A,B}
\]
by Lemma \ref{L-4.3}(1) again and due to Definition \ref{D-4.1}(2). Since $\phi(R_{A,B},S_{A,B})$
is uniquely determined due to Definition \ref{D-4.1}(1), this implies the uniqueness of $\phi(A,B)$
(if it exists). Therefore, it suffices to show that
$(A,B)\mapsto T_{A,B}^*\phi(R_{A,B},S_{A,B})T_{A,B}$ indeed enjoys postulates (1) and (2) of
Definition \ref{D-4.1}.

\medskip
We first confirm item (1). Assume that $(A,B)$ is a commuting pair, and let $A_1, B_1$ be the
restrictions to $\cH_{A,B}$ of $A,B$ respectively. Consider the Borel functions 
\[
r(x,y) := \begin{cases} x/(x+y) & (x+y > 0), \\ 0 & (x+y=0), \end{cases} \qquad
s(x,y) := \begin{cases} y/(x+y) & (x+y > 0), \\ 0 & (x+y=0). \end{cases}
\]
Then $R_{A,B} = r(A_1,B_1)$ and $S_{A,B} = s(A_1,B_1)$ were established in the proof of
\cite[Proposition 2]{HU}. Namely, $R_{A,B}$ and $S_{A,B}$ are given by the restrictions
to $\cH_{A,B}$ of $r(A,B)$ and $s(A,B)$, respectively, i.e.,
\begin{align*}
R_{A,B}&=\int_{\sigma(A)\times\sigma(B)}r(x,y)\,dE_{(A,B)}(x,y)\,\Big|_{\cH_{A,B}}, \\
S_{A,B}&=\int_{\sigma(A)\times\sigma(B)}s(x,y)\,dE_{(A,B)}(x,y)\,\Big|_{\cH_{A,B}}.
\end{align*}
Here it should be noticed that the projection from $\cH$ onto $\cH_{A,B}$ is exactly
$I_\cH-E_{(A,B)}(\{(0,0)\})$, and hence the above restrictions are well defined. For any Borel sets
$\Phi,\Psi \subseteq [0,1]$, we observe that
\[
\chi_\Phi(R_{A,B})\chi_\Psi(S_{A,B}) 
= 
\int_{\sigma(A)\times\sigma(B)} \chi_\Phi(r(x,y)) \chi_\Psi(s(x,y))
\,dE_{(A,B)}(x,y)\,\Big|_{\cH_{A,B}}
\]
and hence 
\begin{align*}
&\< T_{A,B}^* E(\Phi\times\Psi)T_{A,B}\xi,\xi\> \\
&\quad= 
\< T_{A,B}^* \chi_\Phi(R_{A,B})\chi_\Psi(S_{A,B}) T_{A,B}\xi,\xi\> \\
&\quad= 
\< \chi_\Phi(R_{A,B})\chi_\Psi(S_{A,B})(A+B)^{1/2}\xi,(A+B)^{1/2}\xi\> \\
&\quad=
\int_{\sigma(A)\times\sigma(B)} \chi_\Phi(r(x,y)) \chi_\Psi(s(x,y))
\,d\< E_{(A,B)}(x,y)(A+B)^{1/2}\xi,(A+B)^{1/2}\xi\> \\
&\quad=
\int_{\sigma(A)\times\sigma(B)} \chi_\Phi(r(x,y)) \chi_\Psi(s(x,y))(x+y)\,d\< E_{(A,B)}(x,y)\xi,\xi\>
\end{align*}
for any $\xi \in \cH$. Thus, for each $\rho \in B(\cH)_*^+$, the monotone class theorem ensures that  
\[
\rho(T_{A,B}^* E(\Lambda)T_{A,B})
= \int_{\sigma(A)\times\sigma(B)}\chi_\Lambda(r(x,y),s(x,y))(x+y)\,d\rho(E_{(A,B)}(x,y))
\]
for any Borel set $\Lambda\subseteq [0,1]^2$. Therefore, for every $\rho\in B(\cH)_*^+$, we have
\begin{align*}
(T_{A,B}^*\phi(R_{A,B},S_{A,B})T_{A,B})(\rho) 
&= \phi(R_{A,B},S_{A,B})(T_{A,B} \rho\, T_{A,B}^* ) \\
&= \int_{[0,1]^2} \phi(r,s)\,d\rho(T_{A,B}^* E(r,s) T_{A,B}) \\
&= \int_{\sigma(A)\times\sigma(B)} \phi(r(x,y),s(x,y))(x+y)\,d\rho(E_{(A,B)}(x,y)) \\
&= \int_{\sigma(A)\times\sigma(B)} \phi(x,y)\,d\rho(E_{(A,B)}(x,y))
\end{align*}
thanks to the homogeneity of $\phi$. This is the identity in item (1).

Next we confirm item (2), i.e., operator homogeneity. Let $C$ be a bounded operator from
another Hilbert space $\mathcal{K}$ to $\cH$ as in item (2) of Definition \ref{D-4.1}. In what follows,
we will use Lemma \ref{L-4.3}(2) freely and write $A' = C^*AC$, $B'=C^*BC$ for short. Let $E$ is
the joint spectral measure of $(R_{A,B},S_{A,B})$; then that of $(R_{A',B'},S_{A',B'})$ is given by
$\widehat{C}^* E(\,\cdot\,)\widehat{C}$ thanks to the simultaneous unitary equivalence of
$(R_{A',B'},S_{A',B'})$ to $(R_{A,B},S_{A,B})$ with $\widehat{C}$. Hence we have
\[
\phi(R_{A',B'},S_{A',B'})(\rho')
= \int_{[0,1]^2} \phi(r,s)\,d\rho'(\widehat{C}^* E(r,s)\widehat{C})
= \phi(R_{A,B},S_{A,B})(\widehat{C}\rho' \widehat{C}^*)
\]
for every $\rho' \in B(\mathcal{K}_{A',B'})_*^+$. Therefore, it follows that 
\begin{align*}
(T_{A',B'}^*\phi(R_{A',B'},S_{A',B'}) T_{A',B'})(\rho) 
&= \phi(R_{A',B'},S_{A',B'})(T_{A',B'}\rho\,T_{A',B'}^*) \\
&= \phi(R_{A,B},S_{A,B})(\widehat{C}T_{A',B'}\rho\,T_{A',B'}^*\widehat{C}^*) \\
&= \phi(R_{A,B},S_{A,B})(T_{A,B}C\rho\,C^* T_{A,B}^*) \\
&= (C^* T_{A,B}^*\phi(R_{A,B},S_{A,B})T_{A,B}C)(\rho) 
\end{align*} 
for every $\rho \in B(\mathcal{K})_*^+$. This is the required operator homogeneity.
\end{proof}

The proof of Theorem \ref{T-4.2} (or Definition \ref{D-4.1} itself) says that an explicit
realization of the PW-functional calculus associated with $\phi$ is
\begin{align}\label{F-4.5}
\phi(A,B)=T_{A,B}^*\phi(R_{A,B},S_{A,B})T_{A,B}
\end{align}
for every $A,B\in B(\cH)_+$. We point out that it is indeed one of the explicit realizations of
$\phi(A,B)$ and there are many ways to construct $\phi(A,B)$ corresponding to `sections' in
Pusz and Woronowicz's terminology (see \cite{PW2}). Remark that formula \eqref{F-4.5} is clearly
rewritten as
\begin{equation}\label{F-4.6}
\begin{aligned}
\phi(A,B)&=(A+B)^{1/2}\phi(R,S)(A+B)^{1/2} \\
&=(A+B)^{1/2}\phi(R,1-R)(A+B)^{1/2} \\
&=(A+B)^{1/2}\phi(1-S,S)(A+B)^{1/2}
\end{aligned}
\end{equation}
with $R := R_{A,B}\oplus0$ and $S := S_{A,B}\oplus0$ on $\cH=\cH_{A,B}\oplus\cH_{A,B}^\perp$,
though equality $\phi(R,S) = \phi(R,I-R) = \phi(I-S,S)$ does not hold in general.

\begin{remark}\label{R-4.4}\rm
As explained in \cite{HU}, $(T_{A,B}:\cH\to\cH_{A,B},R_{A,B},S_{A,B})$ is a compatible
representation (in the sense of \cite{PW1}) of two positive forms 
\[
\alpha(\xi,\eta) := \langle A\xi,\eta\rangle, \quad \beta(\xi,\eta):=\langle B\xi,\eta\rangle,
\qquad \xi,\eta \in \cH.
\]
Then, Pusz and Woronowicz's original functional calculus $\phi(\alpha,\beta)$ becomes, by their
construction in \cite{PW1,PW2}, as follows:
\begin{align*}
\phi(\alpha,\beta)(\xi,\xi)
&=
\< \phi(R_{A,B},S_{A,B})T_{A,B}\xi,T_{A,B}\xi\> \\
&=
\int_{[0,1]^2} \phi(x,y)\,d\< T_{A,B}^* E(x,y)T_{A,B}\xi,\xi\> \\
&=
(T_{A,B}^* \phi(R_{A,B},S_{A,B})T_{A,B})(\omega_\xi) \\
&=
\phi(A,B)(\omega_\xi),
\end{align*}
where $\omega_\xi(X):=\langle X\xi,\xi\rangle$ for $\xi\in\cH$, $X \in B(\cH)$. Thus, the formulation
of the PW-functional calculus in this section agrees with their original one. Therefore, their
technique of obtaining variational integral expressions is available in our setting too; yet we will
discuss in \S\ref{Sec-9} integral expressions in a different manner.
\end{remark}

Here is a basic property of the PW-functional calculus, saying that the PW-functional calculus is well behaved with respect to direct sums. This will be used in the proof of Theorem \ref{T-4.9} below. 

\begin{proposition}\label{P-4.5}
Let $\phi$ be any Borel function on $[0,\infty)^2$ which is homogeneous and locally bounded from
below. If $A,B\in B(\cH)_+$ be given in direct sums as 
\[
\cH = \bigoplus_{i\in I} \cH_i, \quad A = \bigoplus_{i\in I}A_i, \quad B = \bigoplus_{i\in I}B_i
\]
with $A_i, B_i \in B(\cH_i)_+$, then we have
\[ 
\phi(A,B)(\rho)=\sum_{i\in I}\phi(A_i,B_i)(P_i \rho P_i)
\] 
for all $\rho\in B(\cH)_*^+$, where $P_i$ is the orthogonal projection $\mathcal{H}$ onto $\mathcal{H}_i$ for each $i\in I$. When $\phi$ is bounded,  
\[
\phi(A,B)=\bigoplus_{i\in I}\phi(A_i,B_i)
\]
holds on $\mathcal{H} = \bigoplus_{i\in I}\mathcal{H}_i$ in the usual sense.
\end{proposition}

\begin{proof}
We easily observe that 
\[
\mathcal{H}_{A,B} = \bigoplus_{i\in I} \mathcal{H}_{A_i,B_i}, \quad
T_{A,B} = \bigoplus_{i\in I} T_{A_i,B_i}, \quad
R_{A,B} = \bigoplus_{i\in I}R_{A_i,B_i}, \quad S_{A,B} = \bigoplus_{i\in I} S_{A_i,B_i} 
\]
on $\mathcal{H} = \bigoplus_{i\in I}\mathcal{H}_i$. Let $n \in \mathbb{N}$ be fixed for a while,
and set $\phi_n(x,y) := \phi(x,y)\wedge n$, a bounded Borel function on $[0,\infty)^2$. Since the
usual functional calculus is compatible with direct sum, we observe that 
\[
\phi_n(R_{A,B},S_{A,B}) = \bigoplus_{i\in I} \phi_n(R_{A_i,B_i},S_{A_i,B_i}),
\]
and hence
\[
T_{A,B}^*\phi_n(R_{A,B},S_{A,B})T_{A,B} = \bigoplus_{i\in I} T_{A_i,B_i}^*\phi_n(R_{A_i,B_i},S_{A_i,B_i})T_{A_i,B_i}.
\]
It follows that
\begin{align*}
(T_{A,B}\rho\,T_{A,B}^*)(\phi_n(R_{A,B},S_{A,B})) 
&= 
\rho(T_{A,B}^*\phi_n(R_{A,B},S_{A,B})T_{A,B}) \\
&= 
\sum_{i\in I} \rho(P_i T_{A_i,B_i}^*\phi_n(R_{A_i,B_i},S_{A_i,B_i})T_{A_i,B_i} P_i) \\
&= 
\sum_{i\in I} (T_{A_i,B_i} (P_i\rho P_i) T_{A_i,B_i}^*)(\phi_n(R_{A_i,B_i},S_{A_i,B_i})) 
\end{align*} 
for any  $\rho \in B(\mathcal{H})_*^+$. Since $\phi$ is lower bounded, we can take the limit as $n\to\infty$ (see the proof of Lemma \ref{L-3.1}) to obtain 
\[
(T_{A,B}^*\phi(R_{A,B},S_{A,B})T_{A,B})(\rho) = 
\sum_{i \in I} (T_{A_i,B_i}^*\phi(R_{A_i,B_i},S_{A_i,B_i})T_{A_i,B_i})(P_i\rho P_i). 
\]
The statement when $\phi$ is bounded is clear from the above discussion. 
\end{proof}

The following special values of $\phi(A,B)$ will play a role in our further study of the PW-functional
calculus.

\begin{lemma}\label{L-4.6}
Let $A\in B(\cH)_+$ with the spectral measure $E_A$, and $\alpha,\beta\in[0,\infty)$. Then the
following hold:
\begin{itemize}
\item[\rm(1)] $\phi(A,\alpha I)=\phi(A,\alpha)$ and $\phi(\alpha I,A)=\phi(\alpha,A)$, where
$\phi(A,\alpha),\phi(\alpha,A)\in\widehat{B(\cH)}_\lb$ are defined by
\[
\phi(A,\alpha)(\rho) := \int_{\sigma(A)}\phi(t,\alpha)\,d\rho(E_A(t)), \quad 
\phi(\alpha,A)(\rho) := \int_{\sigma(A)}\phi(\alpha,t)\,d\rho(E_A(t))
\]
for every $\rho \in B(\cH)_*^+$.
\item[\rm(2)] $\phi(\alpha A,\beta A)=\phi(\alpha,\beta)A$, where $\infty A$ means
$\infty\cdot E_A((0,\infty))$. In particular, $\phi(A,O)=\phi(A,0)=\phi(1,0)A$ and
$\phi(O,A)=\phi(0,A)=\phi(0,1)A$, where $O$ denotes the zero operator at this moment
to distinguish it from the scalar zero.
\end{itemize}
\end{lemma}

\begin{proof}
(1)\enspace
By Lemma \ref{L-2.4} note that $\phi(A,\alpha),\phi(\alpha,A) \in \widehat{B(\cH)}_\lb$ are well
defined. Let $E$ be the joint spectral measure of $(A,\alpha I)$. For any Borel set
$\Phi,\Psi \subseteq [0,\infty)$, since $E(\Phi\times\Psi) = \chi_\Psi(\alpha)E_A(\Phi)$ holds
by definition, it follows that
\[
\rho(E(\Phi\times\Psi)) = \chi_\Psi(\alpha)\,\rho(E_A(\Phi)) = 
\int_{\sigma(T)} \chi_{\Phi\times\Psi}(t,\alpha)\,d\rho(E_A(t))
\]
for each $\rho \in B(\cH)_*^+$. Hence, by the monotone class theorem,  
\[
\rho(E(\Lambda)) = \int_{\sigma(T)} \chi_\Lambda(t,\alpha)\,d\rho(E_A(t))
\]
holds for any Borel set $\Lambda \subseteq [0,\infty)^2$. This means that
$\phi(A,\alpha I) = \phi(A,\alpha)$ holds. The equality $\phi(\alpha I,A) = \phi(\alpha,A)$ can be
shown in the same way.

(2)\enspace
From the operator homogeneity with $C=A^{1/2}$, it follows that
\[
\phi(\alpha A,\beta A)=A^{1/2}\phi(\alpha I_{\cH_A},\beta I_{\cH_A})A^{1/2},
\]
where $\cH_A:=\overline{\ran}\,A$. By definition,
$\phi(\alpha I_{\cH_A},\beta I_{\cH_A})=\phi(\alpha,\beta)I_{\cH_A}$ is clear. Hence
$\phi(\alpha A,\beta A)=\phi(\alpha,\beta)A$ holds. The last two equalities are immediate from
this and item (1).
\end{proof}

The next two theorems are our main observations in this section. 

\begin{theorem}\label{T-4.7}{\rm(cf.~\cite[Theorem 4(1)]{HU})}\enspace
For any pair $(A,B)$ in $B(\cH)_+$ we have
\[
\phi(A,B) = 
\begin{cases} 
A^{1/2} \phi(1,A^{-1/2}BA^{-1/2})A^{1/2} & (\text{if $A\in B(\cH)_{++}$}), \\
B^{1/2} \phi(B^{-1/2}AB^{-1/2},1)B^{1/2} & (\text{if $B\in B(\cH)_{++}$}),
\end{cases}
\]
where $\phi(1,A^{-1/2}BA^{-1/2})$ and $\phi(B^{-1/2}AB^{-1/2},1)$ are given as in Lemma
\ref{L-4.6}(1).
\end{theorem}

\begin{proof}
Assume that $A\in B(\cH)_{++}$. The operator homogeneity with $C=A^{1/2}$ implies that 
\begin{align*} 
\phi(A,B) 
&= \phi(A^{1/2}I A^{1/2},A^{1/2}(A^{-1/2}BA^{-1/2})A^{1/2}) \\
&= A^{1/2}\phi(I,A^{-1/2}BA^{-1/2})A^{1/2} \\
&= A^{1/2}\phi(1,A^{-1/2}BA^{-1/2})A^{1/2}
\end{align*} 
by Lemma \ref{L-4.6}(1). The other case can be confirmed in the same way.
\end{proof}

The discussion above uses operator homogeneity (arising from the homogeneity of $\phi$)
explicitly and hence is more conceptual than that of \cite[Theorem 4(1)]{HU}.

\begin{remark}\label{R-4.8}\rm
As shown in \cite{HU}, any \emph{operator connection} (in the Kubo--Ando sense \cite{KA})
$A\sigma B$ for $A,B\in B(\cH)_+$ corresponding to an ($\bR$-valued) positive operator
monotone function $h$ on $[0,\infty)$ is captured as the PW-functional calculus $\phi(A,B)$
associated with the function $\phi$ defined by
\begin{align}\label{F-4.7}
\phi(x,y):=xh(y/x)\quad(x>0,\ y\ge0),\qquad
\phi(0,y):=\beta y\quad (y\ge0),
\end{align}
where $\beta:=\lim_{t\to\infty}h(t)/t$. In this setting, operator homogeneity in Definition \ref{D-4.1}
was formerly shown in \cite[Theorem 3]{Fu2} by appealing to \cite[Theorem 3.4]{KA}, and recently
extended to operator connections of positive $\tau$-measurable operators (in the von Neumann
algebra setting) in \cite[Theorem 3.31]{HK}. In our general formalism, this property is incorporated
into the definition of the PW-functional calculus. When $A$ is invertible, the first expression of
Theorem \ref{T-4.7} becomes
\[
\phi(A,B)=A^{1/2}h(A^{-1/2}BA^{-1/2})A^{1/2},
\]
which is a familiar expression of $A\sigma B$. 
\end{remark}

The next theorem is concerned with the joint convexity problem for the PW-functional calculus,
whose weaker form was given in \cite[Theorem 9]{HU}. The present statement can be understood
as a Hilbert space operator reformulation of Pusz and Woronowicz's original results in \cite{PW2}
with a thorough proof.

\begin{theorem}\label{T-4.9}{\rm(cf.~\cite[Theorem 9]{HU})}\enspace
The following conditions are equivalent, where Hilbert spaces $\cH,\cK$ are arbitrary:
\begin{itemize}
\item[\rm(i)] for every $A_i,B_i\in B(\cH)_+$ ($i=1,2$) and any $\lambda\in(0,1)$,
\[
\phi((1-\lambda)A_1+\lambda A_2,(1-\lambda)B_1+\lambda B_2)
\le (1-\lambda)\phi(A_1,B_1)+\lambda\phi(A_2,B_2),
\]
or equivalently,
\[
\phi(A_1+A_2,B_1+B_2)\le\phi(A_1,B_1)+\phi(A_2,B_2)\quad
(\mbox{jointly subadditive});
\]
\item[\rm(ii)] for every $A,B\in B(\cH)_+$ and any bounded operator $C:\cK\to\cH$,
\[
\phi(C^*AC,C^*BC) \leq C^*\phi(A,B)C;
\]
\item[\rm(iii)] for every $A,B\in B(\cH)_+$ and any isometry $V:\cK\to\cH$,
\[
\phi(V^*AV,V^*BV) \leq V^*\phi(A,B)V;
\]
\item[\rm(iv)] $t \in [0,1] \mapsto \phi(t,1-t) \in (-\infty,\infty]$ is operator convex (see
Theorem \ref{T-3.7}(1));
\item[\rm(iv$'$)] $t \in [0,1] \mapsto \phi(1-t,t) \in (-\infty,\infty]$ is operator convex;
\item[\rm(v)] $t \in [0,\infty) \mapsto \phi(t,1) \in (-\infty,\infty]$ is operator convex and
\begin{align}\label{F-4.8}
\phi(1,0)\ge\lim_{t\to\infty}{\phi(t,1)\over t},
\end{align}
or equivalently, $\phi(1,0)\ge\lim_{t\nearrow1}\phi(t,1-t)$;
\item[\rm(v$'$)] $t \in [0,\infty) \mapsto \phi(1,t) \in (-\infty,\infty]$ is operator convex and
\begin{align}\label{F-4.9}
\phi(0,1)\ge\lim_{t\to\infty}{\phi(1,t)\over t},
\end{align}
or equivalently, $\phi(0,1)\ge\lim_{t\nearrow1}\phi(1-t,t)$.
\end{itemize} 
\end{theorem}

\begin{proof}
The proof of the equivalence of (i)--(iii) is essentially the same as that of \cite[Theorem 2.1]{HP}
(also \cite[Theorem 2.5.2]{Hi}), so a point here is how the proof of this part goes in the
framework of $\widehat{B(\cH)}_\lb$. In the proof we repeatedly use Proposition 4.5 (for $I = {1, 2}$).

(i)$\implies$(ii).\enspace
We may assume that $C:\cK\to\cH$ is a contraction, since $\phi(\alpha A,\alpha B)=\alpha\phi(A,B)$
for any $\alpha>0$ (the special case of operator homogeneity). On the direct sum $\cK\oplus\cH$
define $\widetilde A:=0\oplus A$, $\widetilde B:=0\oplus B$ and unitaries
\[
U:=\begin{bmatrix}Y&-C^*\\C&Z\end{bmatrix},\qquad
V:=\begin{bmatrix}Y&C^*\\-C&Z\end{bmatrix},
\]
where $Y:=(I_\cK-C^*C)^{1/2}$ and $Z:=(I_\cH-CC^*)^{1/2}$. Then one can easily verify that
\begin{align*}
&\phi(C^* AC\oplus ZAZ, C^* BC\oplus ZBZ) \\
&\quad=\phi\bigl(\half(U^*\widetilde AU+V^*\widetilde AV),
\half(U^*\widetilde BU+V^*\widetilde BV)\bigr) \\
&\quad\le\half\bigl(\phi(U^*\widetilde AU,U^*\widetilde BU)
+\phi(V^*\widetilde AV,V^*\widetilde BV)\bigr)
\qquad\mbox{(by assumption (i))} \\
&\quad=\half\bigl(U^*\phi(\widetilde A, \widetilde B)U+V^*\phi(\widetilde A, \widetilde B)V\bigr)
\qquad\mbox{(by operator homogeneity)}.
\end{align*}
Therefore, for any $\rho\in B(\cK)_*^+$ we find that
\begin{align*}
&\phi(C^* AC,C^*BC)(\rho) \\
&\quad= \phi(C^* AC\oplus ZAZ, C^* BC\oplus ZBZ)(\rho\oplus 0)
\qquad\mbox{(by Proposition \ref{P-4.5})}\\
&\quad\le\half\bigl(\phi(\widetilde A, \widetilde B)(U(\rho\oplus 0)U^*)
+\phi(\widetilde A, \widetilde B)(V(\rho\oplus0)V^*)\bigr) \\
&\quad=\half\bigl(\phi(A,B)(P_2 U(\rho\oplus 0)U^*P_2)
+\phi(A,B)(P_2V(\rho\oplus0)V^*P_2)\bigr) \quad\mbox{(by Proposition \ref{P-4.5})}\\
&\quad= \phi(A,B)(C\rho\,C^*) = (C^*\phi(A,B)C)(\rho),
\end{align*}
where $P_2$ denotes the orthogonal projection from $\cK\oplus\cH$ onto the second direct
summand.

(ii)$\implies$(iii) is trivial.

(iii)$\implies$(i).\enspace
For any $\lambda\in(0,1)$ consider an isometry $V\xi:=\sqrt{1-\lambda}\xi\oplus\sqrt\lambda\xi$
from $\cH$ to $\cH\oplus\cH$. For every $A_i,B_i\in B(\cH)_+$ ($i=1,2$) one has
\begin{align*}
\phi((1-\lambda)A_1+\lambda A_2,(1-\lambda)B_1+\lambda B_2)
&=\phi(V^*(A_1\oplus A_2)V,V^*(B_1\oplus B_2)V)) \\
&\le V^*(\phi(A_1\oplus A_2,B_1 \oplus B_2))V
\end{align*}
by assumption (iii). Therefore, for any $\rho\in B(\cH)_+^*$ one has
\begin{align*}
&\phi((1-\lambda)A_1+\lambda A_2,(1-\lambda)B_1+\lambda B_2)(\rho) \\
&\quad\le\phi(A_1\oplus A_2,B_1 \oplus B_2)(V\rho\,V^*) \\
&\quad=\phi(A_1,B_1)(P_1 V\rho\,V^* P_1) + \phi(A_2,B_2)(P_2 V\rho\,V^* P_2)
\quad \mbox{(by Proposition \ref{P-4.5})} \\
&\quad=(1-\lambda)\phi(A_1,B_1)(\rho)+\lambda\phi(A_2,B_2)(\rho), 
\end{align*}
where $P_1, P_2$ denote the orthogonal projections from $\cH\oplus\cH$ onto the first and the
second direct summands, respectively. Note also that the equivalence of joint convexity and joint
subadditivity in (i) is clear from the scalar homogeneity of $\phi(A,B)$ mentioned at the beginning
of the proof of (i)$\implies$(ii).

(iii)$\implies$(iv) and (iii)$\implies$(iv$'$) are obvious by Lemma \ref{L-4.3}(3) (and Proposition
\ref{P-3.3}).

(iv)$\implies$(ii).\enspace
Write $A' = C^*AC$ and $B'=C^*BC$ for short. By Lemma \ref{L-4.3}(2) note that
$R_{A',B'} = \widehat{C}^*R_{A,B}\widehat{C}$ and $\widehat{C}T_{A',B'} = T_{A,B}C$ for an
isometry $\widehat{C}$. We then have 
\begin{align*} 
\phi(A',B') 
&= T_{A',B'}^*\phi(R_{A',B'},S_{A',B'})T_{A',B'} \\
&= 
T_{A',B'}^*\phi(R_{A',B'},1-R_{A',B'})T_{A',B'}\quad\quad\ \text{(by Lemma \ref{L-4.3}(3))} \\
&= 
T_{A',B'}^*\phi(\widehat{C}^* R_{A,B}\widehat{C},1-\widehat{C}^* R_{A,B}\widehat{C}^*) T_{A',B'} \\
&\leq 
T_{A',B'}^* \widehat{C}^* \phi(R_{A,B},1-R_{A,B})\widehat{C}\,T_{A',B'} \\
&\hskip4cm \text{(by assumption (iv) and Proposition \ref{P-3.3})} \\
&= 
T_{A',B'}^* \widehat{C}^*\phi(R_{A,B},S_{A,B})\widehat{C}\,T_{A',B'}
\quad\quad\quad \text{(by Lemma \ref{L-4.3}(3))} \\
&= 
C^* T_{A,B}^*\phi(R_{A,B},S_{A,B})T_{A,B} C \\
&= 
C^*\phi(A,B)C,
\end{align*}
where the first and the last equalities are due to operator homogeneity (Definition \ref{D-4.1}(2)).
The proof of  (iv$'$)$\implies$(ii) is similar. (Also, (iv)$\iff$(iv$'$) is obvious.)

(iii)$\implies$(v).\enspace
Assume item (iii). Then $\phi(t,1)$ is operator convex on $[0,\infty)$ by Lemma \ref{L-4.6}(1)
(and Proposition \ref{P-3.3}). Since (iii)$\implies$(iv) (already seen), it follows from Theorem
\ref{T-3.7} that $\phi(1,0)\ge\lim_{t\nearrow1}\phi(t,1-t)$. Furthermore, since
\[
{\phi(t,1)\over t}
= {t+1\over t}\,\phi\Bigl({t\over t+1},1-{t\over t+1}\Bigr),\qquad t>0,
\]
it is clear that the above condition is equivalent to \eqref{F-4.8}. The proof of (iii)$\implies$(v$'$)
is similar.

(v)$\implies$(iv).\enspace
Assume item (v). Then, as in the proof of \cite[Theorem 9, (iii)$\implies$(ii)]{HU}, we observe that
$t \in [0,1-\delta] \mapsto \phi(t,1-t) \in (-\infty,\infty]$ is operator convex for any $\delta\in(0,1)$ and
hence so is $t \in [0,1) \mapsto \phi(t,1-t) \in (-\infty,\infty]$. Here let us briefly explain this proof.
For each $\delta \in(0,1)$ we set $c_\delta := (1-\delta)/\delta$ and define
\[
\psi_\delta(x,y) :={x+y\over c_\delta}\,\phi\Bigl({c_\delta x\over x+y},1\Bigr)
\qquad\mbox{with\ \ $\psi_\delta(0,0) := 0$}.
\]
Then $\psi_\delta(t,c_\delta -t) = \phi(t,1)$ holds for $0 \leq t \leq c_\delta$ and
$\psi_\delta(t,1-t) = \phi(c_\delta t,1)/c_\delta$ is operator convex on $[0,1]$ by assumption (v).
Applying (iv)$\implies$(iii) (already established) to $\psi_\delta$, we see  that
$(A,B) \mapsto \psi_\delta(A,B)$ satisfies the inequality in (iii). Observe that
$\phi(t,1-t) = \psi_\delta(t,c_\delta(1-t)-t)$ for every $t\in[0 ,1-\delta]$. Then, for every
$A\in B(\cH)_{[0,1-\delta]}$ and every isometry $V:\cK\to\cH$ (so $V^*AV\in B(\cK)_{[0,1-\delta]}$),
we have
\begin{align*}
\phi(V^* AV,1-V^*AV) 
&= 
\psi_\delta(V^* AV,c_\delta(I_\cK-V^* AV)-V^* AV) \\
&=
\psi_\delta(V^* AV,V^*(c_\delta(I_\cH-A)-A)V) \\
&\leq
V^*\psi_\delta(A,c_\delta(I_\cH-A)-A)V \\
&= 
V^*\phi(A,1-A)V,
\end{align*}
where the first and the last equalities can be confirmed as in the proof of Lemma \ref{L-4.3}(3).
Hence we have shown that $\phi(t,1-t)$ is operator convex on $[0,1)$. To prove item (iv), in view
of Example \ref{E-3.6} we may assume that $\phi(t,1-t)<\infty$ at more than one point in $[0,1]$.
Then, since $\phi(1,0)\ge\lim_{t\nearrow1}\phi(t,1-t)$, we obtain (iv) by applying Theorem
\ref{T-3.7}(1) to $\phi(t,1-t)$ on $[0,1]$.

(v$'$)$\implies$(iv$'$) can be shown in exactly the same way as above.
\end{proof}

Under an additional assumption on the behavior of $\phi(t,1-t)$ as $t\nearrow1$, we have alternative
equivalent conditions to those of Theorem \ref{T-4.9}.

\begin{theorem}\label{T-4.10}
Assume that $\lim_{t\nearrow1}\phi(t,1-t)\le\phi(1,0)\le0$. Then the conditions of Theorem
\ref{T-4.9} are also equivalent to the following:
\begin{itemize}
\item[\rm(vi)] for every $A_1,A_2,B\in B(\cH)_+$,
\[
A_1\le A_2\,\implies\,\phi(A_1,B)\ge\phi(A_2,B);
\]
\item[\rm(vii)] $t\in[0,\infty)\mapsto\phi(t,1)\in(-\infty,\infty]$ is operator monotone decreasing
(see Theorem \ref{T-3.7}(2)).
\end{itemize}
\end{theorem}

\begin{proof}
(vi)$\implies$(vii).\enspace
This is immediately seen by letting $B=I$ in (vi) and using Lemma \ref{L-4.6}(1).

(vii)$\implies$(v).\enspace
Since $\lim_{t\nearrow1}\phi(t,1-t)\le0$ or equivalently $\lim_{t\nearrow\infty}\phi(t,1)/t \leq 0$,
it follows that  $\phi(t,1)<\infty$ for all sufficiently large $t>0$.  Hence by applying Theorem
\ref{T-3.7}(2) to $-\phi(t,1)$, assumption (vii) implies that $\phi(t,1)$ is $\bR$-valued and operator
monotone decreasing on $(0,\infty)$ and $\phi(0,1)\ge\lim_{t\searrow0}\phi(t,1)$. Then it is
well known that $\phi(t,1)$ is operator convex on $(0,\infty)$. (Indeed, let $f(t):=\phi(t,1)$.
For every $\eps>0$, $f(\eps)-f(t+\eps)$ is non-negative and operator monotone on $[0,\infty)$,
which is operator concave by \cite[Theorem 2.5]{HP}. This means that $f$ is operator convex
on $(0,\infty)$.) Hence, item (v) holds by Theorem \ref{T-3.7}(1).

(i)$\implies$(vi).\enspace
Let $A_1,A_2,B\in B(\cH)_+$ with $A_1\le A_2$. By assumption (i) one has
\[
\phi(A_2,B)=\phi((A_1+(A_2-A_1),B+0)\le\phi(A_1,B)+\phi(A_2-A_1,0).
\]
Since $\phi(A_2-A_1,0) = \phi(1,0)(A_2-A_1)$ by Lemma \ref{L-4.6}(2),  item (vi) holds
thanks to the assumption $\phi(1,0) \le 0$.
\end{proof}

\begin{example}\label{E-4.11}\rm
Let $h$ be a positive operator monotone function on $[0,\infty)$ with $h(1)=1$. Define the
function $\phi(x,y)$ as in Remark \ref{R-4.8}. Then for every $A,B\in B(\cH)_+$, $\phi(B,A)$
becomes the \emph{operator mean} $A\sigma B$ \cite{KA} corresponding to the function $h$.
Since $\phi(t,1)=h(t)$ is operator concave (i.e., $-\phi(t,1)$ is operator convex) and
$\lim_{t\to\infty}\phi(t,1)/t=\phi(1,0)$, Theorem \ref{T-4.9}(i) (applied to $-\phi$) reduces to the
familiar joint concavity of an operator mean. In particular, for $\phi_\alpha(x,y)$ corresponding to
$t^\alpha$ with $0\le\alpha\le1$, we have $\phi_\alpha(B,A)=A\#_\alpha B$, the
\emph{$\alpha$-weighted geometric mean}, which has been the most studied operator mean
since the beginning of the subject \cite{PW1,PW2,An0,An1}. 
\end{example}

\begin{remark}\label{R-4.12}\rm
Some remarks related to Theorems \ref{T-4.9} and \ref{T-4.10} are in order.

(1)\enspace
Assume that $\phi(t,1-t)<\infty$ at more than one point in $[0,1]$. If $\phi$ satisfies the
convexity properties of Theorem \ref{T-4.9}, then $\phi(x,y)$ becomes a real analytic function
on $(0,\infty)^2$. Indeed, in this case, $\phi(t,1)$ is an $\bR$-valued operator convex
(hence real analytic) function on $(0,\infty)$ and $\phi(x,y)=y\phi(x/y,1)$.

(2)\enspace
Theorem \ref{T-4.9} in particular shows that an $\bR$-valued function $f$ on $(0,\infty)$ is
operator convex if and only if so is $g(t):=(1-t)f\bigl({t\over1-t}\bigr)$ on $(0,1)$. Indeed, let
$\phi$ be the perspective function of $f$, i.e., $\phi(x,y):=yf(x/y)$ for $(x,y)\in(0,\infty)^2$
(extended to $[0,\infty)^2$ with $\phi(0,0)=0$ and $\phi(1,0)=\phi(0,1)=\infty$). 
Since $f(t)=\phi(t,1)$ for $t>0$ and $g(t)=\phi(t,1-t)$ for $t\in(0,1)$, the
assertion follows from Theorem \ref{T-4.9}. For a different approach to this result, see \cite{Fu1}
and \cite[Theorem C.1]{Hi3}.

(3)\enspace
Assume that $\phi(t,1-t)<\infty$ for all $t\in(0,1)$. Then Theorem \ref{T-4.7} says that $\phi(A,B)$
for every $A,B\in B(\cH)_{++}$ becomes the \emph{operator perspective}
$B^{1/2}f(B^{-1/2}AB^{-1/2})B^{1/2}$ of $A,B$ associated with the function $f(t):=\phi(t,1)$ on
$(0,\infty)$; see Definition \ref{D-7.1} in \S\ref{Sec-7}. Thus Theorem \ref{T-4.9} in particular
contains the result in \cite{Ef,ENG}, saying that an $\bR$-valued function $f$ on $(0,\infty)$ is
operator convex if and only if the operator perspective associated with $f$ is jointly operator
convex on $B(\cH)_{++}\times B(\cH)_{++}$.
\end{remark}

The PW-functional calculus $\phi(A,B)$ enjoys several continuity properties under certain
constraints on the given function $\phi$, as we will discuss in \S\ref{Sec-6} and \S\ref{Sec-7}.

\section{PW-functional calculus with restricted domain}\label{Sec-5}

For the convenience of notations, we set
\begin{align*}
(B(\cH)_+\times B(\cH)_+)_\le
&:=\{(A,B)\in B(\cH)_+\times B(\cH)_+:A\le\alpha B\ \mbox{for some $\alpha>0$}\}, \\
(B(\cH)_+\times B(\cH)_+)_\ge
&:=\{(A,B)\in B(\cH)_+\times B(\cH)_+:A\ge\alpha B\ \mbox{for some $\alpha>0$}\},
\end{align*}
which are sub-cones of $B(\cH)_+\times B(\cH)_+$ including $B(\cH)_{++}\times B(\cH)_{++}$.

In the following we continue to use the notations $\cH_{A,B}$, $T_{A,B}$ and $(R_{A,B},S_{A,B})$
given in \eqref{F-4.1}, \eqref{F-4.2} and Lemma \ref{L-4.3}(1) for $A,B\in B(\cH)_+$.

\begin{lemma}\label{L-5.1}
For every $A,B\in B(\cH)_+$ the following conditions are equivalent:
\begin{itemize}
\item[\rm(1)] $(A,B)\in(B(\cH)_+\times B(\cH)_+)_\ge$;
\item[\rm(2)] $R_{A,B}\ge\alpha S_{A,B}$ for some $\alpha>0$;
\item[\rm(3)] $R_{A,B}\in B(\cH_{A,B})_{++}$.
\end{itemize}
\end{lemma}

\begin{proof}
(1)$\implies$(2).\enspace
Assume that $A\ge\alpha B$ with $\alpha>0$. For every $\xi\in\cH$ one has
\[
\<R_{A,B}T_{A,B}\xi,T_{A,B}\xi\>=\<A\xi,\xi\>\ge\alpha\<B\xi,\xi\> \\
=\alpha\<S_{A,B}T_{A,B}\xi,T_{A,B}\xi\>
\]
so that $R_{A,B}\ge\alpha S_{A,B}$.

(2)$\implies$(3).\enspace
Assume (2); then $(\alpha+1)R_{A,B}\ge\alpha(R_{A,B}+S_{A,B})=\alpha I_{\cH_{A,B}}$.

(3)$\implies$(1).\enspace
Assume (3); then one can choose an $\alpha>0$ such that $R_{A,B}\ge\alpha S_{A,B}$. Hence
$A=T_{A,B}^*R_{A,B}T_{A,B}\ge\alpha T_{A,B}^*S_{A,B}T_{A,B}=\alpha B$.
\end{proof}

In this section we consider a Borel function
$\phi:[0,\infty)^2\setminus(\{0\}\times(0,\infty))\to(-\infty,\infty]$ which is homogeneous
(i.e., $\phi(\lambda x,\lambda y)=\lambda\phi(x,y)$ for all $x>0$, $y\ge0$ and $\lambda\ge0$).
We assume that $\phi(t,1-t)$ on $(0,1]$ is locally bounded from below (i.e., bounded from
below on any compact subset of $(0,1]$). We can define the PW-functional calculus
\[
(A,B)\in(B(\cH)_+\times B(\cH)_+)_\ge\,\longmapsto\,\phi(A,B)\in\widehat{B(\cH)}_\lb
\]
by restricting $(A,B)$ to pairs in $(B(\cH)_+\times B(\cH)_+)_\ge$ in postulates (1) and (2) of
Definition \ref{D-4.1}. Alternatively, by Lemma \ref{L-5.1} we may define more explicitly as
follows:
\begin{align}\label{F-7.1}
\phi(A,B)=T_{A,B}^*\phi(R_{A,B},S_{A,B})T_{A,B},\qquad
(A,B)\in(B(\cH)_+\times B(\cH)_+)_\ge.
\end{align}
The class of functions $\phi$ on $[0,\infty)^2\setminus(\{0\}\times(0,\infty))$ for which we can
define the PW-functional calculus $\phi(A,B)$ is somewhat more flexible than those on
$[0,\infty)^2$, although the domain is restricted to $(B(\cH)_+\times B(\cH)_+)_\ge$.

For this version of PW-functional calculus $\phi(A,B)$ with restricted domain, we can
rephrase all the results in \S\ref{Sec-4} by restricting $(A,B)$ to $(B(\cH)_+\times B(\cH)_+)_\ge$
with slight necessary modifications. For instance, Theorem \ref{T-4.7} in this setting is
\[
\phi(A,B) = 
\begin{cases} 
A^{1/2} \phi(1,A^{-1/2}BA^{-1/2})A^{1/2} & \text{for $A\in B(\cH)_{++}$, $B\in B(\cH)_+$}, \\
B^{1/2} \phi(B^{-1/2}AB^{-1/2},1)B^{1/2} & \text{for $A,B\in B(\cH)_{++}$}.
\end{cases}
\]

In the next theorem, instead of the restricted version of Theorems \ref{T-4.9} and \ref{T-4.10},
we present its complementary counterpart, where the inequality sign is reversed and `operator
convex' (resp., `operator monotone decreasing') is replaced with `operator concave' (resp.,
`operator monotone').

\begin{theorem}\label{T-5.2}
Let $\phi:[0,\infty)^2\setminus(\{0\}\times(0,\infty))\to(-\infty,\infty]$ be as stated above. Then
the following conditions are equivalent, where Hilbert spaces $\cH,\cK$ are arbitrary:
\begin{itemize}
\item[\rm(i)] for every $(A_i,B_i)\in(B(\cH)_+\times B(\cH)_+)_\ge$ ($i=1,2$),
\[
\phi(A_1+A_2,B_1+B_2)\ge\phi(A_1,B_1)+\phi(A_2,B_2);
\]
\item[\rm(ii)] for every $(A,B)\in(B(\cH)_+\times B(\cH)_+)_\ge$ and any bounded operator
$C:\cK\to\cH$,
\[
\phi(C^*AC,C^*BC) \ge C^*\phi(A,B)C;
\]
\item[\rm(iii)] for every $(A,B)\in(B(\cH)_+\times B(\cH)_+)_\ge$ and any isometry
$V:\cK\to\cH$,
\[
\phi(V^*AV,V^*BV) \ge V^*\phi(A,B)V;
\]
\item[\rm(iv)] $t \in (0,1] \mapsto \phi(t,1-t)$ is operator concave;
\item[\rm(iv$'$)] $t \in [0,1) \mapsto \phi(1-t,t)$ is operator concave;
\item[\rm(v)] $t \in (0,\infty) \mapsto \phi(t,1)$ is operator concave and
$\phi(1,0)\le\lim_{t\nearrow1}\phi(t,1-t)$;
\item[\rm(v$'$)] $t \in [0,\infty) \mapsto \phi(1,t)$ is operator concave.
\end{itemize}

Furthermore, assume that $\lim_{t\nearrow1}\phi(t,1-t)\ge\phi(1,0)\ge0$. Then the above conditions
are also equivalent to the following:
\begin{itemize}
\item[\rm(vi)] for every $(A_1,B),(A_2,B)\in(B(\cH)_+\times B(\cH)_+)_\ge$,
\[
A_1\le A_2\,\implies\,\phi(A_1,B)\le\phi(A_2,B);
\]
\item[\rm(vii)] $\phi(t,1)$ is operator monotone on $(0,\infty)$.
\end{itemize}
\end{theorem}

\begin{proof}
The proof of the equivalence of (i)--(iii) is the same as that for (i)--(iii) of Theorem \ref{T-4.9} just
by restricting $(A,B)$ to $(B(\cH)_+\times B(\cH)_+)_\ge$ and by reversing the inequality sign.
The proof for the remaining items can be carried out in a similar way to that of Theorem \ref{T-4.9}
with necessary modifications. Alternatively (and more conveniently), we may first confirm the
restricted (but not complementary) version of Theorems \ref{T-4.9} and \ref{T-4.10}
for $\phi$ in the present setting, which is stated as above with the reverse inequality sign and with
`operator convex' and `operator monotone decreasing' (instead of `operator concave' and
`operator monotone'). The proof of this version is carried out just by restricting $(A,B)$ etc.\ to
$(B(\cH)_+\times B(\cH)_+)_\ge$ in the proofs of Theorems \ref{T-4.9} and \ref{T-4.10} (with
Lemma \ref{L-5.1}). Next, let us prove the present complementary version. Note that all the
conditions hold trivially in the case $\phi\equiv\infty$. So we may assume that $\phi\not\equiv\infty$.
In this case, it is easy to observe the following:
\begin{itemize}
\item $\phi(t,1-t)$ is numerically concave on $(0,1]$ if one of (i) (hence (ii), (iii)), (iv) and (iv$'$) is
satisfied,
\item $\phi(t,1)$ is numerically concave on $(0,\infty)$ with $\phi(1,0)<\infty$ if one of (v), (vi) and
(vii) is satisfied,
\item $\phi(1,t)$ is numerically concave on $[0,\infty)$ if (v$'$) is satisfied.
\end{itemize}
From this observation we see that if any of the conditions of the theorem is satisfied, then
$\phi(t,1-t)$ is locally bounded from above (as well as from below) on $(0,1]$. Hence, to prove
the theorem, it suffices to assume that $\phi(t,1-t)$ is locally bounded from above on $(0,1]$.
Thus, we can apply the above-mentioned version of Theorems \ref{T-4.9} and \ref{T-4.10} to $-\phi$,
which immediately shows the present complementary version.
\end{proof}

A typical example of $\phi$ to apply Theorem \ref{T-5.2} is
\[
\phi(x,y):=\begin{cases}y\log(x/y) & (x,y>0), \\
0 & (x\ge0=y).\end{cases}
\]
Then $\phi(t,1)=\log t$ is operator monotone (and operator concave) on $(0,\infty)$ and
$\lim_{t\nearrow1}\phi(t,1-t)=0=\phi(1,0)$. The function $\phi(t,1-t)$ is locally bounded from below
on $(0,1]$, while it cannot be extended to a function on $[0,1]$ that is locally bounded from below,
since $\lim_{t\searrow0}\phi(t,1-t)=-\infty$.

In the next proposition we characterize the case where $\phi(A,B)$ is bounded for any
$(A,B)\in(B(\cH)_+\times B(\cH)_+)_\ge$.

\begin{proposition}\label{P-5.3}
Let $\phi$ be as above, and assume that $\cH$ is infinite-dimensional. Then the following
conditions are equivalent:
\begin{itemize}
\item[\rm(1)] $\phi(A,B)\in B(\cH)_\sa$ for all $(A,B)\in(B(\cH)_+\times B(\cH)_+)_\ge$;
\item[\rm(2)] $\phi(t,1-t)$ is locally bounded on $(0,1]$ (i.e., bounded on $[\delta,1]$ for any
$\delta>0$);
\item[\rm(3)] $\phi(1,t)$ is locally bounded on $[0,\infty)$.
\end{itemize}
\end{proposition}

\begin{proof}
(2)$\iff$(3) is immediately seen from
\begin{gather*}
\phi(1,t)=(1+t)\phi\Bigl({1\over 1+t},1-{1 \over 1+t}\Bigr)\quad(t\ge0),\\
\phi(t,1-t)=t\phi\Bigl(1,{1-t\over t}\Bigr)\quad(0<t\le1).
\end{gather*}

(2)$\implies$(1) follows from Lemma \ref{L-5.1} and \eqref{F-7.1}.

(1)$\implies$(3).\enspace
We prove by contraposition. Assume that there is a $c>0$ such that $\phi(1,t)$ is unbounded on
$[0,c]$. Choose a sequence $t_n\in[0,c]$ such that $\phi(1,t_n)\to\infty$. With a sequence
$\{P_n\}$ of orthogonal projections with $\sum_nP_n=I$, we define
$B:=\sum_{n=1}^\infty t_nP_n$ in $B(\cH)_+$. Then, modifying Lemma \ref{L-4.6}(1) in the
present setting, one has $\phi(I,B) = \phi(1,B) = \sum_{n=1}^\infty \phi(1,t_n)P_n$, which is
unbounded.
\end{proof}

For example, when $\phi$ ($\not\equiv\infty$) satisfies (v$'$) of Theorem \ref{T-5.2}, it is clear that
condition (3) of Proposition \ref{P-5.3} holds. Therefore, we have $\phi(A,B)\in B(\cH)_\sa$ for all
$(A,B)\in(B(\cH)_+\times B(\cH)_+)_\ge$ whenever one of the conditions of Theorem \ref{T-5.2}
is satisfied and $\phi\not\equiv\infty$.

\begin{remark}\label{R-5.4}\rm
As is immediately seen, Theorem \ref{T-5.2} and Proposition \ref{P-5.3} hold also in the situation
where $[0,\infty)^2\setminus(\{0\}\times(0,\infty))$ and $(B(\cH)_+\times B(\cH)_+)_\ge$ are
replaced with $[0,\infty)^2\setminus((0,\infty)\times\{0\})$ and $(B(\cH)_+\times B(\cH)_+)_\le$,
respectively, and the roles of two variables in $\phi(x,y)$ and $\phi(A,B)$ are interchanged.
\end{remark}

\section{Upper continuity for PW-functional calculus}\label{Sec-6}

Recall \cite{KA} that an operator connection $A\sigma B$ enjoys the upper continuity for
decreasing sequences in $B(\cH)_+$ in such a way that for $A,B,A_n,B_n\in B(\cH)_+$,
\begin{align}\label{F-6.1}
A_n\searrow A,\ B_n\searrow B\,\implies\,A_n\sigma B_n\searrow A\sigma B,
\end{align}
where $A_n\searrow A$ means that $A_1\ge A_2\ge\cdots$ and $A_n\to A$ in the strong operator
topology (SOT for short).

In this section we show the convergence of the PW-functional calculus in SOT for decreasing
sequences as above in certain situations. The next theorem vastly improves \cite[Theorem 6]{HU}.
In fact, it is an apparently best possible counterpart of \eqref{F-6.1} for operator connections in the
case of the PW-functional calculus, though $\phi$ is assumed $\bR$-valued continuous.

\begin{theorem}\label{T-6.1}
Let $\phi$ be a homogeneous and $\bR$-valued continuous function on $[0,\infty)^2$. Let
$A,B,A_n,B_n\in B(\cH)_+$ ($n\in\bN$) be such that $A_n\searrow A$ and $B_n\searrow B$. Then
$\phi(A_n,B_n)$, $\phi(A,B)$ are all bounded and
\begin{align*}
\phi(A_n,B_n)\,\longrightarrow\,\phi(A,B)\quad\mbox{in SOT}.
\end{align*}
\end{theorem}

\begin{proof}
Let $T_n:=(A_n+B_n)^{1/2}$ and $T:=(A+B)^{1/2}$ in $B(\cH)_+$; so $T_n\searrow T$ by
assumption. We have (unique) $R_n,R\in B(\cH)_+$ such that $0\le R_n,R\le I$,
$\overline\ran\,R_n\subseteq\overline\ran\,T_n$, $\overline\ran R\subseteq\overline\ran\,T$,
$A_n=T_nR_nT_n$ and $A=TRT$. Set $f(t):=\phi(t,1-t)$ for $t\in[0,1]$, which is a continuous
function on $[0,1]$ by assumption. As in \eqref{F-4.6} we can write
\begin{align}\label{F-6.2}
\phi(A_n,B_n)=T_nf(R_n)T_n,\qquad\phi(A,B)=Tf(R)T,
\end{align}
where $f(R_n)$ and $f(R)$ are the continuous functional calculus in the present situation. Hence
$\phi(A_n,B_n)$ and $\phi(A,B)$ are clearly bounded (as described in \cite{HU}). Let
$\cH_{A,B}:=\overline\ran\,T$ as before. One finds that, for every $\xi,\xi'\in\cH$,
\begin{equation}\label{F-6.3}
\begin{aligned}
|\<(R_n-R)T\xi,T\xi'\>|
&\le|\<R_n(T-T_n)\xi,T\xi'\>|+|\<R_nT_n\xi,(T-T_n)\xi'\>| \\
&\qquad+|\<R_nT_n\xi,T_n\xi'\>-\<RT\xi,T\xi'\>| \\
&\le\|T\xi'\|\,\|(T-T_n)\xi\|+\|T_1\|\,\|\xi\|\,\|(T-T_n)\xi'\| \\
&\qquad+|\<(A_n-A)\xi,\xi'\>| \\
&\longrightarrow\,0\quad\mbox{as $n\to\infty$}.
\end{aligned}
\end{equation}
Consider the function 
\[
\psi(x,y):=\begin{cases}{x^2\over x+y}=x-{xy\over x+y} & \text{if $x+y>0$}, \\
\ 0 & \text{if $x=y=0$}.\end{cases}
\]
Since $\psi(t,1-t)=t^2$ for $t\in[0,1]$, it follows as \eqref{F-6.2} that
\begin{align}\label{F-6.4}
T_nR_n^2T_n=\psi(A_n,B_n)=A_n-(A_n:B_n),\qquad
TR^2T=A-(A:B),
\end{align}
where $A:B$ is the \emph{parallel sum} of $A,B$, an operator connection having the
representing function  $t/(t+1)$; see Remark \ref{R-4.8}. Therefore, one
has (thanks to \eqref{F-6.1})
\[
T_nR_n^2T_n\,\longrightarrow\,TR^2T\quad\mbox{in SOT}.
\]
For every $\xi\in\cH$,
\begin{equation}\label{F-6.5}
\begin{aligned}
|\,\|R_nT\xi\|^2-\|RT\xi\|^2|
&\le|\<R_n^2(T-T_n)\xi,T\xi\>|+|\<R_n^2T_n\xi,(T-T_n)\xi\>| \\
&\qquad+|\<T_nR_n^2T_n\xi,\xi\>-\<TR^2T\xi,\xi\>| \\
&\le(\|T\|+\|T_1\|)\|\xi\|\,\|(T-T_n)\xi\| \\
&\qquad+\|\xi\|\,\|T_nR_n^2T_n\xi-TR^2T\xi\| \\
&\longrightarrow\,0\quad\mbox{as $n\to\infty$}.
\end{aligned}
\end{equation}
For every $\eta,\eta'\in\cH_{A,B}$, approximating $\eta,\eta'$ with $T\xi,T\xi'$ ($\xi,\xi'\in\cH$)
and applying the above estimates given in \eqref{F-6.3} and \eqref{F-6.5}, one can immediately
find that $\<R_n\eta,\eta'\>\to\<R\eta,\eta'\>$ and $\|R\eta\|\to\|R\eta\|$. By these with
$R\eta\in\cH_{A,B}$, s standard argument shows that $\|R_n\eta-R\eta\|\to0$ as $n\to\infty$ for all
$\eta\in\cH_{A,B}$ (though we cannot say that $R_n\to R$ in SOT on $\cH_{A,B}$, because
$\cH_{A,B}$ is not necessarily invariant for $R_n$'s).

Now, for any $k\in\bN$, let us prove by induction that $\|R_n^k\eta-R^k\eta\|\to0$ as
$n\to\infty$ for all $\eta\in\cH_{A,B}$. Assume that this holds for some $k\in\bN$. For every
$\eta\in\cH_{A,B}$, since $R^k\eta\in\cH_{A,B}$, we have
\begin{align*}
\|(R_n^{k+1}-R^{k+1})\eta\|
&\le\|R_n(R_n^k-R^k)\eta\|+\|(R_n-R)R^k\eta\| \\
&\le\|(R_n^k-R^k)\eta\|+\|(R_n-R)R^k\eta\| \\
&\longrightarrow\,0\quad\mbox{as $n\to\infty$},
\end{align*}
so that the convergence in question holds for $k+1$ too. Applying the Weierstrass
approxaimation theorem to $f$, we therefore see that $\|f(R_n)\eta-f(R)\eta\|\to0$ for all
$\eta\in\cH_{A,B}$. For every $\xi\in\cH$ it then follows from \eqref{F-6.2} that
\begin{align*}
\|\phi(A_n,B_n)\xi-\phi(A,B)\xi\|
&\le\|T_nf(R_n)(T_n-T)\xi\|+\|T_n(f(R_n)-f(R))T\xi\| \\
&\qquad+\|(T_n-T)f(R)T\xi\| \\
&\le\|T_1\|\,\|f\|_\infty\|(T_n-T)\xi\|+\|T_1\|\,\|(f(R_n)-f(R))T\xi\| \\
&\qquad+\|(T_n-T)f(R)T\xi\| \\
&\longrightarrow\,0\quad\mbox{as $n\to\infty$},
\end{align*}
where $\|f\|_\infty:=\max_{0\le t\le1}|f(t)|$. Hence the stated assertion follows.
\end{proof}

Concerning the above proof, we remark that the expression in \eqref{F-6.4} above was already
observed in the proof of \cite[Theorem 1.2]{PW1} and the strong convergence
$T_n R_n^2 T_n \to TRT$ was essentially derived there. However, the argument after obtaining
this convergence has been missing so that no general continuity result has probably been
observed so far for the PW-functional calculus.

\begin{corollary}\label{C-6.2}
Let $\phi:[0,\infty)^2\to(-\infty,\infty]$ be a function dealt with in \S\ref{Sec-4}, and assume that
$\phi(t,1-t)$ is continuous on $[0,1]$ as a function to $(-\infty,\infty]$. Then for every decreasing
sequences $A_n\searrow A$ and $B_n\searrow B$ in $B(\cH)_+$, we have
\[
\phi(A,B)(\rho)\le\liminf_{\eps\searrow0}\phi(A_n,B_n)(\rho),
\qquad\rho\in B(\cH)_*^+.
\]
\end{corollary}

\begin{proof}
Let $\ffi_k(t):=\phi(t,1-t)\wedge k$ for $t\in[0,1]$ and $k\in\bN$, and $\phi_k$ be the perspective
function of $\ffi_k$, i.e., $\phi_k(x,y)=(x+y)\ffi_k(x/(x+y))$ for $x+y>0$ and $\phi_k(0,0)=0$. For
each fixed $k$, it follows from Theorem \ref{T-6.1} that $\phi_k(A_n,B_n)\to\phi_k(A,B)$ in SOT.
Since $\phi_k(A_n,B_n)\le\phi(A_n,B_n)$ as immediately seen from \eqref{F-4.5}, one has
\[
\rho(\phi_k(A,B))=\lim_{n\to\infty}\rho(\phi_k(A_n,B_n))
\le\liminf_{n\to\infty}\phi(A_n,B_n)(\rho),\qquad\rho\in B(\cH)_+.
\]
Hence the assertion follows since $\phi_k(A,B)(\rho)\nearrow\phi(A,B)(\rho)$ as $k\to\infty$ by
the monotone convergence theorem.
\end{proof}

The next theorem is a modification of Theorem \ref{T-6.1} to the setting of \S\ref{Sec-5}. 

\begin{theorem}\label{T-6.3}
Let $\phi$ be a homogeneous and $\bR$-valued function on
$[0,\infty)^2\setminus(\{0\}\times(0,\infty))$ (resp., on
$[0,\infty)^2\setminus((0,\infty)\times\{0\})$) such that $\phi(t,1-t)$ is continuous on $(0,1]$
(resp., on $[0,1)$). Let $A,B,A_n,B_n\in B(\cH)_+$ ($n\in\bN$) be such that $A_n\ge\alpha B_n$
(resp., $A_n\le\alpha B_n$) for all $n$ with some $\alpha>0$ (independent of $n$),
$A_n\searrow A$ and $B_n\searrow B$. Then $\phi(A_n,B_n)$, $\phi(A,B)$ are all bounded and
\[
\phi(A_n,B_n)\,\longrightarrow\,\phi(A,B)\quad\mbox{in SOT}.
\]
\end{theorem}

\begin{proof}
The two cases are switched by interchanging the roles of two variables in $\phi(x,y)$ and
$\phi(A,B)$ as noted in Remark \ref{R-5.4}; so we will show only the second case. Let $T_n,T$ and
$R_n,R$ be taken in the same way as in the proof of Theorem \ref{T-6.1}. Also, we have
$S_n\in B(\cH)_+$ such that $0\le S_n\le I$, $\overline\ran\,S_n\subseteq\overline\ran\,T_n$ and
$B_n=T_nS_nT_n$. Since $A_n\le\alpha B_n$, it follows that $R_n\le\alpha S_n$ similarly to
Lemma \ref{L-5.1}, so that $R_n\le(\alpha/(\alpha+1))I$ holds thanks to $R_n+S_n\le I$. Since
$A\le\alpha B$ as well, $R\le(\alpha/(\alpha+1))I$ holds. Therefore, we have the same
expressions as in \eqref{F-6.2} in the present case too, where $f(R_n)$ and $f(R)$ are the
continuous functional calculus for $f(t):=\phi(t,1-t)$, $t\in[0,1)$. In particular, $\phi(A_n,B_n)$ and
$\phi(A,B)$ are bounded. Now the proof can be carried out in the same way as that of Theorem
\ref{T-6.1}.
\end{proof}

\section{Extended operator convex perspectives}\label{Sec-7}

In \S\ref{Sec-4} we have considered a Borel function $\phi:[0,\infty)^2\to(-\infty,\infty]$ which is
homogeneous and locally bounded below. Let $f(t):=\phi(t,1)$ for $t\in(0,\infty)$, $\alpha:=\phi(1,0)$
and $\beta:=\phi(0,1)$. As immediately seen, the local boundedness of $\phi$ from below is
rephrased as $\alpha,\beta>-\infty$ and} $f(t)\ge at+b$, $t\in(0,\infty)$, for some $a,b\in\bR$.
In this section we begin with a Borel function $f:(0,\infty)\to(-\infty,\infty]$ such that $f(t)\ge at+b$
for all $t\in(0,\infty)$ with some $a,b\in\bR$. With $\alpha,\beta\in(-\infty,\infty]$ given as
\begin{align}\label{F-7.1}
\alpha:=\limsup_{t\to\infty}{f(t)\over t},\qquad\beta:=\limsup_{t\searrow0}f(t),
\end{align}
we define the perspective function $\phi_f:[0,\infty)^2\to(-\infty,\infty]$ by
\begin{align}\label{F-7.2}
\phi_f(x,y):=\begin{cases}
yf(x/y) & (x,y>0), \\
\alpha x & (x\ge0,\ y=0), \\
\beta y & (x=0,\ y\ge0).
\end{cases}
\end{align}
Note that the positions of $x,y$ in the above definition of $\phi_f(x,y)$ are different from those
in \eqref{F-4.7} of Remark \ref{R-4.8}. This is because the roles of $A,B$ are reversed between
operator connections and operator perspectives of $A,B$ in the literature.

Then $\phi_f$ is homogeneous and locally bounded from below on $[0,\infty)^2$, so that the
PW-functional calculus $\phi_f(A,B)$ is defined for any $A,B\in B(\cH)_+$; see Definition \ref{D-4.1}.
Of course, $\alpha,\beta$ can be arbitrary numbers in $(-\infty,\infty]$, but \eqref{F-7.1} is suitable
for our discussions below. Also, note that if $f$ is numerically convex (or concave) on $(0,\infty)$,
then $\limsup$'s in \eqref{F-7.1} become $\lim$'s. In this case we write $\alpha=f'(\infty)$ (as
justified since $\lim_{t\to\infty}f(t)/t=\lim_{t\to\infty}f_+'(t)$) and $\beta=f(0^+)$.

\begin{definition}\label{D-7.1}\rm
Let $f$ and $\phi_f$ be mentioned as above. For every $A,B\in B(\cH)_+$ we have the
PW-functional calculus $\phi_f(A,B)\in\widehat{B(\cH)}_\lb$ (see Definition \ref{D-4.1} and
Theorem \ref{T-4.2}). We call $\phi_f(A,B)$ the \emph{(extended) operator perspective} of $A,B$
associated with $f$ as well.
\end{definition}

The \emph{transpose} of $\widetilde f$ of $f$ is defined by
\[
\widetilde f(t):=tf(t^{-1}),\qquad t\in(0,\infty),
\]
and set $\widetilde\alpha:=\limsup_{t\to\infty}\widetilde f(t)/t$,
$\widetilde\beta:=\lim_{t\searrow0}\widetilde f(t)$ as in \eqref{F-7.1}. Then $\widetilde f(t)\ge bt+a$
for $t\in(0,\infty)$, and $\widetilde\alpha=\beta$, $\widetilde\beta=\alpha$. Hence
$\phi_{\widetilde f}(x,y)=\phi_f(y,x)$ on $[0,\infty)^2$, so that we have
\begin{align}\label{F-7.3}
\phi_{\widetilde f}(A,B)=\phi_f(B,A),\qquad A,B\in B(\cH)_+.
\end{align}
Theorem \ref{T-4.7} is rewritten in the present situation as follows: For any $A,B\in B(\cH)_+$,
\begin{align}\label{F-7.4}
\phi_f(A,B)=\begin{cases}B^{1/2}f(B^{-1/2}AB^{-1/2})B^{1/2}
&\mbox{(if $B\in B(\cH)_{++}$)}, \\
A^{1/2}\widetilde f(A^{-1/2}BA^{-1/2})A^{1/2}
&\mbox{(if $A\in B(\cH)_{++}$)}.\end{cases}
\end{align}
When $A,B\in B(\cH)_{++}$, the first expression in \eqref{F-7.4} is the definition of the operator
perspective of $A,B$ in \cite{Ef,ENG,EH}, where $f$ is an $\bR$-valued continuous function on
$(0,\infty)$ and $f(B^{-1/2}AB^{-1/2})$ is the continuous functional calculus so that
$\phi_f(A,B)\in B(\cH)_\sa$. In Definition \ref{D-7.1} the definition of operator perspectives is made
from the beginning for general pairs $(A,B)\in B(\cH)_+\times B(\cH)_+$, extending that in
\cite{Ef,ENG,EH}, while the values are in $\widehat{B(\cH)}_\lb$ rather than $B(\cH)_\sa$.

For convenience we restate main results in \S\ref{Sec-4} in the next theorem.

\begin{theorem}\label{T-7.2}
Let $f$, $\widetilde f$ and $\phi_f$ be as stated above.
\begin{itemize}
\item[\rm(1)] Assume that $f(t)<\infty$ for at least two points in $(0,\infty)$. Then the following
conditions are equivalent:
\begin{itemize}
\item[\rm(i)] for all $A_i,B_i\in B(\cH)_+$ ($i=1,2$) with any $\cH$,
\[
\phi_f(A_1+A_2,B_1+B_2)\le\phi_f(A_1,B_1)+\phi_f(A_2,B_2);
\]
\item[\rm(ii)] $f$ is $\bR$-valued and operator convex on $(0,\infty)$;
\item[\rm(iii)] $\widetilde f$ is $\bR$-valued and operator convex on $(0,\infty)$.
\end{itemize}
\item[\rm(2)] Assume that $f\not\equiv\infty$. Then the following conditions are equivalent:
\begin{itemize}
\item[\rm(i$'$)] for every $A_1,A_2,B\in B(\cH)_+$ with any $\cH$,
\[
A_1\le A_2\,\implies\,\phi_f(A_1,B)\ge\phi_f(A_2,B);
\]
\item[\rm(ii$'$)] $f$ is $\bR$-valued and operator monotone decreasing on $(0,\infty)$;
\item[\rm(iii$'$)] $f$ is $\bR$-valued and operator convex, and $\lim_{t\to\infty}f(t)/t\le0$.
\end{itemize}
\end{itemize}
\end{theorem}

\begin{proof}
(1)\enspace
In the present setting, note that $\phi_f(t,1)=f(t)$ and $\phi_f(1,t)=\widetilde f(t)$. If either
(ii) or (iii) holds, then we have
\[
\lim_{t\to\infty}{\phi_f(t,1)\over t}=\alpha=\phi_f(1,0),\qquad
\lim_{t\to\infty}{\phi_f(1,t)\over t}=\beta=\phi_f(0,1)
\]
(see the discussion just above Definition \ref{D-7.1}). Hence (1) is a consequence of Theorem
\ref{T-4.9} (and Theorem \ref{T-3.7}(1)).

(2)\enspace
Assume item (i$'$). By Theorem \ref{T-4.7} (or \eqref{F-7.4}), $f(t)=\phi_f(t,1)$ is operator
monotone decreasing on $(0,\infty)$ and hence operator convex on $(0,\infty)$. Moreover,
$\phi_f(1,0)\le\phi_f(0,0)=0$. Therefore, we have
\[
\lim_{t\nearrow1}\phi_f(t,1-t)=\alpha=\phi_f(1,0)\le0,\qquad
\lim_{t\searrow0}\phi_f(t,1-t)=\beta=\phi_f(0,1)
\]
(see the proof of (iii)$\implies$(v) of Theorem \ref{T-4.9}). These are also satisfied if either
(ii$'$) or (iii$'$) holds, so that the assumption imposed in Theorem \ref{T-4.10} is automatically
satisfied in each case of (i$'$)--(iii$'$). Hence (2) follows from Theorem \ref{T-4.10} (and
Theorem \ref{T-3.7}(2)).
\end{proof}

\begin{remark}\label{R-7.3}\rm
(1)\enspace
Let $f$ be an operator convex function on $(0,\infty)$. Then it is known (see
\cite[Proposition 8.4]{HMPB}) that both $\alpha,\beta$ in \eqref{F-7.1} are finite if and only if
there are $a,b\in\bR$ and an operator monotone function $h\ge0$ on $(0,\infty)$ such that
$f(t)=at+b-h(t)$ for all $t\in(0,\infty)$. In this case, $\phi_f(A,B)$ essentially reduces to the minus
sign of the Kubo--Ando operator connection $A\sigma_hB$; more precisely,
$\phi_f(A,B)=aA+bB-A\sigma_hB$ for all $A,B\in B(\cH)_+$. Therefore, the joint convexity in (i)
of Theorem \ref{T-7.2}(1) reduces to the joint concavity of the operator connection $\sigma_h$.

(2)\enspace
Let $f$ be an operator monotone decreasing function on $(0,\infty)$. If $f$ is negative on
$(0,\infty)$, then $-\phi_f(A,B)$ is the Kubo--Ando operator connection corresponding to $-f$,
so that the inequality in (i$'$) of Theorem \ref{T-7.2}(2) is strengthened to being jointly monotone
decreasing. We note that this occurs only when $f$ is negative on $(0,\infty)$. Indeed, if this is
the case, then by \eqref{F-7.3} and Theorem \ref{T-7.2}(2) both $f$ and $\widetilde f$
are operator monotone decreasing on $(0,\infty)$. Therefore, for all $t\in(0,\infty)$ it follows
that $0\ge\widetilde f'(t)=f(t^{-1})-t^{-1}f'(t^{-1})$ and so $f(t^{-1})\le t^{-1}f'(t^{-1})\le0$. Thus we
have seen that \emph{any jointly operator monotone operations arising as PW-functional calculus
are exactly Kubo and Ando's operator connections}.
\end{remark}

In a conventional approach to Kubo and Ando's theory, the operator
connection $A\sigma B$ defined first for $A,B\in B(\cH)_{++}$ is extended to general
$A,B\in B(\cH)_+$ as
\begin{align}\label{F-7.5}
A\sigma B=\lim_{\eps\searrow0}A_\eps\sigma B_\eps\quad\mbox{in SOT}
\end{align}
based on the upper continuity in \eqref{F-6.1}, where $A_\eps:=A+\eps I$ and similarly for
$B_\eps$. A main aim of this section is to show (Theorem \ref{T-7.7}) that a similar approach is
available for the operator perspective $\phi_f(A,B)$ when $f$ is an operator convex function on
$(0,\infty)$. Before showing this, we will prove (Theorem \ref{T-7.5}) that $\phi_f(A,B)$ enjoys
a joint lower semicontinuity property.

From now on, for the convenience of presentation, we will write $\OC(0,\infty)$ for the set of all
$\bR$-valued operator convex functions on $(0,\infty)$. Furthermore, we set
\[
\OC_0(0,\infty):=\{f\in\OC(0,\infty):f(1)=0\}.
\]

\begin{lemma}\label{L-7.4}
Let $f\in\OC_0(0,\infty)$ and $A,B\in B(\cH)_+$.
\begin{itemize}
\item[\rm(1)] The mapping
\[
X\in B(\cH)_+\,\longmapsto\,\phi_f(A+X,B+X)\in\widehat{B(\cH)}_\lb
\]
is decreasing, that is, for $X,X'\in B(\cH)_+$,
\[
X\ge X'\,\implies\,\phi_f(A+X,B+X)\le\phi_f(A+X',B+X').
\]
\item[\rm(2)] We have
\[
\phi_f(A,B)(\rho)=\sup_{\eps>0}\rho(\phi_f(A_\eps,B_\eps)),\qquad\rho\in B(\cH)_*^+.
\]
\end{itemize}
\end{lemma}

\begin{proof}
(1)\enspace
Let $X,X'\in B(\cH)_+$ and $X\ge X'$. By Theorem \ref{T-7.2}(1) and Lemma \ref{L-4.6}(2)
one has
\begin{align*}
\phi_f(A+X,B+X)
&=\phi_f(A+X'+(X-X'),B+X'+(X-X')) \\
&\le\phi_f(A+X',B+X')+\phi_f(X-X',X-X') \\
&=\phi_f(A+X',B+X')+\phi_f(1,1)(X-X') \\
&=\phi_f(A+X',B+X')
\end{align*}
thanks to $\phi_f(1,1)=f(1)=0$.

(2)\enspace
Item (1) gives
\[
\phi_f(A,B)(\rho)\ge\sup_{\eps>0}\rho(\phi_f(A_\eps,B_\eps)).
\]
The reverse inequality follows from Corollary \ref{C-6.2}.
\end{proof}

The next theorem shows a joint lower semicontinuity property of $\phi_f(A,B)$ for
SOT-converging sequences, a stronger version of Corollary \ref{C-6.2} though in the case of
operator convex perspectives.

\begin{theorem}\label{T-7.5}
Let $f\in\OC(0,\infty)$. If $A,B,A_n,B_n\in B(\cH)_+$ ($n\in\bN$), $A_n\to A$ and $B_n\to B$
in SOT, then
\[
\phi_f(A,B)(\rho)\le\liminf_{n\to\infty}\phi_f(A_n,B_n)(\rho),\qquad\rho\in B(\cH)_*^+.
\]
\end{theorem}

\begin{proof}
Set $f_0(t):=f(t)-f(1)$ for $t\in(0,\infty)$. Then it is obvious that $\phi_f(A,B)=\phi_{f_0}(A,B)+f(1)B$
for all $A,B\in B(\cH)_+$. Hence we may assume that $f\in\OC_0(0,\infty)$. For any fixed
$\eps>0$, since $(A_n)_\eps\to A_\eps$ and $(B_n)_\eps\to B_\eps$ in SOT, it follows that
$(B_n)_\eps^{-1/2}(A_n)_\eps(B_n)_\eps^{-1/2}\to B_\eps^{-1/2}A_\eps B_\eps^{-1/2}$ in SOT.
Moreover, $(B_n)_\eps^{-1/2}(A_n)_\eps(B_n)_\eps^{-1/2}\ge\delta I$ ($n\in\bN$) for some
$\delta>0$, so that we can use \eqref{F-7.4} to confirm that
$\phi_f((A_n)_\eps,(B_n)_\eps)\to\phi_f(A_\eps,B_\eps)$ in SOT as $n\to\infty$. Therefore, the
mapping $(A,B)\in B(\cH)_+\times B(\cH)_+\mapsto\phi_f(A_\eps,B_\eps)\in B(\cH)_\sa$ is
(sequentially) continuous in SOT. From this and Lemma \ref{L-7.4}(2) the result follows.
\end{proof}

\begin{remark}\label{R-7.6}\rm
Let $\sigma$ be any operator connection corresponding to an operator monotone function $h\ge0$
on $(0,\infty)$. When applied to $f=-h$, Theorem \ref{T-7.5} says that if $A,B,A_n,B_n\in B(\cH)_+$
($n\in\bN$), $A_n\to A$ and $B_n\to B$ in SOT, then
\[
\<(A\sigma B)\xi,\xi\>\ge\limsup_{n\to\infty}\<(A_n\sigma B_n)\xi,\xi\>,\qquad\xi\in\cH.
\]
Here note (see \cite[Remark A.2]{HL}) that there are $A_n,B_n\in B(\cH)_{++}$ such that $A_n\to I$,
$B_n\to I$ in SOT and $\<(A_n\sigma B_n)\xi,\xi\>\not\to\<\xi,\xi\>$ for some operator mean
$\sigma$ and some $\xi\in\cH$. This fact also says that Theorem \ref{T-6.1} does not hold for
general SOT-converging sequences $A_n\to A$ and $B_n\to B$.
\end{remark}

We are now in a position to prove the following:

\begin{theorem}\label{T-7.7}
Let $f\in\OC(0,\infty)$ and $A,B\in B(\cH)_+$. If $X_n\in B(\cH)_+$ ($n\in\bN$) and $X_n\to0$ in
SOT, then
\begin{align}\label{F-7.6}
\phi_f(A,B)(\rho)=\lim_{n\to\infty}\phi_f(A+X_n,B+X_n)(\rho),
\qquad\rho\in B(\cH)_*^+.
\end{align}
Furthermore, when $f\in\OC_0(0,\infty)$ and $X_n\searrow0$, the convergence in \eqref{F-7.6}
is increasing.
\end{theorem}

\begin{proof}
As in the proof of Theorem \ref{T-7.5}, we may assume that $f\in\OC_0(0,\infty)$. For every
$\rho\in B(\cH)_*^+$ it follows from Lemma \ref{L-7.4}(1) that
\[
\phi_f(A,B)(\rho)\ge\sup_n\phi_f(A+X_n,B+X_n)(\rho).
\]
On the other hand, Theorem \ref{T-7.5} gives
\[
\phi_f(A,B)(\rho)\le\liminf_{n\to\infty}\phi_f(A+X_n,B+X_n)(\rho).
\]
Hence \eqref{F-7.6} follows, and the latter assertion also follows from Lemma \ref{L-7.4}(1).
\end{proof}

The assertion of Theorem \ref{T-7.7} was dealt with by Fujii and Seo \cite[Theorem 2.8]{FS}
in the case $f(t)=t\log t\in\OC(0,\infty)$; see Example \ref{E-8.12} in \S\ref{Sec-8} for more about
this case. 

In particular, when $X=\eps_nI$ with $\eps_n\searrow0$ in Theorem \ref{T-7.7}, expression
\eqref{F-7.6} becomes
\[
\phi_f(A,B)(\rho)=\lim_{\eps\searrow0}\rho(\phi_f(A_\eps,B_\eps))
=\lim_{\eps\searrow0}\rho(B_\eps^{1/2}f(B_\eps^{-1/2}A_\eps B_\eps^{-1/2})B_\eps^{-1/2})
\]
thanks to Theorem \ref{T-4.7} (or \eqref{F-7.4}). This means that the conventional approach
based on the limit from $A_\eps,B_\eps$ is also available to define the operator perspective
$\phi_f(A,B)$, like \eqref{F-7.5} for operator connections.

\begin{remark}\label{R-7.8}\rm
Unlike the convergence property \eqref{F-6.1} for operator connections, it is not possible to
strengthen \eqref{F-7.6} to the convergence of $\rho(\phi_f(A_n,B_n))$ when $A_n\searrow A$
and $B_n\searrow B$. Consider $A_n=\alpha_n X$, $B_n=\beta_n X$ with $X\in B(\cH)_+$ and
$\alpha_n,\beta_n>0$. Then $\phi_f(A_n,B_n)=\phi_f(\alpha_n,\beta_n)X$ thanks to Lemma
\ref{L-4.6}(2). For instance, if $f$ is an operator convex function $t^p$ for any $p\in[-1,0)\cup(1,2]$,
then $\phi_f(\alpha_n,\beta_n)$ can diverge for some $\alpha_n,\beta_n\searrow0$.
\end{remark}

In the rest of the section we apply an approximation procedure given in
\cite[Lemmas 3.1--3.3]{Hi1} for operator convex functions on $(0,\infty)$. Since the procedure
will be useful in \S\ref{Sec-9} too, we give a brief account on that here. It is known \cite{LR,FHR}
that any $f\in\OC(0,\infty)$ admits an integral expression
\begin{align}\label{F-7.7}
f(t)=a+b(t-1)+c(t-1)^2+d\,{(t-1)^2\over t}
+\int_{(0,\infty)}{(t-1)^2\over t+\lambda}\,d\mu(\lambda)
\end{align}
for all $t\in(0,\infty)$, where $a,b\in\bR$, $c,d\ge0$ and $\mu$ is a positive measure on $(0,\infty)$
with $\int_{(0,\infty)}(1+\lambda)^{-1}\,d\mu(\lambda)<\infty$ (moreover, $a,b,c,d$ and $\mu$ are
uniquely determined from $f$). As easily verified, the values $\alpha=f'(\infty)$ and $\beta=f(0^+)$
are given in terms of the above expression as follows:
\begin{align}\label{F-7.8}
\alpha=b+c\cdot\infty+d+\int_{(0,\infty)}d\mu(\lambda),\quad
\beta=a-b+c+d\cdot\infty+\int_{(0,\infty)}{1\over\lambda}\,d\mu(\lambda).
\end{align}

For each $n\in\bN$ we define
\begin{align}\label{F-7.9}
f_n(t):=a+b(t-1)+nc\,{(t-1)^2\over t+n}+d\,{(t-1)^2\over t+(1/n)}
+\int_{[1/n,n]}{(t-1)^2\over t+\lambda}\,d\mu(\lambda)
\end{align}
for all $t\in(0,\infty)$. Then it is clear that $f_n\in\OC(0,\infty)$ (where $f_n\in\OC_0(0,\infty)$ if
$f\in\OC_0(0,\infty)$) and $f_n(t)\nearrow f(t)$ for all $t\in(0,\infty)$. Furthermore, in view of
\[
{(t-1)^2\over t+\lambda}
=t+{1\over\lambda}-{(1+\lambda)^2\over\lambda}\cdot{t\over t+\lambda},
\qquad t,\lambda\in(0,\infty),
\]
we can rewrite $f_n$ as
\begin{align}\label{F-7.10}
f_n(t)=\alpha_nt+\beta_n-\int_{[1/n,n]}{t(1+\lambda)\over t+\lambda}\,d\nu_n(\lambda),
\qquad t\in(0,\infty),
\end{align}
where
\begin{align}
\alpha_n&:=b+nc+d+\int_{[1/n,n]}d\mu(\lambda)<\infty, \label{F-7.11}\\
\beta_n&:=a-b+c+nd+\int_{[1/n,n]}{1\over\lambda}\,d\mu_n(\lambda)<\infty, \label{F-7.12}\\
d\nu_n(\lambda)&:=(1+n)c\delta_n+(1+n)d\delta_{1/n}
+\chi_{[1/n,n]}{1+\lambda\over\lambda}\,d\mu(\lambda)\label{F-7.13}
\end{align}
with the point masses $\delta_n$ at $n$ and $\delta_{1/n}$ at $1/n$. Define an operator
monotone function $h_n$ on $[0,\infty)$ by
\begin{align}\label{F-7.14}
h_n(t):=\int_{[1/n,n]}{t(1+\lambda)\over t+\lambda}\,d\nu_n(\lambda),
\qquad t\in(0,\infty),
\end{align}
and consider the corresponding operator connection $\sigma_{h_n}$. Then we have the following:

\begin{lemma}\label{L-7.9}
Let $f\in\OC(0,\infty)$ and define $f_n$, $\alpha_n$, $\beta_n$ and $h_n$ ($n\in\bN$) by
\eqref{F-7.9}--\eqref{F-7.14}. Then for every $A,B\in B(\cH)_+$ we have
\[
\phi_{f_n}(A,B)=\alpha_nA+\beta_nB-B\sigma_{h_n}A\ \ (\in B(\cH)_\sa)
\]
and
\[
\phi_f(A,B)(\rho)=\lim_{n\to\infty}\rho(\phi_{f_n}(A,B))\ \ \mbox{increasingly},
\quad\rho\in B(\cH)_*^+.
\]
Moreover, if $f\in\OC_0(0,\infty)$, then $f_n$'s are in $\OC_0(0,\infty)$.
\end{lemma}

The following is an application of Lemma \ref{L-7.9}. Let $\Phi:B(\cH)\to B(\cK)$ be a positive
linear map (hence bounded automatically), where $\cK$ is another Hilbert space. Assume that
$\Phi$ is normal in the sense that if $\{A_i\}$ is a net in $B(\cH)_+$ and
$A_i\nearrow A\in B(\cH)_+$, then $\Phi(A_i)\nearrow\Phi(A)$; in other words, $\Phi$ is
continuous with respect to the $\sigma$-weak topologies on $B(\cH)$ and $B(\cK)$. Then we have
the predual map $\Phi_*:B(\cK)_*^+\to B(\cH)_*^+$ of $\Phi$, which is a positive linear map such
that $(\Phi_*(\rho))(X)=\rho(\Phi(X))$ for all $\rho\in B(\cK)_*^+$ and $X\in B(\cH)$. For every
$T\in\widehat{B(\cH)}_\lb$ define
\[
\Phi(T)(\rho):=T(\Phi_*(\rho))=T(\rho\circ\Phi),\qquad\rho\in B(\cK)_*^+.
\]
Then it is easy to confirm that $\Phi(T)\in\widehat{B(\cK)}_\lb$, $\Phi(\alpha T)=\alpha\Phi(T)$
and $\Phi(T_1+T_2)=\Phi(T_1)+\Phi(T_2)$ for all $\alpha\ge0$ and
$T,T_1,T_2\in\widehat{B(\cH)}_\lb$. Moreover, the map
$\Phi:\widehat{B(\cH)}_\lb\to\widehat{B(\cK)}_\lb$ is an extension of $\Phi:B(\cH)_+\to B(\cK)_+$.

The next proposition is an extended version of (ii) and (iii) of Theorem \ref{T-4.9} and also extends
\cite[Theorem 6.7]{HSW}.

\begin{proposition}\label{P-7.10}
Let $f\in\OC(0,\infty)$ and $\Phi$ be as stated above. Then for every $A,B\in B(\cH)_+$ we have
\[
\phi_f(\Phi(A),\Phi(B))\le\Phi(\phi_f(A,B))\quad\mbox{in $\widehat{B(\cK)}_\lb$}.
\]
\end{proposition}

\begin{proof}
As before it suffices to assume that $f\in\OC_0(0,\infty)$. For every $A,B\in B(\cH)_+$ and
$\rho\in B(\cH)_*^+$, by Lemma \ref{L-7.9} we have
\[
\phi_f(\Phi(A),\Phi(B))(\rho)
=\sup_n\rho\bigl(\alpha_n\Phi(A)+\beta_n\Phi(B)-\Phi(B)\sigma_{h_n}\Phi(A)\bigr).
\]
It is well-known that
\begin{align}\label{F-7.15}
\Phi(B\sigma_{h_n}A)\le\Phi(B)\sigma_{h_n}\Phi(A),
\end{align}
which is due to Ando \cite{An1} (though stated in \cite{An1} only for geometric mean and parallel
sum). Therefore, we have
\begin{align*}
&\rho\bigl(\alpha_n\Phi(A)+\beta_n\Phi(B)-\Phi(B)\sigma_{h_n}\Phi(A)\bigr) \\
&\qquad\le\rho\bigl(\Phi(\alpha_nA+\beta_nB-B\sigma_{h_n}A)\bigr)
=(\Phi_*(\rho))(\phi_{f_n}(A,B)) \\
&\qquad\le(\phi_f(A,B))(\Phi_*(\rho))=\Phi(\phi_f(A,B))(\rho),
\end{align*}
showing the result.
\end{proof}

Inequality \eqref{F-7.15} holds for a general (not necessarily normal) positive linear map $\Phi$;
however the normality of $\Phi$ is necessary to define $\Phi(T)$ for $T\in\widehat{B(\cH)}_\lb$.

\begin{corollary}\label{C-7.11}\rm
Let $f\in\OC(0,\infty)$ and $A,B\in B(\cH)_+$. Then we have
\begin{align}\label{F-7.16}
\phi_f(\rho(A),\rho(B))\le\phi_f(A,B)(\rho),\qquad\rho\in B(\cH)_*^+,
\end{align}
and
\begin{equation}\label{F-7.17}
\begin{aligned}
&\sup_{\xi\in\cH,\|\xi\|=1}\phi_f(\<A\xi,\xi\>,\<B\xi,\xi\>)
=\sup_{\rho\in B(\cH)_*^+,\,\rho(I)=1}\phi_f(\rho(A),\rho(B)) \\
&\qquad\le\sup_{\xi\in\cH,\|\xi\|=1}\phi_f(A,B)(\omega_\xi)
=\sup_{\rho\in B(\cH)_*^+,\,\rho(I)=1}\phi_f(A,B)(\rho).
\end{aligned}
\end{equation}
Furthermore, if $A,B$ are commuting, then the inequality in \eqref{F-7.17} becomes equality.
\end{corollary}

\begin{proof}
Inequality \eqref{F-7.16} is immediate by Proposition \ref{P-7.10} applied to $\Phi:B(\cH)\to\bC$
sending $X\in B(\cH)$ to $\rho(X)$. Note that each $\rho\in B(\cH)_*^+$ with $\rho(I)=1$ is given
as $\rho=\sum_n\omega_{\xi_n}$ for some $\{\xi_n\}$ in $\cH$ with $\sum_n\|\xi_n\|^2=1$. By
Theorems \ref{T-7.5} and \ref{T-7.2}(1) we have
\[
\phi_f(\rho(A),\rho(B))
\le\liminf_{k\to\infty}\sum_{n=1}^k\phi_f(\<A\xi_n,\xi_n\>,\<B\xi_n,\xi_n\>)
\le\sup_{\xi\in\cH,\|\xi\|=1}\phi_f(\<A\xi,\xi\>,\<B\xi,\xi\>),
\]
which yields the first equality in \eqref{F-7.17}. The latter equality in \eqref{F-7.17} is
immediately seen, and the inequality is obvious from \eqref{F-7.16}.

To prove the last assertion, let $\ell$ and $r$ be the left-hand and the right-hand sides,
respectively, of \eqref{F-7.17}. Assume first that $(A,B)$ is of the form $A=\sum_{i=1}^ma_iP_i$
and $B=\sum_{i=1}^mb_iQ_i$, where $P_i,Q_i$ ($1\le i\le m$) are projections with
$\sum_iP_i=\sum_iQ_i=I$ and $P_iQ_j=Q_jP_i$ for all $i,j$. 
Then it follows from Definition \ref{D-4.1}(1) that, for every $\rho\in B(\cH)_*^+$ with $\rho(I)=1$,
\[
\phi_f(A,B)(\rho)=\sum_{i,j}\phi_f(a_i,b_j)\rho(P_iQ_j)
\le\max_{P_iQ_j\ne0}\phi_f(a_i,b_j)\le\ell.
\]
Hence $\ell=r$ holds when $(A,B)$ is of this specified form. For a general commuting pair
$(A,B)$ and any $\rho\in B(\cH)_*^+$, by approximation of the spectral decompositions of
$A,B$ like Riemann sums, one can choose a sequence of pairs $(A_n,B_n)$ as specified above
such that $\rho(A_n)=\rho(A)$, $\rho(B_n)=\rho(B)$ and $A_n\to A$, $B_n\to B$ in the operator
norm. From Theorem \ref{T-7.5} and the above case it follows that
\[
\phi_f(A,B)(\rho)\le\liminf_{n\to\infty}\phi_f(A_n,B_n)(\rho)\le\ell,
\]
which implies that $\ell=r$ holds for all commuting pairs $(A,B)$.  
\end{proof}

It is worth noting that inequalities \eqref{F-7.16} and \eqref{F-7.17} are conceptually similar
to the \emph{Peierls--Bogolieubov inequality} for quantum divergences (in particular, relative
entropy); see, e.g., \cite{Uh,Pe1,Pe2} (also \cite{Hi3}) for more details on this and related matters.

\begin{example}\label{E-7.12}\rm
We remark that a strict inequality occurs in \eqref{F-7.17} for non-commuting $A,B$. For example,
let $A=\begin{bmatrix}1&1\\1&1\end{bmatrix}$ and $B=\begin{bmatrix}1&0\\0&2\end{bmatrix}$ in
$B(\bC^2)_+$. Let $1<\alpha\le2$. Then a simple computation gives
$\phi_{t^\alpha}(A,B)=B^{1/2}(B^{-1/2}AB^{-1/2})^\alpha B^{1/2}=(3/2)^{\alpha-1}A$, so that the
right-hand side of \eqref{F-7.17} is $\|\phi_{t^\alpha}(A,B)\|=2(3/2)^{\alpha-1}$. On the other hand,
the left-hand side of \eqref{F-7.17} is equal to
\[
\max\biggl\{{(x_1+x_2)^{2\alpha}\over(x_1^2+2x_2^2)^{\alpha-1}}
=(x_1+x_2)^2\biggl({(x_1+x_2)^2\over x_1^2+2x_2^2}\biggr)^{\alpha-1}:
x_1,x_2\ge0,\,x_1^2+x_2^2=1\biggr\}.
\]
Note that for every $x_1,x_2\ge0$ with $x_1^2+x_2^2=1$, $(x_1+x_2)^2\le2$ with equality only
when $(x_1,x_2)=(1/\sqrt2,1/\sqrt2)$, and $(x_1+x_2)^2\le(3/2)(x_1^2+2x_2^2)$ with equality
only when $(x_1,x_2)=(2/\sqrt5,1/\sqrt5)$. Therefore, the above maximum is strictly less than
$2(3/2)^{\alpha-1}$.
\end{example}

We end the section with clarifying the relation between the extended operator perspective
$\phi_f(A,B)$ and the maximal $f$-divergence $\widehat S_f(A\|B)$ for $A,B\in\cC_1(\cH)_+$
(identified with $B(\cH)_*^+$), i.e., positive trace-class operators on $\cH$.    
\emph{Maximal $f$-divergences} are a type of quantum $f$-divergences, whose details are
found in, e.g., \cite{HM,Hi2}.

We set
\[
\widehat{B(\cH)}_++\cC_1(\cH)_\sa:=\{T+X:T\in\widehat{B(\cH)}_+,\,X\in\cC_1(\cH)_\sa\},
\]
which is a sub-cone of $\widehat{B(\cH)}_\lb=\widehat{B(\cH)}_++B(\cH)_\sa$. The trace $\Tr$ on
$B(\cH)_+$ naturally extends to $\widehat{B(\cH)}_+$ due to \cite[Proposition 1.10]{Ha}. We can
further extend $\Tr$ to $\widehat{B(\cH)}_++\cC_1(\cH)_\sa$ by $\Tr(T+X):=\Tr\,T+\Tr\,X$
($\in(-\infty,\infty]$), and it is easy to confirm that $\Tr$ is positive homogeneous, additive and
normal (i.e., $T_i\nearrow T$$\implies$$\Tr\,T_i\nearrow\Tr\,T$) on
$\widehat{B(\cH)}_++\cC_1(\cH)_\sa$.

\begin{proposition}\label{P-7.13}
Let $f\in\OC(0,\infty)$ and $A,B\in\cC_1(\cH)_+$. Then
$\phi_f(A,B)\in\widehat{B(\cH)}_++\cC_1(\cH)_\sa$ and
\begin{align}\label{F-7.18}
\Tr\,\phi_f(A,B)=\widehat S_f(A\|B),
\end{align}
where $\widehat S_f(A\|B)$ is the maximal $f$-divergence of $A,B$ defined in \cite{Hi2}.
\end{proposition}

\begin{proof}
Set $f_0(t):=f(t)-(at+b)$ where $a:=f'(1)$ and $b:=f(1)-f'(1)$; then
$\phi_f(A,B)=\phi_{f_0}(A,B)+(aA+bB)$. Since $\phi_{f_0}(A,B)\in\widehat{B(\cH)}_+$ thanks to
$f_0\ge0$, we have $\phi_f(A,B)\in\widehat{B(\cH)}_++\cC_1(\cH)_\sa$, so that $\Tr\,\phi_f(A,B)$
is well defined as seen before the proposition.

Let us show equality \eqref{F-7.18}. Assume first that $\lambda^{-1}B\le A\le\lambda B$ for some
$\lambda>0$, and let $P$ be the support projection of $B$. Then $A=B^{1/2}WB^{1/2}$ for a
(unique) invertible $W\in B(P\cH)_+$. By operator homogeneity (Definition \ref{D-2.1}(2)) and
Proposition \ref{P-4.5} we have
\[
\phi_f(A,B)=B^{1/2}\phi_f(W,P)B^{1/2}=B^{1/2}f(W)B^{1/2},
\]
where $f(W)\in B(P\cH)_\sa$ is the continuous functional calculus of $W$. Hence
$\phi_f(A,B)\in\cC_1(\cH)_\sa$ and $\Tr\,\phi_f(A,B)=\Tr\,B^{1/2}f(W)B^{1/2}=\Tr\,Bf(W)$; see
\cite[Lemma 3.4.11]{Ped} for the last equality. The last term of this is exactly $\widehat S_f(A\|B)$
in \cite[Definition 2.3]{Hi2}. For general $A,B\in\cC_1(\cH)_+$ with $C:=A+B$ the case shown just
above gives $\Tr\,\phi_f(A+\eps C,B+\eps C)=\widehat S_f(A+\eps C\|B+\eps C)$ for any $\eps>0$.
Furthermore, with $f_0$ as above we have
\[
\Tr\,\phi_f(A+\eps C,B+\eps C)
=\Tr\,\phi_{f_0}(A+\eps C,B+\eps C)+\Tr(aA+bB)+\eps(a+b)\Tr\,C,
\]
which converges to $\Tr\,\phi_f(A,B)$ as $\eps\searrow0$ by Theorem \ref{T-7.7} and the normality
of $\Tr$ (mentioned above). Hence equality \eqref{F-7.18} holds thanks to \cite[Lemma 2.6]{Hi2}.
(Alternatively, \eqref{F-7.18} can be shown by using \cite[Theorem 4.2]{Hi2} which is more directly
related to the PW-functional calculus.)
\end{proof}

Proposition \ref{P-7.13} gives a justification for our formulation of extended operator perspectives.
Basic properties of maximal $f$-divergences given in \cite{Hi2} are also derived via \eqref{F-7.18}
from those of $\phi_f(A,B)$ shown in this section (though in the $B(\cH)$ setting). For example,
when a normal positive map $\Phi:B(\cH)\to B(\cK)$ is trace-preserving, the inequality of
Proposition \ref{P-7.10} yields the monotonicity inequality (called the
\emph{data-processing inequality}) for the maximal $f$-divergence shown in
\cite[Theorem 2.9]{Hi2}.

\section{Dense domain case and boundedness}\label{Sec-8}

For $f\in\OC(0,\infty)$ we will further discuss the operator perspective $\phi_f(A,B)$ introduced in
Definition \ref{D-7.1} via \eqref{F-7.2} with $\alpha=f'(\infty)$ and $\beta=f(0^+)$. Throughout this
section, for any $A,B\in B(\cH)_+$ we use the notations $\cH_{A,B}$, $T_{A,B}$, $R_{A,B}$ and
$S_{A,B}$ given in \S\ref{Sec-4}, and let $R_{A,B}=\int_0^1t\,dE_{R_{A,B}}(t)$ be the spectral
decomposition.

Our questions in this section are when $\phi_f(A,B)$ is bounded and when $\phi_f(A,B)$ has
a dense domain (i.e., the $\infty$-part is trivial); see Definition \ref{D-2.3}. First we consider the
latter question. When both $\alpha$ and $\beta$ are finite, $\phi_f(A,B)$ is bounded for all
$A,B\in B(\cH)_+$; see Remark \ref{R-7.3}(1). Now assume that $\alpha=\infty>\beta$. For every
$\xi\in\cH$, by \eqref{F-4.5} and Lemma \ref{L-4.3}(3), we find that
\begin{equation}\label{F-8.1}
\begin{aligned}
\phi_f(A,B)(\omega_\xi)&=\phi_f(R,S)(T_{A,B}\omega_\xi T_{A,B}^*) \\
&=\biggl(\int_{[0,1]}\phi_f(t,1-t)\,dE_{R_{A,B}}(t)\biggr)(\omega_{T_{A,B}\xi}) \\
&=\int_{[0,1)}\phi_f(t,1-t)\,d\|E_{R_{A,B}}(t)T_{A,B}\xi\|^2
+\infty\cdot\|E_{R_{A,B}}(\{1\})T_{A,B}\xi\|^2.
\end{aligned}
\end{equation}
Write $Q_{A,B}:=\int_{[0,1)}(\phi_f(t,1-t)\vee0)\,dE_{R_{A,B}}(t)$, which is a positive self-adjoint
operator on $\cH_{A,B}$. Then it follows from \eqref{F-8.1} that
\[
\{\xi\in\cH:\phi_f(A,B)(\omega_\xi)<\infty\}
=\ker(E_{R_{A,B}}(\{1\})T_{A,B})\cap\cD(Q_{A,B}^{1/2}T_{A,B}).
\]
Therefore, the essential part  $\cH_0$ (see Proposition \ref{P-2.2}) of $\phi_f(A,B)$ is
\begin{align}\label{F-8.2}
\cH_0=\overline{\ker(E_{R_{A,B}}(\{1\})T_{A,B})\cap\cD(Q_{A,B}^{1/2}T_{A,B})}
\subseteq\ker(E_{R_{A,B}}(\{1\})T_{A,B}),
\end{align}
and $\cH_0=\cH$ (i.e., $\phi_f(A,B)$ has a dense domain) if and only if $E_{R_{A,B}}(\{1\})T_{A,B}=0$
and $\cD(Q_{A,B}^{1/2}T_{A,B})$ is dense in $\cH$. Since $\overline\ran\,T_{A,B}=\cH_{A,B}$,
$E_{R_{A,B}}(\{1\})T_{A,B}=0$ is equivalent to $E_{R_{A,B}}(\{1\})=0$ or  $\ker S_{A,B}=\{0\}$.
This is also equivalent to $\ker R_{A,B}\supseteq\ker S_{A,B}$ thanks to
$R_{A,B} + S_{A,B} = I_{\mathcal{H}_{A,B}}$.
Note that $\ker R_{A,B}\supseteq\ker S_{A,B}$ implies $\ker A\supseteq\ker B$ but the converse
is not true; see Lemma \ref{P-8.4} below. From the argument so far we have item (1) in the following
proposition. Items (2) and (3) are similarly shown ((2) is also clear by applying (1) to $\widetilde f$).

\begin{proposition}\label{P-8.1}
Let $f\in\OC(0,\infty)$ and $A,B\in B(\cH)_+$.
\begin{itemize}
\item[\rm(1)] Assume that $f'(\infty)=\infty>f(0^+)$, and set
$Q_{A,B}:=\int_{[0,1)}(\phi_f(t,1-t)\vee0)\,dE_{R_{A,B}}(t)$ on $\cH_{A,B}$. Then $\phi_f(A,B)$
has a dense domain if and only if $\ker S_{A,B}=\{0\}$ and $\cD(Q_{A,B}^{1/2}T_{A,B})$ is
dense in $\cH$.
\item[\rm(2)] Assume that $f'(\infty)<\infty=f(0^+)$, and set
$Q_{A,B}:=\int_{(0,1]}(\phi_f(t,1-t)\vee0)\,dE_{R_{A,B}}(t)$ on $\cH_{A,B}$. Then $\phi_f(A,B)$
has a dense domain if and only if $\ker R_{A,B}=\{0\}$ and $\cD(Q_{A,B}^{1/2}T_{A,B})$ is
dense in $\cH$.
\item[\rm(3)] Assume that $f'(\infty)=f(0^+)=\infty$, and set
$Q_{A,B}:=\int_{(0,1)}(\phi_f(t,1-t)\vee0)\,dE_{R_{A,B}}(t)$ on $\cH_{A,B}$. Then $\phi_f(A,B)$
has a dense domain if and only if $\ker R_{A,B}=\ker S_{A,B}=\{0\}$ and
$\cD(Q_{A,B}^{1/2}T_{A,B})$ is dense in $\cH$.
\end{itemize}
\end{proposition}

\begin{remark}\label{R-8.2}\rm
In particular, when $\cH$ is finite-dimensional, the situation in Proposition \ref{P-8.1} is much
simpler. Note that any densely-defined operator is bounded and
$\ker A=\ker R_{A,B}\oplus\cH_{A,B}^\perp$ in this case. Hence, for instance, item (1) simply
says that if $f'(\infty)=\infty>f(0^+)$, then $\phi_f(A,B)$ is bounded if and only if
$\ker A\supseteq\ker B$ (or $(A,B)\in(B(\cH)_+\times B(\cH)_+)_\le$); (2) and (3) are similar.
\end{remark}

Concerning the question on the boundedness of $\phi_f(A,B)$, one can state, similarly to
Proposition \ref{P-8.1}(1) for example, that if $f'(\infty)=\infty>f(0^+)$, then $\phi_f(A,B)$ is
bounded if and only if $\ker S_{A,B}=\{0\}$ and $Q_{A,B}^{1/2}T_{A,B}$ is bounded. But this
seems just a restatement of $\phi_f(A,B)$ being bounded, so a more intrinsic condition is desirable.
Although the problem seems difficult for general $f\in\OC(0,\infty)$, the special case of $f(t)=t^2$
is tractable as discussed below. This function is an extreme case of operator convex power
functions $t^p$ ($p\in[-1,0]\cup[1,2]$).

As for the function $t^2$, we begin by setting
\begin{align}\label{F-8.3}
g^{(n)}(t):={n(t-1)^2\over t+n}
=nt+1-(1+n)^2\,{t\over t+n},\qquad t\in(0,\infty),
\end{align}
for each $n\in\bN$. Obviously, $g^{(n)}(t)\nearrow(t-1)^2$ for all $t\in(0,\infty)$.  For any
$A,B\in B(\cH)_+$ note that
\begin{equation}\label{F-8.4}
\begin{aligned}
\phi_{g^{(n)}}(A,B)&=nA+B-(1+n)^2\phi_{t/(t+n)}(A,B) \\
&=nA+B-{(1+n)^2\over n}(A:nB),
\end{aligned}
\end{equation}
where $A:B$ denote the \emph{parallel sum} of $A,B\in B(\cH)_+$ (see \cite{AD,An1}). Moreover,
recall a well-known formula
\begin{align}\label{F-8.5}
A-(A:B)=A(A+B)^{-1}A,\qquad A,B\in B(\cH)_+,
\end{align}
which is more precisely understood as
\[
A-(A:B)=\lim_{\eps\searrow0}A(A+B+\eps I)^{-1}A\quad\mbox{in SOT}.
\]
(This formula is easy to see. Also, the original definition of parallel sum \cite{AD} is
$A:B=A(A+B)^{-1}B$ for matrices, where $(A+B)^{-1}$ is the generalized inverse; see also
\cite[p.~103]{Bh2}.)

\begin{proposition}\label{P-8.3}
For every $A,B\in B(\cH)_+$ and $\rho\in B(\cH)_*^+$, we have
\begin{align}\label{F-8.6}
\phi_{t^2}(A,B)(\rho)=\lim_{n\to\infty}n\rho(A-(A:nB))
=\lim_{\eps\searrow0}\rho(A(\eps A+B)^{-1}A)\ \ \mbox{increasingly}.
\end{align}
\end{proposition}

\begin{proof}
Let $A,B\in B(\cH)_+$ and $\rho\in B(\cH)_*^+$. Since
$\phi_{g^{(n)}}(A,B)\nearrow \phi_{(t-1)^2}(A,B)$, it follows from \eqref{F-8.4} that
\begin{align*}
\phi_{t^2}(A,B)(\rho)&=\phi_{(t-1)^2}(A,B)(\rho)+\phi_{2t-1}(A,B)(\rho) \\
&=\lim_{n\to\infty}\rho\biggl(nA+B-{(1+n)^2\over n}(A:nB)\biggr)+\rho(2A-B) \\
&=\lim_{n\to\infty}\rho\biggl({n+2\over n}\,n(A-(A:nB))-{1\over n}(A:nB)\biggr) \\
&=\lim_{n\to\infty}n\rho(A-(A:nB)).
\end{align*}
The above last equality is immediate since ${1\over n}(A:nB)\le{1\over n}A\to0$. Furthermore,
by \eqref{F-8.5} we find that
\[
n(A-(A:nB))=nA(A+nB)^{-1}A=A(n^{-1}A+B)^{-1}A,
\]
showing the second equality and the limits being increasing.
\end{proof}

Here we recall Ando's work \cite{An2} on Lebesgue decomposition of positive operators. For
$B,X\in B(\cH)_+$ it is said that $X$ is \emph{$B$-absolutely continuous} if there is a sequence
$\{X_n\}$ in $B(\cH)_+$ such that $X_n\nearrow X$ and $X_n\le\lambda_n B$ for some
$\lambda_n\ge0$. Also, $X$ is said to be \emph{$B$-singular} if $Y\in B(\cH)_+$ satisfies
$0\le Y\le B$ and $0\le Y\le X$, then $Y=0$. For every $A,B\in B(\cH)_+$ Ando \cite{An2}
introduced a \emph{$B$-absolutely continuous part} $[B]A$ of $A$ by
\[
[B]A:=\lim_{n\to\infty}(A:nB)\quad\mbox{increasingly},
\]
which is the maximum of all $B$-absolutely continuous $X\in B(\cH)_+$ with $X\le A$. Then
$A-[B]A$ is $B$-singular and we have a \emph{$B$-Lebesgue decomposition} \cite{An2}
\[
A=[B]A+(A-[B]A).
\]

\begin{proposition}\label{P-8.4}
For every $A,B\in B(\cH)_+$ we have
\begin{align}\label{F-8.7}
A-[B]A=T_{A,B}^*E_{R_{A,B}}(\{1\})T_{A,B}.
\end{align}
Hence $\ker S_{A,B}=\{0\}$ if and only if $A$ is $B$-absolutely continuous.
\end{proposition}

\begin{proof}
We observe that
\begin{align*}
A-(A:nB)&=T_{A,B}^*(R_{A,B}-(R_{A,B}:n(I_{\cH_{A,B}}-R_{A,B})))T_{A,B} \\
&=T_{A,B}^*\biggl({R_{A,B}^2\over n(I_{\cH_{A,B}}-R_{A,B})+R_{A,B}}\biggr)T_{A.B} \\
&=T_{A,B}^*\biggl(\int_0^1{t^2\over n(1-t)+t}\,dE_{R_{A,B}}(t)\biggr)T_{A,B},
\end{align*}
where the first equality is due to operator homogeneity \cite{Fu2} (see also Definition
\ref{D-4.1}(2)) applied to parallel sum. Letting $n\to\infty$ gives \eqref{F-8.7}. From \eqref{F-8.7}
and the argument above Proposition \ref{P-8.1}, it follows that $\ker S_{A,B}=\{0\}$ if and only if
$A-[B]A=0$, that is, $A$ is $B$-absolutely continuous.
\end{proof}

Proposition \ref{P-8.4} shows that the (maximal) absolutely continuous part $[B]A$ is expressed
as
\[
[B]A=A-T_{A,B}^*E_{R_{A,B}}(\{1\})T_{A,B}
=T_{A,B}^* R_{A,B}E_{R_{A,B}}([0,1))T_{A,B},
\]
which is somewhat similar to the formula given in \cite{Ko1}. 

By Propositions \ref{P-8.1} and \ref{P-8.4}  we have the following necessary condition for
$\phi_f(A,B)$ to be bounded. 

\begin{corollary}\label{C-8.5}
Let $f\in\OC(0,\infty)$ and $A,B\in B(\cH)_+$. Assume that $\phi_f(A,B)$ is bounded. Then
\begin{itemize}
\item[\rm(1)] $A$ is $B$-absolutely continuous if $f'(\infty)=\infty>f(0^+)$,
\item[\rm(2)] $B$ is $A$-absolutely continuous if $f'(\infty)<\infty=f(0^+)$,
\item[\rm(3)] $A,B$ are mutually absolutely continuous if $f'(\infty)=f(0^+)=\infty$.
\end{itemize}
\end{corollary}

The next theorem gives characterizations for $\phi_{t^2}(A,B)$ to have a dense domain and to be
bounded. The same descriptions hold for $\phi_{t^{-1}}(A,B)$ with the roles of $A,B$ exchanged.

\begin{theorem}\label{T-8.6}
For every $A,B\in B(\cH)_+$ set $Q_{A,B}:=\int_{[0,1)}t^2/(1-t)\,dE_{R_{A,B}}(t)$. Then the
following hold:
\begin{itemize}
\item[\rm(1)] The essential part $\cH_0$ of $\phi_{t^2}(A,B)$ is
\begin{align}\label{F-8.8}
\cH_0=\overline{\ker(A-[B]A)\cap\cD(Q_{A,B}^{1/2}T_{A,B})}\subseteq\ker(A-[B]A).
\end{align}
\item[\rm(2)] $\phi_{t^2}(A,B)$ has a dense domain if and only if $A$ is $B$-absolutely
continuous and $\cD(Q_{A,B}^{1/2}T_{A,B})$ is dense in $\cH$.
\item[\rm(3)] $\phi_{t^2}(A,B)$ is bounded if and only if $A^2\le\lambda B$ for some
$\lambda>0$ (which is strictly weaker than $(A,B)\in(B(\cH)_+\times B(\cH)_+)_\le$).
In this case, the operator norm of $\phi_{t^2}(A,B)$ is
\[
\|\phi_{t^2}(A,B)\|=\min\{\lambda\ge0:A^2\le\lambda B\}.
\]
\end{itemize}
\end{theorem}

\begin{proof}
(1) immediately follows from \eqref{F-8.2} and \eqref{F-8.7}.

(2)\enspace
By Proposition \ref{P-8.4} this is a restatement of Proposition \ref{P-8.1}(1) for $f(t)=t^2$.

(3)\enspace
From \eqref{F-8.6}, for any $\lambda\ge0$ we find that $\phi_{t^2}(A,B)$ is bounded with
$\|\phi_{t^2}(A,B)\|\le\lambda$ if and only if $A(\eps A+B+\delta I)^{-1}A\le\lambda I$ holds for
all $\eps,\delta>0$. Since $A(\eps A+B+\delta I)^{-1}A\le\lambda I$ if and only if
$A^2\le\lambda(\eps A+B+\delta I)$, the condition is equivalent to $A^2\le\lambda B$. Hence
the result follows.
\end{proof}

\begin{example}\label{E-8.7}\rm
Consider two projections $P,Q\in B(\cH)$. Since $P:(nQ)={n\over1+n}(P\wedge Q)$ (see the
proof of \cite[Theorem 3.7]{KA}), we have, for every $\rho\in B(\cH)_*^+$,
\begin{align*}
\phi_{t^2}(P,Q)(\rho)&=\lim_{n\to\infty}n\rho\Bigl(P-{n\over1+n}(P\wedge Q)\Bigr)
\quad\mbox{(by \eqref{F-8.6})}\\
&=\lim_{n\to\infty}n\rho\Bigl(P-P\wedge Q+{1\over 1+n}(P\wedge Q)\Bigr) \\
&=\rho(P\wedge Q)+\infty\cdot\rho(P-P\wedge Q)
\end{align*}
and 
\[
P-[Q]P=\lim_{n\to\infty}\Bigl(P-{n\over1+n}(P\wedge Q)\Bigr)=P-P\wedge Q.
\]
Hence the essential part of $\phi_{t^2}(P,Q)$ is $(P-P\wedge Q)^\perp\cH=\ker(P-[Q]P)$.
\end{example}

\begin{remark}\label{R-8.8}\rm
The essential part of $\phi_{t^2}(A,B)$ is equal to $\ker({A-[B]A})$, for example, in the
finite-dimensional case (see Remark \ref{R-8.2}) and in the two projection case (Example
\ref{E-8.7}). It is also easy to verify that this is the case when $A,B$ commute. However, this is
not true in general. Here, we exemplify that the presence of $\cD(Q_{A,B}^{1/2}T_{A,B})$ in
\eqref{F-8.8} can make the essential part even trivial while $A$ is $B$-absolutely continuous.
Let $T$ be any non-singular bounded positive operator, and $Q$ be any non-singular positive
self-adjoint operator with the spectral decomposition $Q=\int_0^\infty\lambda\,dF_\lambda$. With
a strictly increasing function $w:[0,1)\to[0,\infty)$ given by
\[
w(t)={t^2\over1-t}\quad(0\le t<1),\qquad
w^{-1}(\lambda)={-\lambda+\sqrt{\lambda^2+4\lambda}\over2}\quad(0\le\lambda<\infty),
\]
we define
\begin{align}\label{F-8.9}
R:=\int_0^\infty w^{-1}(\lambda)\,dF_\lambda=\int_{[0,1)}t\,dE_t,
\end{align}
where $E_t:=F_{w(t)}$. Then $0\le R\le I$, and we further define
\begin{align}\label{F-8.10}
A:=TRT,\qquad B:=T(I-R)T.
\end{align}
Then we have $T=(A+B)^{1/2}=T_{A,B}$ (where $\cH_{A,B}=\cH$), $R=R_{A,B}$,
$E=E_{R_{A,B}}$ and
\[
Q_{A,B}=\int_{[0,1)}w(t)\,dE_t=\int_0^\infty\lambda\,dF_\lambda=Q,
\]
so that $T$ and $Q$ are realized as $T_{A,B}$ and $Q_{A,B}$, respectively, in Theorem
\ref{T-8.6} (for $A,B$ defined by \eqref{F-8.10}). Furthermore, since $E_{R_{A,B}}(\{1\})=0$ for
the spectral measure of $R=R_{A,B}$ thanks to \eqref{F-8.9}, note by \eqref{F-8.7} that
$A=[B]A$ in this case. From a classical result of von Neumann (whose readable account is
found in \cite{FW}), there are non-singular positive self-adjoint operators $K,L$ with bounded
inverses such that $\cD(K)\cap\cD(L)=\{0\}$. Now consider the above construction with $T:=K^{-1}$
and $Q:=L^2$. Then
\begin{align*}
\cD(Q_{A,B}^{1/2}(A+B)^{1/2})&=\cD(LK^{-1})=\{\xi\in\cH:K^{-1}\xi\in\cD(L)\} \\
&=\{K\eta:\eta\in\cD(K)\cap\cD(L)\}=\{0\},
\end{align*}
which implies by Theorem \ref{T-8.6}(1) that the essential part of $\phi_{t^2}(A,B)$ is $\{0\}$.
Thus, we arrive at the extremely pathological situation that $\phi_{t^2}(A,B)$ is identically
$\infty$ and $A$ is $B$-absolutely continuous.
\end{remark}

Based on the integral expression given in \eqref{F-7.7} and Theorem \ref{T-8.6}(3), we can
show the boundedness of $\phi_f(A,B)$ in a more general situation.

\begin{proposition}\label{P-8.9}
Let $f\in\OC(0,\infty)$ and $A,B\in B(\cH)_+$.
\begin{itemize}
\item[\rm(1)] Assume that $f(0^+)<\infty$. If $A^2\le\lambda B$ for some $\lambda>0$, then
$\phi_f(A,B)$ is bounded. In addition, assume that $c>0$ in \eqref{F-7.7} (hence
$f'(\infty)=\infty$). Then $\phi_f(A,B)$ is bounded if and only if $A^2\le\lambda B$ for some
$\lambda>0$.
\item[\rm(2)] Assume that $f'(\infty)<\infty$. If $B^2\le\lambda A$ for some $\lambda>0$, then
$\phi_f(A,B)$ is bounded. In addition, assume that $d>0$ in \eqref{F-7.7} (hence
$f(0^+)=\infty$). Then $\phi_f(A,B)$ is bounded if and only if $B^2\le\lambda A$ for some
$\lambda>0$.
\item[\rm(3)] Let $f\in\OC(0,\infty)$ be arbitrary. If $A^2\le\lambda B$ and $B^2\le\lambda A$ for
some $\lambda>0$, then $\phi_f(A,B)$ is bounded. Moreover, assume that $c>0$ and $d>0$ in
\eqref{F-7.7} (hence $f'(\infty)=f(0^+)=\infty$). Then $\phi_f(A,B)$ is bounded if and only if
$A^2\le\lambda B$ and $B^2\le\lambda A$ for some $\lambda>0$.
\end{itemize}
\end{proposition}

\begin{proof}
We divide expression \eqref{F-7.7} in two parts as
\begin{align*}
f_1(t)&:=c(t-1)^2+\int_{[1,\infty)}{(t-1)^2\over t+\lambda}\,d\mu(\lambda), \\
f_2(t)&:=a+b(t-1)+d\,{(t-1)^2\over t}+\int_{(0,1)}{(t-1)^2\over t+\lambda}\,d\mu(\lambda),
\qquad t\in(0,\infty).
\end{align*}
Of course, we have $\phi_f(A,B)=\phi_{f_1}(A,B)+\phi_{f_2}(A,B)$. Note that $f_1(0^+)<\infty$
and $f_2'(\infty)<\infty$.

(1)\enspace
By \eqref{F-7.8} the assumption forces $d=0$ and
$\int_{(0,\infty)}\lambda^{-1}\,d\mu(\lambda)<\infty$. Since $f_2(0^+)<\infty$ as well as
$f_2'(\infty)<\infty$ in this case, $\phi_{f_2}(A,B)$ is bounded so that the question reduces to the
boundedness of $\phi_{f_1}(A,B)$. Note that
\[
{f_1(t)\over(t+1)^2}\le c+\int_{[1,\infty)}{1\over t+\lambda}\,d\mu(\lambda)\le k,
\qquad t\in(0,\infty),
\]
where $k:=c+\int_{[1,\infty)}\lambda^{-1}\,d\mu(\lambda)<\infty$. Hence one finds that
\[
\phi_{f_1}(A,B)\le k\phi_{(t+1)^2}(A,B)=k(\phi_{t^2}(A,B)+2A+B).
\]
Therefore, the first assertion holds by Theorem \ref{T-8.6}(3). Moreover, assume $c>0$; then
$f_1(t)\ge c(t-1)^2$ for all $t\in(0,\infty)$. If $\phi_f(A,B)$ is bounded, then so is $\phi_{f_1}(A,B)$
and hence so is $\phi_{(t-1)^2}(A,B)$. This means that $\phi_{t^2}(A,B)$ is bounded, and hence
Theorem \ref{T-8.6}(3) implies that $A^2\le\lambda B$ for some $\lambda>0$.

(2) is seen by applying item (1) to $\widetilde f$ and noting that $c,d$ are exchanged for
$\widetilde f$.

(3)\enspace
The proof is easy by applying items (1) and (2) to $f_1$ and $f_2$, respectively, given at the
beginning of the proof. The details are omitted here.
\end{proof}

\begin{remark}\label{R-8.10}\rm
A naive criterion for $\phi_f(A,B)$ to be bounded is given as follows. When $f\in\OC_0(0,\infty)$,
Theorem \ref{T-7.7} enables us to see that $\phi_f(A,B)$ is bounded if and only if there is a
$\lambda>0$ such that $\phi_f(A_\eps,B_\eps)\le\lambda I$ for all sufficiently small $\eps>0$.
This criterion can be extended to any $f\in\OC(0,\infty)$ by taking $f_0\in\OC_0(0,\infty)$ as in
the proof of Theorem \ref{T-7.5}. Thus, for any $f\in\OC(0,\infty)$ we notice that $\phi_f(A,B)$ is
bounded if and only if there is a $\lambda>0$ such that
$f(B_\eps^{-1/2}A_\eps B_\eps^{-1/2})\le\lambda B_\eps^{-1}$ for all sufficiently small $\eps>0$.
For instance, when $f(t)=t^2$, the last condition is rewritten as
$A_\eps B_\eps^{-1}A_\eps\le\lambda I$ or equivalently $A_\eps^2\le\lambda B_\eps$ (for all
small $\eps>0$), from which one can give an alternative proof of Theorem \ref{T-8.6}(3).
\end{remark}

As for $f(t)=t^\alpha$ with $1<\alpha\le2$, we here collect sufficient or necessary
conditions for $\phi_{t^\alpha}(A,B)$ being bounded as follows.

\begin{corollary}\label{C-8.11}
Let $1<\alpha\le2$. For a pair $(A,B)$ in $B(\cH)_+$ consider the following conditions:
\begin{itemize}
\item[\rm(a)] $A^2\le\lambda B$ for some $\lambda>0$,
\item[\rm(b)] $\phi_{t^\alpha}(A,B)$ ($=\phi_{t^{1-\alpha}}(B,A)$) is bounded,
\item[\rm(c)] $A^\alpha\le\lambda B^{\alpha-1}$ for some $\lambda>0$,
\item[\rm(d)] there is a $\lambda>0$ such that $\<A\xi,\xi\>^\alpha\le\lambda\<B\xi,\xi\>^{\alpha-1}$
for all $\xi\in\cH$, $\|\xi\|=1$,
\item[\rm(e)] $A\le\lambda B^{(\alpha-1)/\alpha}$ for some $\lambda>0$.
\end{itemize}
Then we have (a)$\implies$(b)$\implies$(d)$\implies$(e) and (a)$\implies$(c)$\implies$(d).
\end{corollary}

\begin{proof}
We have (a)$\implies$(b) by Proposition \ref{P-8.9}(1) and (b)$\implies$(d) by \eqref{F-7.17} in
Example \ref{E-7.12}. If (d) holds and $\|\xi\|=1$, then
\[
\<A\xi,\xi\>\le\lambda^{1/\alpha}\<B\xi,\xi\>^{(\alpha-1)/\alpha}
\le\lambda^{1/\alpha}\<B^{(\alpha-1)/\alpha}\xi,\xi\>.
\]
Hence (d)$\implies$(e) holds. Since (a) gives
$A^\alpha\le\|A\|^{2-\alpha}A^{2(\alpha-1)}\le\|A\|^{2-\alpha}\lambda^{\alpha-1}B^{\alpha-1}$,
we have (a)$\implies$(c). If (c) holds and $\|\xi\|=1$, then
\[
\<A\xi,\xi\>^\alpha\le\<A^\alpha\xi,\xi\>\le\lambda\<B^{\alpha-1}\xi,\xi\>
\le\lambda\<B\xi,\xi\>^{\alpha-1}.
\]
Hence (c)$\implies$(d) holds.
\end{proof}

\begin{example}\label{E-8.12}\rm
The function $t\log t$ ($t>0$) in $\OC(0,\infty)$ with its transpose $-\log t$ plays a significant role
in (quantum) information theory. Similarly to \cite[(2.2)]{FS}, for every $A,B\in B(\cH)_+$ and
$\rho\in B(\cH)_*^+$, we have
\[
\phi_{t\log t}(A,B)(\rho)=\lim_{\alpha\searrow0}\rho\Bigl({A-A\#_\alpha B\over\alpha}\Bigr)
\ \ \mbox{increasingly}.
\]
Indeed, since $(1-t^\alpha)/\alpha\nearrow-\log t$ ($t>0$) as $\alpha\searrow0$, it follows from
the monotone convergence theorem that
$\rho(\phi_{(1-t^\alpha)/\alpha}(B,A))\nearrow\phi_{-\log t}(B,A)(\rho)$ as $\alpha\searrow0$.
Therefore, when it is bounded, $\phi_{t\log t}(A,B)$ is the minus sign of the \emph{relative operator
entropy} $S(A\,|B)$ studied in \cite{FK,FS}. In the case $f(t)=t\log t$, Theorem \ref{T-7.2}(2),
Proposition \ref{P-8.9}(1) and Corollary \ref{C-8.5}(1) read as follows.
\begin{itemize}
\item For $A,B_1,B_2\in B(\cH)_+$.  
$B_1\le B_2$\,$\implies$\,$\phi_{t\log t}(A,B_1)\ge\phi_{t\log t}(A,B_2)$.
\item If $A^2\le\lambda B$ for some $\lambda>0$, then $\phi_{t\log t}(A,B)$ is bounded.
\item If $\phi_{t\log t}(A,B)$ is bounded, then $A$ is $B$-absolutely continuous. 
\end{itemize}
These improve the corresponding facts given in \cite[\S2]{FS}. Here we emphasize that the
PW-functional calculus $\phi_{t\log t}(A,B)=\phi_{-\log t}(B,A)$ extends a definition of $-S(A\,|B)$
to all $A,B\in B(\cH)_+$, whose value is though admitted to an element of $\widehat{B(\cH)}_\lb$.
This extension is conceptually natural because the original entropy quantity can be $\infty$,
and it gives a better understanding of $S(A\,|B)$ beyond the discussions in \cite[\S2]{FS}.
\end{example}

In the rest of the section we further discuss the question of boundedness of $\phi_f(A,B)$ in
connection with AH (Ando--Hiai) inequalities. An operator perspective $\phi_f$ is said to satisfy
an \emph{AH inequality} if we have
\begin{align}\label{F-8.11}
\phi_f(A,B)\le I\,\implies\,\phi_f(A^p,B^p)\le I
\end{align}
for all $(A,B)$ in (a certain subset of) $B(\cH)_+\times B(\cH)_+$ and for either all $p\ge1$
or all $p\in(0,1]$. Inequalities of this type were first shown in \cite{AH} for the weighted
geometric means and further studied in, e.g., \cite{HSW,Wa}. A positive $\bR$-valued function
$f$ on $(0,\infty)$ is said to be \emph{power monotone increasing} (\emph{pmi} for short) if
$f(t^p)\ge f(t)^p$ for all $t>0$ and $p\ge1$. The next proposition is a slight extension of
an AH-inequality in \cite{HSW}.

\begin{proposition}\label{P-8.13}
Let $f$ be a pmi positive function on $(0,\infty)$. Assume that either $f\in\OC(0,\infty)$ with
$f(0^+)=0$ or $f$ is operator monotone decreasing on $(0,\infty)$. Then for any
$A,B\in B(\cH)_+$, \eqref{F-8.11} holds for all $p\in(0,1]$ or equivalently,
\[
\|\phi_f(A^p,B^p)\|\le\|\phi_f(A,B)\|^p,\qquad0<p\le1,
\]
where the operator norm $\|\phi_f(A,B)\|$ is understood to be $\infty$ if $\phi_f(A,B)$ is
unbounded. Consequently, if $\phi_f(A,B)$ is bounded, then so is $\phi_f(A^p,B^p)$ for all
$p\in(0,1]$.
\end{proposition}

\begin{proof}
Let $A,B\in B(\cH)_+$ and assume that $\phi_f(A,B)\le I$. For any $\eps>0$, Theorem \ref{T-7.2}(1)
gives
\[
\phi_f(A_\eps,B_\eps)\le\phi_f(A,B)+\phi_f(\eps I,\eps I)\le(1+\eps f(1))I,
\]
so that
\[
\phi_f\biggl({A_\eps\over1+\eps f(1)},{B_\eps\over1+\eps f(1)}\biggr)\le I
\]
thanks to the scalar homogeneity of $\phi_f(A,B)$. By \cite[Corollary 3.8 or Proposition 6.10]{HSW}
we have, for every $p\in(0,1]$,
\[
\phi_f\biggl({A_\eps^p\over(1+\eps f(1))^p},
{B_\eps^p\over(1+\eps f(1))^p}\biggr)\le I
\]
and hence $\phi_f(A_\eps^p,B_\eps^p)\le(1+\eps f(1))^pI$. Since $A_\eps^p\to A^p$ and
$B_\eps^p\to B^p$ in SOT (even in the operator norm) as $\eps\searrow0$, Theorem \ref{T-7.5}
implies that $\phi_f(A^p,B^p)\le I$. Hence the first assertion follows and the remaining are immediate.
\end{proof}

In particular, when $f(t)=t^2$ (or $f(t)=t^{-1}$), Proposition \ref{P-8.13} is an immediate
consequence of Theorem \ref{T-8.6}(3) since $A^2\le\lambda B\implies A^{2p}\le\lambda^p B^p$
for $0<p\le1$.

\section{Integral expressions and variational expressions}\label{Sec-9}

Integral expression is an important ingredient of theory of operator means and connections in
\cite{PW1,KA}. The integral expression for operator connections $\sigma$ in \cite{KA} is
\begin{align}\label{F-9.1}
A\sigma B=aA+bB+\int_{(0,\infty)}{1+\lambda\over\lambda}((\lambda A):B)\,d\mu(\lambda),
\qquad A,B\in B(\cH)_+,
\end{align}
where $a,b\ge0$ and $\mu$ is a finite positive measure on $(0,\infty)$. The expression is based
on the integral representation of operator monotone functions on $[0,\infty)$ (see, e.g., \cite{Bh,Hi}).
Furthermore, variational expressions for various functional calculi have played an important
role in topics related to this paper; see \cite{AT,PW1,PW2,Ku,Do,Ko3} and so on.

In the first half of this section, for $f\in\OC(0,\infty)$ we discuss integral expressions of
$\phi_f(A,B)$ in a similar fashion to \eqref{F-9.1}. The first result is based on the integral
representation \eqref{F-7.7} of general $f\in\OC(0,\infty)$.

\begin{theorem}\label{T-9.1}
Let $f\in\OC(0,\infty)$ be given in the representation \eqref{F-7.7}. Let $a_0:=b-2c+d$ and
$b_0:=a-b+c-2d$. Then for every $A,B\in B(\cH)_+$ and $\rho\in B(\cH)_*^+$, we have
\begin{equation}\label{F-9.2}
\begin{aligned}
\phi_f(A,B)(\rho)
&=a_0\rho(A)+b_0\rho(B)+c\phi_{t^2}(A,B)(\rho)+d\phi_{t^2}(B,A)(\rho) \\
&\qquad+\int_{(0,\infty)}\Bigl[\rho(A)+{1\over\lambda}\,\rho(B)
-\Bigl({1+\lambda\over\lambda}\Bigr)^2\rho(A:(\lambda B))
\Bigr]\,d\mu(\lambda).
\end{aligned}
\end{equation}
\end{theorem}

\begin{proof}
For each $\lambda\in(0,\infty)$ we set
\begin{align}\label{F-9.3}
g_\lambda(t):={(t-1)^2\over t+\lambda}
=t+{1\over\lambda}-\Bigl({1+\lambda\over\lambda}\Bigr)^2{\lambda t\over t+\lambda},
\qquad t\in(0,\infty).
\end{align}
We notice that
\begin{align}\label{F-9.4}
\phi_{g_\lambda}(A,B)=A+{1\over\lambda}\,B
-\Bigl({1+\lambda\over\lambda}\Bigr)^2(A:(\lambda B))
\ \ (\in B(\cH)_+),\quad A,B\in B(\cH)_+.
\end{align}
For each $n\in\bN$ we define
\[
f_n(t):=a+b(t-1)+cng_n(t)+dg_{1/n}(t)+\int_{[1/n,n]}g_\lambda(t)\,d\mu(\lambda),
\qquad t\in(0,\infty),
\]
which is the same as $f_n$ given in \eqref{F-7.9}. For every $A,B\in B(\cH)_+$ and
$\rho\in B(\cH)_*^+$, it follows from Lemma \ref{L-7.9} that
\begin{equation}\label{F-9.5}
\begin{aligned}
&\phi_f(A,B)(\rho) \\
&\quad=\lim_{n\to\infty}\rho(\phi_{f_n}(A,B)) \\
&\quad=b\rho(A)+(a-b)\rho(B) \\
&\qquad+\lim_{n\to\infty}\biggl[c\rho(\phi_{ng_n}(A,B))+d\rho(\phi_{g_{1/n}}(A,B))
+\int_{[1/n,n]}\rho(\phi_{g_\lambda}(A,B))\,d\mu(\lambda)\biggr].
\end{aligned}
\end{equation}
Since $ng_n(t)\nearrow(t-1)^2$ and $g_{1/n}(t)\nearrow(t-1)^2/t$ for all $t\in(0,\infty)$, by the
monotone convergence theorem we have
\begin{align*}
\lim_{n\to\infty}\rho(\phi_{ng_n}(A,B))
&=(\phi_{(t-1)^2}(A,B))(\rho)=\phi_{t^2}(A,B)(\rho)-2\rho(A)+\rho(B), \\
\lim_{n\to\infty}\rho(\phi_{g_{1/n}}(A,B))&=(\phi_{(t-1)^2/t}(A,B))(\rho)
=(\phi_{(t-1)^2}(B,A))(\rho)\quad\mbox{(by \eqref{F-7.3})} \\
&=\phi_{t^2}(B,A)-2\rho(B)+\rho(A),
\end{align*}
and
\begin{align*}
&\lim_{n\to\infty}\int_{[1/n,n]}\rho(\phi_{g_\lambda}(A,B))\,d\mu(\lambda) \\
&\qquad=\int_{(0,\infty)}\rho(\phi_{g_\lambda}(A,B))\,d\mu(\lambda) \\
&\qquad=\int_{(0,\infty)}\Bigl[\rho(A)+{1\over\lambda}\,\rho(B)
-\Bigl({1+\lambda\over\lambda}\Bigr)^2\rho(A:(\lambda B))\Bigr]\,d\mu(\lambda).
\end{align*}
Hence the asserted expression follows by combining \eqref{F-9.5} and these (increasing)
convergences.
\end{proof}

We have obtained a handy description of $\phi_{t^2}(A,B)$ in the preceding section. Apart from
two $\phi_{t^2}$-terms, the main term of the integral expression \eqref{F-9.2} is (minus) parallel
sum with a particular parametrization, though not so simple as \eqref{F-9.1}.

Assume that $f(0^+)<\infty$ in Theorem \ref{T-9.1}. Then $d=0$ and
$\int_{(0,\infty)}\lambda^{-1}\,d\mu(\lambda)<\infty$ thanks to \eqref{F-7.8}. Thus we can pull
$\bigl(\int_{(0,\infty)}\lambda^{-1}\,d\mu(\lambda)\bigr)\rho(B)$ out of the integral in \eqref{F-9.2}.
Since $b_0+\int_{(0,\infty)}\lambda^{-1}\,d\mu(\lambda)=f(0^+)$ by \eqref{F-7.8}, we can
rewrite \eqref{F-9.2} as
\begin{equation}\label{F-9.6}
\begin{aligned}
\phi_f(A,B)(\rho)&=a_0\rho(A)+f(0^+)\rho(B)+c\phi_{t^2}(A,B)(\rho) \\
&\qquad+\int_{(0,\infty)}
\Bigl[\rho(A)-\Bigl({1+\lambda\over\lambda}\Bigr)^2\rho(A:(\lambda B))
\Bigr]\,d\mu(\lambda).
\end{aligned}
\end{equation}

For any $f\in\OC(0,\infty)$ we define
\[
f'(0^+):=\lim_{t\searrow0}f'(t),
\]
whose limit exists in $[-\infty,\infty)$ by the numerical convexity of $f$ . Obviously, $f'(0^+)>-\infty$
implies $f(0^+)<\infty$. When $f'(0^+)>-\infty$, it is known that $f$ admits, besides expression
\eqref{F-7.7}, an integral expression
\begin{align}\label{F-9.7}
f(t)=f(0^+)+f'(0^+)t+ct^2+\int_{(0,\infty)}{t^2\over t+\lambda}\,d\nu(\lambda),
\qquad t\in(0,\infty),
\end{align}
where $c\ge0$ and $\nu$ is a positive measure on $(0,\infty)$ with
$\int_{(0,\infty)}(1+\lambda)^{-1}\,d\nu(\lambda)<\infty$. Indeed, in this case, $h(t):=(f(t)-f(0^+))/t$
with $h(0):=f'(0^+)$ is a non-negative operator monotone function on $[0,\infty)$ (see
\cite[Theorem 2.4]{HP}), so that \eqref{F-9.7} immediately follows from a familiar integral
expression of $h$. Since
\[
{t^2\over t+\lambda}=t-{\lambda t\over t+\lambda},\qquad t\in(0,\infty),
\]
we note as \eqref{F-9.4} that
\[
\phi_{t^2/(t+\lambda)}(A,B)=A-(A:(\lambda B)),\qquad A,B\in B(\cH)_+.
\]
Hence the next proposition can be shown, based on \eqref{F-9.7}, similarly to Theorem \ref{T-9.1},
whose proof is omitted here.

\begin{proposition}\label{P-9.2}
Let $f\in\OC(0,\infty)$ with $f'(0^+)>-\infty$ so that $f$ has expression \eqref{F-9.7}. Then
for every $A,B\in B(\cH)_+$ and $\rho\in B(\cH)_*^+$, we have
\begin{equation}\label{F-9.8}
\begin{aligned}
\phi_f(A,B)(\rho)&=f'(0^+)\rho(A)+f(0^+)\rho(B)+c\phi_{t^2}(A,B)(\rho) \\
&\qquad+\int_{(0,\infty)}[\rho(A)-\rho(A:(\lambda B))]\,d\nu(\lambda).
\end{aligned}
\end{equation}
\end{proposition}

The following formula of $\phi_f(P,Q)$ for two projections $P,Q$ is similar to \cite[Theorem 3.7]{KA}
for operator connections.

\begin{proposition}[Two projections]\label{P-9.3}
For every projections $P,Q\in B(\cH)$ we have
\begin{align}\label{F-9.9}
\phi_f(P,Q)=f(1)(P\wedge Q)+f'(\infty)(P-P\wedge Q)+f(0^+)(Q-P\wedge Q).
\end{align}
\end{proposition}

\begin{proof}
Example \ref{E-8.7} says that
\[
\phi_{t^2}(P,Q)=P\wedge Q+\infty\cdot(P-P\wedge Q),\quad
\phi_{t^2}(Q,P)=P\wedge Q+\infty\cdot(Q-P\wedge Q),
\]
and furthermore we have
\begin{align*}
P+{1\over\lambda}\,Q-\Bigl({1+\lambda\over\lambda}\Bigr)^2(P:(\lambda Q))
&=P+{1\over\lambda}\,Q-{1+\lambda\over\lambda}(P\wedge Q) \\
&=(P-P\wedge Q)+{1\over\lambda}(Q-P\wedge Q).
\end{align*}
Inserting these into expression \eqref{F-9.2} yields
\begin{align*}
\phi_f(P,Q)&=a_0P+b_0Q+(c+d)(P\wedge Q)
+\biggl(c\cdot\infty+\int_{(0,\infty)}d\mu(\lambda)\biggr)(P-P\wedge Q) \\
&\qquad+\biggl(d\cdot\infty+\int_{(0,\infty)}{1\over\lambda}\,d\mu(\lambda)\biggr)
(Q-P\wedge Q) \\
&=(a_0+b_0+c+d)(P\wedge Q)
+\biggl(a_0+c\cdot\infty+\int_{(0,\infty)}d\mu(\lambda)\biggr)(P-P\wedge Q) \\
&\qquad+\biggl(b_0+d\cdot\infty+\int_{(0,\infty)}{1\over\lambda}\,d\mu(\lambda)\biggr)
(Q-P\wedge Q),
\end{align*}
which is \eqref{F-9.9} because of $a_0+b_0+c+d=a=f(1)$ as well as
\[
a_0+c\cdot\infty+\int_{(0,\infty)}d\mu(\lambda)=f'(\infty),\qquad
b_0+d\cdot\infty+\int_{(0,\infty)}{1\over\lambda}\,d\mu(\lambda)=f(0^+)
\]
thanks to \eqref{F-7.8}.
\end{proof}

In the second half of the section, we are concerned with variational expressions of $\phi_f(A,B)$.
For any $f\in\OC(0,\infty)$ let $\alpha_n$, $\beta_n$, $\nu_n$ and $h_n$ ($n\in\bN$) be defined
by \eqref{F-7.11}--\eqref{F-7.14}. For every $A,B\in B(\cH)_+$ and $\rho\in B(\cH)_*^+$, by
Lemma \ref{L-7.9} we write
\begin{equation}\label{F-9.10}
\begin{aligned}
\phi_f(A,B)(\rho)
&=\sup_n\bigl[\alpha_n\rho(A)+\beta_n\rho(B)-\rho(B\sigma_{h_n}A)\bigr] \\
&=\sup_n\Bigl[\alpha_n(\rho(A)+\beta_n\rho(B)
-\int_{[1/n,n]}{1+\lambda\over\lambda}\,\rho(A:(\lambda B))\,d\nu_n(\lambda)\Bigr]
\end{aligned}
\end{equation}
thanks to \eqref{F-9.1}.

In the discussions below, for each $\xi\in\cH$ and an interval $J\subseteq(0,\infty)$, we will use
the notation $\cP(\xi;J,\cH)$ to denote the set of all pairs $(\eta(\cdot),\zeta(\cdot))$ of piecewise
constant functions on $J$ with finitely many values in $\cH$ such that $\eta(t)+\zeta(t)=\xi$ for all
$t\in J$.

The following variational expressions in Theorem \ref{T-9.4}, Proposition \ref{P-9.6} and Example
\ref{E-9.7} are certainly related to Pusz and Woronowicz's ones in \cite{PW1}, \cite[\S2]{PW2}
(also \cite[\S4]{Do}), and more directly related to \cite{Ko3} and \cite[\S III]{Hi1}; see Remark
\ref{R-9.8} below for more specific discussion.

\begin{theorem}\label{T-9.4}
Let $f\in\OC(0,\infty)$, and for each $n\in\bN$ let $\alpha_n$, $\beta_n$ and $\nu_n$ be as stated
above. Then for every $A,B\in B(\cH)_+$ and $\xi\in\cH$, we have
\begin{equation}\label{F-9.11}
\begin{aligned}
\phi_f(A,B)(\omega_\xi)
&=\sup_n\sup_{\eta(\cdot),\zeta(\cdot)}
\Bigl[\alpha_n\<A\xi,\xi\>+\beta_n\<B\xi,\xi\> \\
&\qquad\quad-\int_{[1/n,n]}{1+t\over t}(\<A\eta(t),\eta(t)\>
+t\<B(\zeta(t),\zeta(t)\>)\,d\nu_n(t)\Bigr],
\end{aligned}
\end{equation}
where the second supremum is taken over all pairs $(\eta(\cdot),\zeta(\cdot))$ in
$\cP(\xi;[1/n,n],\cH)$.
\end{theorem}

\begin{proof}
In view of \eqref{F-9.10} (for $\rho=\omega_\xi$), it suffices to prove that, for each fixed $n\in\bN$,
\begin{align*}
&\int_{[1/n,n]}{1+t\over t}\<(A:(tB))\xi,\xi\>\,d\nu_n(t) \\
&\qquad=\inf_{\eta(\cdot),\zeta(\cdot)}\int_{[1/n,n]}{1+t\over t}
(\<A\eta(t),\eta(t)\>+t\<B\zeta(t),\zeta(t)\>)\,d\nu_n(t).
\end{align*}
Denote the above left-hand and the right-hand sides by $L_n(A,B,\xi)$ and $R_n(A,B,\xi)$,
respectively. In the following, we will crucially use the well-known variational formula for parallel
sum due to \cite[Theorem 9]{AT}, saying that
\[
\<(A:(tB))\xi,\xi\>=\inf\bigl\{\<A\eta,\eta\>+t\<B\zeta,\zeta\>:
\eta,\zeta\in\cH,\,\eta+\zeta=\xi\bigr\}.
\]
Hence it is clear that $L_n(A,B,\xi)\le R_n(A,B,\xi)$. (At this point, we remark that the discussion
below overlaps with Pusz and Woronowicz's method of variational expressions based on
essentially the same formula \cite[p.\ 161, Lemma]{PW1} as above.)

Conversely, for any $\delta>0$ and $s\in[1/n,n]$, one can choose $\eta,\zeta\in\cH$ (depending
on $\delta$ and $s$) with $\eta+\zeta=\xi$ such that
\begin{align}\label{F-9.12}
\<A\eta,\eta\>+t\<B\zeta,\zeta\><\<(A:(tB))\xi,\xi\>+\delta
\end{align}
holds for $t=s$. Here we notice that $\<(A:(tB))\xi,\xi\>$ is upper semicontinuous in $t>0$,
because $\<(A:(tB))\xi,\xi\>=\inf_{\eps>0}\<(A_\eps:(tB_\eps))\xi,\xi\>$ and
$t>0\mapsto A_\eps:(tB_\eps)$ is continuous (in the operator norm).  Consequently, \eqref{F-9.12}
holds for $t$ in an interval $(s-\delta_s,s+\delta_s)$. Choosing a finite open covering of
$[1/n,n]$ from $\{(s-\delta_s,s+\delta_s)\}_{s\in[1/n,n]}$, one can easily define a pair
$(\eta(\cdot),\zeta(\cdot))$ as stated in the theorem such that
\[
\<A\eta(t),\eta(t)\>+t\<B\zeta(t),\zeta(t)\><\<(A:(tB))\xi,\xi\>+\delta
\quad\mbox{for all $t\in[1/n,n]$}.
\]
This implies that
\begin{align*}
R_n(A,B,\xi)&\le\int_{[1/n,n]}{1+t\over t}[\<(A:(tB))\xi,\xi\>+\delta]\,d\nu_n(t) \\
&=L_n(A,B,\xi)+\delta\int_{[1/n,n]}{1+t\over t}\,d\nu_n(t).
\end{align*}
Since $\delta>0$ is arbitrary, $R_n(A,B,\xi)\le L_n(A,B,\xi)$ follows.
\end{proof}

\begin{remark}\label{R-9.5}\rm
In view of a remark after \eqref{F-2.5}, note that $\phi_f(A,B)$ is uniquely determined by
expression \eqref{F-9.11}. Furthermore, we have a variational expression of $\phi_f(A,B)$
directly coupled with $\rho\in B(\cH)_*^+$ as follows:
\begin{equation}\label{F-9.13}
\begin{aligned}
\phi_f(A,B)(\rho)
&=\sup_n\sup_{X(\cdot),Y(\cdot)}\Bigl[\alpha_n\rho(A)+\beta_n\rho(B) \\
&\qquad-\int_{[1/n,n]}{1+t\over t}(\Tr(X(t)X(t)^*A)
+t\Tr(Y(t)Y(t)^*B))\,d\nu_n(t)\Bigr],
\end{aligned}
\end{equation}
where the second supremum is taken over all pairs $(X(\cdot),Y(\cdot))$ of piecewise constant
functions on $[1/n,n]$ with finitely many values in $\cC_2(\cH)$ (the Hilbert--Schmidt class) such
that $X(t)+Y(t)=\rho^{1/2}$ (the square-root of the density operator of $\rho$) for all $t\in[1/n,n]$.
In fact, note that $B(\cH)$ is standardly represented on the Hilbert space $\cC_2(\cH)$ with the
inner product $\<X,Y\>=\Tr(Y^*X)$ by left multiplication $\pi(A)X=AX$ for $A\in B(\cH)$,
$X\in\cC_2(\cH)$. Hence, expression \eqref{F-9.13} follows from \eqref{F-9.10} in the same way
as \eqref{F-9.11}. The remark here is also available for all variational expressions given in the
rest of the section.
\end{remark}

A point of the variational expressions in \eqref{F-9.11} and \eqref{F-9.13} is that the function of
$(A,B)\in B(\cH)_+\times B(\cH)_+$ inside the bracket $[\cdots]$ is jointly linear and sequentially
continuous in the weak operator topology. Based on this fact, important properties of $\phi_f(A,B)$
such as Theorems \ref{T-7.2}(1), \ref{T-7.5} and Proposition \ref{P-7.10} are shown in a
straightforward manner. Furthermore, the assumption $A_n\to A$ and $B_n\to B$ in SOT in
Theorem \ref{T-7.5} and Remark \ref{R-7.6} can be relaxed into $A_n\to A$ and $B_n\to B$ in the
weak operator topology.

The above variational expressions are presented with the cut-off interval $[1/n,n]$. But it is also
possible to provide variational expressions without cut-off, based on the integral expression
\eqref{F-9.2} or \eqref {F-9.6} or \eqref{F-9.8}. First, by Proposition \ref{P-8.3} we note that the
term $\phi_{t^2}(A,B)$ enjoys the variational expression
\[
\phi_{t^2}(A,B)(\omega_\xi)
=\sup_n\sup_{\eta,\zeta}n(\<A\xi,\xi\>-\<A\eta,\eta\>-n\<B\zeta,\zeta\>),
\]
where the second supremum is taken over $\eta,\zeta\in\cH$ with $\eta+\zeta=\xi$. So we may
concentrate our consideration to the main integral term of those expressions. As for the main
integral term of \eqref{F-9.2} we show the following:

\begin{proposition}\label{P-9.6}
In the situation of Theorem \ref{T-9.1}, for every $A,B\in B(\cH)_+$ and $\xi\in\cH$, we have
\begin{equation}\label{F-9.14}
\begin{aligned}
&\int_{(0,\infty)}\Bigl[\<A\xi,\xi\>+{1\over t}\<B\xi,\xi\>
-\Bigl({1+t\over t}\Bigr)^2\<(A:(tB))\xi,\xi\>\Bigr]\,d\mu(t) \\
&\qquad=\sup_{\eta(\cdot),\zeta(\cdot)}
\int_{(0,\infty)}\Bigl[\<A\xi,\xi\>+{1\over t}\<B\xi,\xi\> \\
&\hskip3.5cm-\Bigl({1+t\over t}\Bigr)^2(\<A\eta(t),\eta(t)\>+t\<B\zeta(t),\zeta(t)\>)
\Bigr]\,d\mu(t),
\end{aligned}
\end{equation}
where the supremum is taken over all pairs $(\eta(\cdot),\zeta(\cdot))$ in $\cP(\xi;(0,\infty),\cH)$
such that $\eta(t)=0$ for all sufficiently small $t>0$ and $\zeta(t)=0$ for all sufficiently large $t>0$.
\end{proposition}

\begin{proof}
First, we confirm that the function inside the bracket on the right-hand side of \eqref{F-9.14} is
$\mu$-integrable for any $(\eta(\cdot),\zeta(\cdot))$ stated above. For such a pair
$(\eta(\cdot),\zeta(\cdot))$ we choose an $r\in(0,1)$ such that $\eta(t)=0$ for all $t\in(0,r)$ and
$\zeta(t)=0$ for all $t\in(r^{-1},\infty)$. Then the function in question is equal to $k_1(t)$ on $(0,r)$
and to $k_2(t)$ on $(r^{-1},\infty)$, given by
\[
k_1(t):=\<A\xi,\xi\>-(2+t)\<B\xi,\xi\>,\qquad
k_2(t):=-{1+2t\over t^2}\<A\xi,\xi\>+{1\over t}\<B\xi,\xi\>.
\]
In view of $\int_{(0,\infty)}(1+t)^{-1}\,d\mu(t)<\infty$, the functions $k_1(\cdot)$ and $k_2(\cdot)$
are $\mu$-integrable on $(0,r)$ and $(r^{-1},\infty)$, respectively. Also, the function in question is
clearly $\mu$-integrable on $[r,r^{-1}]$, so the desired $\mu$-integrability is verified.

Denote the left-hand and the right-hand sides of \eqref{F-9.14} by $L(A,B,\xi)$ and
$R(A,B,\xi)$. As in the proof of Theorem \ref{T-9.4}, $L(A,B,\xi)\ge R(A,B,\xi)$ is immediately
seen. For any $\alpha<L(A,B,\xi)$, by the Lebesgue convergence and the monotone convergence
theorems, we note that
\begin{align*}
&\int_{(0,r)}k_1(t)\,d\mu(t)+\int_{[r,r^{-1}]}\Bigl[\<A\xi,\xi\>+{1\over t}\<B\xi,\xi\>
-\Bigl({1+t\over t}\Bigr)^2\<(A:(tB))\xi,\xi\>\Bigr]\,d\mu(t) \\
&\hskip3cm+\int_{(r^{-1},\infty)}k_2(t)\,d\mu(t)
\end{align*}
converges to $L(A,B,\xi)$ as $r\searrow0$. Hence one can choose an $r\in(0,1)$ such that
\begin{align*}
&\int_{[r,r^{-1}]}\Bigl[\<A\xi,\xi\>+{1\over t}\<B\xi,\xi\>
-\Bigl({1+t\over t}\Bigr)^2\<(A:(tB))\xi,\xi\>\Bigr]\,d\mu(t) \\
&\qquad>\beta:=\alpha-\int_{(0,r)}k_1(t)\,d\mu(t)-\int_{(r^{-1},\infty)}k_2(t)\,d\mu(t),
\end{align*}
that is,
\[
\int_{[r,r^{-1}]}\Bigl({1+t\over t}\Bigr)^2\<(A:(tB))\xi,\xi\>\,d\mu(t)
<\int_{[r,r^{-1}]}\Bigl(\<A\xi,\xi\>+{1\over t}\<B\xi,\xi\>\Bigr)\,d\mu(t)-\beta.
\]
Now, in a similar way to the proof of Theorem \ref{T-9.4}, one can find a pair
$(\eta(\cdot),\zeta(\cdot))$ in $\cP(\xi;[r,r^{-1}],\cH)$ such that
\begin{align*}
&\int_{[r,r^{-1}]}\Bigl({1+t\over t}\Bigr)^2(\<A\eta(t),\eta(t)\>+t\<B\zeta(t),\zeta(t)\>)\,d\mu(t) \\
&\qquad<\int_{[r,r^{-1}]}\Bigl(\<A\xi,\xi\>+{1\over t}\<B\xi,\xi\>\Bigr)\,d\mu(t)-\beta.
\end{align*}
Extending $\eta(t),\zeta(t)$ to $(0,\infty)$ as $\eta(t)=0$ for all $t\in(0,r)$ and $\zeta(t)=0$ for all
$t\in(r^{-1},\infty)$, we have
\[
R(A,B,\xi)\ge\int_{(0,r)}k_1(t)\,d\mu(t)+\beta+\int_{(r^{-1},\infty)}k_2(t)\,d\mu(t)=\alpha,
\]
so that $R(A,B,\xi)\ge L(A,B,\xi)$ follows by letting $\alpha\nearrow L(A,B,\xi)$.
\end{proof}

In the situation where $f(0^+)<\infty$ or $f'(0^+)>-\infty$, we have a variational expression for the
integral term in \eqref{F-9.6} or \eqref{F-9.8} in a similar manner to Proposition \ref{P-9.6}.
Instead of stating these versions, let us give typical examples in the following:

\begin{example}\label{E-9.7}\rm
(1)\enspace
Consider $t^\alpha$ ($t>0$) where $1<\alpha<2$, whose familiar expression
\[
t^\alpha={\sin((\alpha-1)\pi)\over\pi}\int_{(0,\infty)}
\Bigl({t\over\lambda}-{t\over t+\lambda}\Bigr)\lambda^{\alpha-1}\,d\lambda,
\qquad t\in(0,\infty),
\]
is a special case of \eqref{F-9.7}. The corresponding integral expression (see \eqref{F-9.8})
and the variational expression are given for $A,B\in B(\cH)_+$ and $\xi\in\cH$ as follows:
\begin{align*}
&\phi_{t^\alpha}(A,B)(\omega_\xi)\ (=\phi_{t^{1-\alpha}}(B,A)(\omega_\xi)) \\
&\qquad={\sin((\alpha-1)\pi)\over\pi}\int_{(0,\infty)}[\<A\xi,\xi\>-\<(A:(\lambda B))\xi,\xi\>]
\lambda^{\alpha-2}\,d\lambda \\
&\qquad={\sin((\alpha-1)\pi)\over\pi}\sup_{\eta(\cdot),\zeta(\cdot)}\int_{(0,\infty)}
[\<A\xi,\xi\>-\<A\eta(\lambda),\eta(\lambda)\>
-\lambda\<B\zeta(\lambda),\zeta(\lambda)\>]\lambda^{\alpha-2}\,d\lambda,
\end{align*}
where the supremum is taken over all pairs $(\eta(\cdot),\zeta(\cdot))$ in $\cP(\xi;(0,\infty),\cH)$
such that $\zeta(\lambda)=0$ for all sufficiently large $\lambda>0$. (The last condition guarantees
that the integral against the supremum is well defined.)

(2)\enspace
Next, consider $t\log t$ ($t>0$), whose expression of the form \eqref{F-7.7} is
\[
t\log t=t-1+\int_{(0,\infty)}{(t-1)^2\over t+\lambda}\cdot{\lambda\over(1+\lambda)^2}\,d\lambda.
\]
But a better-known formula is
\begin{align}\label{F-9.15}
t\log t=\int_{(0,\infty)}\Bigl({t\over1+\lambda}-{t\over t+\lambda}\Bigr)\,d\lambda.
\end{align}
Using either expression (though the latter is more convenient), we find the integral expression
and the variational expression of $\phi_{t\log t}(A,B)$ as follows:
\begin{equation}\label{F-9.16}
\begin{aligned}
&\phi_{t\log t}(A,B)(\omega_\xi)\ (=\phi_{-\log t}(B,A)(\omega_\xi)) \\
&\quad=\int_{(0,\infty)}\Bigl[{1\over1+\lambda}\<A\xi,\xi\>
-{1\over\lambda}\<(A:(\lambda B))\xi,\xi\>\Bigr]\,d\lambda \\
&\quad=\sup_{\eta(\cdot),\zeta(\cdot)}\int_{(0,\infty)}\Bigl[{1\over1+\lambda}\<A\xi,\xi\>
-{1\over\lambda}\<A\eta(\lambda),\eta(\lambda)\>
-\<B\zeta(\lambda),\zeta(\lambda)\>\Bigr]\,d\lambda,
\end{aligned}
\end{equation}
where the supremum is taken in the same way as in Proposition \ref{P-9.6}.
\end{example}

\begin{remark}\label{R-9.8}\rm
Pusz and Woronowicz \cite{PW2} provided a variational expression for the PW-functional calculus
associated with $\psi(x,y):=x\log(x/y)$ ($x\ge0$, $y>0$), $\infty$ ($x>0$, $y=0$) and $0$ ($x=y=0$)
by making use of an integral formula
\begin{align}\label{F-9.17}
\psi(x,y)=-\int_{(0,1)}{x(y-x)\over x+(y-x)s}\,ds.
\end{align}
Along the same lines, a variational expression of relative entropy was obtained in \cite{PW2} and
\cite[\S4]{Do} by applying \eqref{F-9.17} to the PW-functional calculus for two positive quadratic
forms suitably induced from two positive linear functionals (on a $C^*$-algebra). A different method
to variational expression was developed in \cite{Ko3}  for relative entropy of normal positive
functionals on a von Neumann algebra (whose method was extended in \cite{Hi1} to more general
$f$-divergences). Here, note that the formula \eqref{F-9.17} is essentially the same as
\eqref{F-9.15}; in fact, a change of variable $\lambda=s/(1-s)$ ($0<s<1$) transforms \eqref{F-9.15}
into \eqref{F-9.17} for $y=1$. Nevertheless, it does not seem possible to directly transform
expression \eqref{F-9.16} for $\psi(A,B)=\phi_{t\log t}(A,B)$ into that given in \cite{PW2}.
\end{remark}

\section{Axiomatization}\label{Sec-10}

Kubo and Ando \cite{KA} formulated operator connections
$\sigma:B(\cH)_+\times B(\cH)_+\to B(\cH)_+$ with the following postulates:
\begin{itemize}
\item[(I)] (\emph{Joint operator monotonicity})\enspace
$A_1\le A_2$ and $B_1\le B_2$ imply $A_1\sigma B_1\le A_2\sigma B_2$ for all
all $A_i,B_i\in B(\cH)_+$.
\item[(II)] (\emph{Transformer inequality})\enspace $C(A\sigma B)C\le(CAC)\sigma(CBC)$ for
$A,B,C\in B(\cH)_+$.
\item[(III)] (\emph{Upper continuity})\enspace If $A_n,B_n\in B(\cH)_+$ ($n\in\bN$),
$A_n\searrow A$ and $B_n\searrow B$, then $A_n\sigma B_n\searrow A\sigma B$.
\end{itemize} 
(Moreover, operator means $\sigma$ have to satisfy $I\sigma I=I$ in addition.) One of their
achievements establishes an order isomorphism between the operator connections $\sigma$ and
the non-negative operator monotone functions $h$ on $[0,\infty)$ in such a way that
$A\sigma B = A^{1/2}h(A^{-1/2}BA^{-1/2})A^{1/2}$ (if $A\in B(\cH)_{++}$). In this way, they
gave an axiomatic formulation of operator connections (and means).

In this section we consider axiomatic formulations of Kubo and Ando's type for the PW-functional
calculus and extended operator convex perspectives. The first theorem is an axiomatization of
the general PW-functional calculus though in the bounded situation. As before, we write
$A_\eps := A+\eps I$ for $A \in B(\cH)^+$ and $\eps > 0$ in the following.

\begin{theorem}\label{T-10.1}
An operation $\Phi$ giving, for each Hilbert space $\cH$, a mapping 
\[
 \Phi_\cH : B(\cH)_+ \times B(\cH)_+ \,\longrightarrow\, B(\cH)_\sa
\]
satisfies
\begin{itemize}
\item[\rm(a)] $\Phi$ enjoys operator homogeneity (Definition \ref{D-4.1}(2)),
\item[\rm(b)] for each $\cH$, $\Phi_\cH$ is well behaved with respect to direct sums
(in the sense of Proposition \ref{P-4.5}),
\item[\rm(c)] for each $\cH$, if $(A_n,B_n) \to (A,B)$ in SOT as $n\to\infty$ and if
$A_n + B_n \geq \eps I$ for all $n$ with some $\eps > 0$, then $\Phi_\cH(A_n,B_n) \to \Phi_\cH(A,B)$
in SOT as $n\to\infty$,
\item[\rm(d)] $\Phi_\cH(A_{\eps},B_{\eps}) \to \Phi_\cH(A,B)$ in SOT as
$\eps\searrow0$ for any $(A,B) \in B(\cH)_+\times B(\cH)_+$ and any $\cH$,
\end{itemize}
if and only if there exists a (unique) $\mathbb{R}$-valued homogeneous and continuous
function $\phi$ on $[0,\infty)^2$ such that $\Phi$ coincides with the PW-functional calculus
associated with $\phi$.
\end{theorem}

\begin{proof}
Let us first show that the PW-functional calculus associated with a function $\phi$ as above
satisfies (a)--(d). Item (a) is in Definition \ref{D-4.1} and (b) is contained in Proposition \ref{P-4.5}.
Item (c) is similar to \cite[Proposition 8]{HU} and can be confirmed as follows. Since
$A_n + B_n \geq \eps I$ for all $n$ as well as $A+B \geq \eps I$, one observes that
$T_{A_n,B_n}\to T_{A,B}$ and $R_{A_n,B_n}\to R_{A,B}$ in SOT as $n \to \infty$.
Consequently, if $\phi(x,y)$ is $\bR$-valued and continuous on $[0,\infty)^2$, then one has,
by Lemma \ref{L-4.3}(3),
\begin{align*}
\phi(A_n,B_n) &= (A_n +B_n)^{1/2}\phi(R_{A_n,B_n},1-R_{A_n,B_n})(A_n+B_n)^{1/2} \\
&\longrightarrow\, (A+B)^{1/2}\phi(R_{A,B},1-R_{A,B})(A+B)^{1/2} = \phi(A,B)
\end{align*}
in SOT as $n\to\infty$. For item (d), see Corollary \ref{C-6.2} (or \cite[Theorem 6]{HU}). 

Next we prove the converse direction. Assume that $\Phi$ satisfies (a)--(d). Define an
$\bR$-valued function $\phi$ on $[0,\infty)^2$ by $\phi(x,y):=\Phi_\bC(x,y)$ for $x,y\in[0,\infty)$,
which is homogeneous by condition (a). Then from (c) (for $\cH=\bC$) it is clear that $\phi$ is
continuous on $[0,\infty)^2\setminus\{(0,0)\}$. From homogeneity this implies that $\phi$ is
continuous at $(0,0)$ too. Now let $A = \sum_{n=1}^N \alpha_n P_n$ and
$B = \sum_{n=1}^N \beta_n P_n$ with orthogonal projections $P_n$ where
$\sum_nP_n=I_\cH$. Item (b) implies that
\[
\Phi_\mathcal{H}(A,B) = \sum_{i=1}^N \phi_{P_n\mathcal{H}}(\alpha_n P_n,\beta_n P_n).  
\]
For each $n$ fixed, choose an orthonormal basis $\{\xi_i\}_{i\in I}$ of $P_n\mathcal{H}$, and
consider isometries $V_i : \bC \to \cH$ sending any scalar $\lambda$ to $\lambda\xi_i$. Then
we have
\begin{align*}
\Phi_{P_n\cH}(\alpha_n P_n,\beta_n P_n) 
&=\sum_{i\in I} \Phi_{\mathbb{C}\xi_i}(\alpha_n V_i^* V_i,\beta_n V_i^* V_i)
\quad\mbox{(by (b))} \\
&=\sum_{i\in I}\Phi_\mathbb{C}(\alpha_n,\beta_n)V_i^* V_i
\quad\mbox{(by (a))}\\
&=\phi(\alpha_n,\beta_n)P_n. 
\end{align*}
Therefore, we see that $\Phi_\mathcal{H}(A,B) = \phi(A,B)$ holds for any pair $(A,B)$ specified
as above. Using a standard approximation procedure with condition (c), we have
$\Phi_\cH(A_\eps,B_\eps) = \phi(A_\eps,B_\eps)$ for any commuting pair
$(A,B) \in B(\cH)_+\times B(\cH)_+$ and $\eps>0$. Item (d) guarantees that the same identity
holds even when $\eps=0$. Thanks to the operator homogeneity of $\Phi$ (in (a)) and of the
PW-functional calculus associated with $\phi$, we obtain the desired conclusion.
\end{proof}

\begin{remark}\label{R-10.2}\rm
In view of Theorem \ref{T-6.1}, items (c) and (d) together can be replaced in Theorem \ref{T-10.1}
with the following single condition:
\begin{itemize}
\item if $A,B,A_n,B_n\in B(\cH)_+$, $A_n\searrow A$ and $B_n\searrow B$, then
$\Phi(A_n,B_n)\to\Phi(A,B)$ in SOT.
\end{itemize}
To confirm this modification, we only need to show that $\phi(x,y):=\Phi_\bC(x,y)$ is continuous
on $[0,\infty)^2$. But this can easily be seen from
\[ 
\phi(x,y) = 
\begin{cases} 
xy\phi(1/y,1/x) & (x,y>0), \\
x\phi(1,y/x) & (x>0), \\
y\phi(x/y,1) & (y>0)
\end{cases}
\] 
by homogeneity. 
\end{remark}

We emphasize that the postulate of `extending usual functional calculus' in Definition \ref{D-4.1}(1)
is implemented by three items (b)--(d) in the above theorem. In this way, the $\bR$-valued
continuous PW-functional calculus can be axiomatized like Kubo and Ando's operator connections,
though operator homogeneity is much stronger than the transformer inequality as a postulate.
From this viewpoint, Theorem\ref{T-10.1} might be understood as a result on non-commutative
functions.

Throughout the rest of the section, we assume that $\cH$ is a (fixed) infinite-dimensional Hilbert
space whenever otherwise stated. Below we will discuss axiomatic characterizations of
Kubo and Ando's type for extended operator perspectives $\phi_f$ for $f\in\OC(0,\infty)$.
For a map $\Phi:B(\cH)_+\times B(\cH)_+\to\widehat{B(\cH)}_\lb$ we consider the following
conditions:
\begin{itemize}
\item[(i)] (\emph{Joint operator convexity})\enspace
$\Phi(A_1+A_2,B_1+B_2)\le\Phi(A_1,B_1)+\Phi(A_2,B_2)$ for all $A_i,B_i\in B(\cH)_+$.
\item[(ii)] (\emph{Transformer inequality})\enspace
$\Phi(CAC,CBC)\le C\Phi(A,B)C$ for all $A,B,C\in B(\cH)_+$.
\item[(iii)] (\emph{Specialized upper continuity})\enspace
$\lim_{\eps\searrow0}\Phi(A_\eps,B_\eps)(\rho)=\Phi(A,B)(\rho)$ for all $A,B\in B(\cH)_+$
and $\rho\in B(\cH)_*^+$.
\end{itemize}

Items (i)--(iii) may be regarded as the counterparts of the above (I)--(III). If $\Phi$ satisfies
(i) and (ii), then letting $C=0$ in (ii) implies that $\Phi(0,0)\le0$. Since (i) gives
$\Phi(0,0)\le2\Phi(0,0)$, we have
\begin{align}\label{F-10.1}
\Phi(0,0)=0.
\end{align}
Note that if $C\in B(\cH)_{++}$ then the transformer inequality in (ii) becomes equality
automatically. In particular, (ii) implies the scalar homogeneity
\begin{align}\label{F-10.2}
\Phi(\alpha A,\alpha B)=\alpha\Phi(A,B),\qquad\alpha\ge0,
\end{align}
where the case $\alpha=0$ follows from \eqref{F-10.1}. Item (i) is subadditivity to be precise, but
with the homogeneity \eqref{F-10.2} this is equivalent to the genuine joint operator convexity.

To prove our axiomatization theorem, we need some more technical conditions as follows:
\begin{itemize}
\item[(iv)] (\emph{Special boundedness})\enspace $\Phi(tI,I)\in B(\cH)_\sa$ for all $t\in(0,\infty)$.
\item[(v)] (\emph{Local upper continuity})\enspace If $X_n\in B(\cH)_+$ ($n\in\bN$),
$X_1\ge X_2\ge\cdots$ and $\|X_n\|\to0$, then
$\lim_{n\to\infty}\Phi(I+X_n,I)(\rho)=\Phi(I,I)(\rho)$ for all $\rho\in B(\cH)_*^+$.
\end{itemize}
Item (iv) is reasonable for our purpose, and we need (v) to obtain an additional continuity property
that cannot be covered by (iii), being not so strong as (III).

We are now in a position to state the main theorem of the section.

\begin{theorem}\label{T-10.3}
A map $\Phi:B(\cH)_+\times B(\cH)_+\to\widehat{B(\cH)}_\lb$ satisfies conditions  {\rm(i)--(v)} if
and only if there exists a (unique) $f\in\OC(0,\infty)$ such that $\Phi(A,B)=\phi_f(A,B)$ for all
$A,B\in B(\cH)_+$, where $\phi_f$ is given in Definitions \ref{D-4.1} and \ref{D-7.1}.
\end{theorem}

\begin{proof}
As for the ``if\,'' part, when $\Phi=\phi_f$ with $f$ as stated above, conditions (i)--(iii) are
guaranteed by Theorems \ref{T-4.9} (also \ref{T-7.2}(1)) and \ref{T-7.7}, and conditions (iv), (v)
are obviously satisfied as well.

To prove the ``only if\,'' part, assume that $\Phi$ satisfies (i)--(v). First let us show that if a projection
$P\in B(\cH)$ commutes with $A,B\in B(\cH)_+$, then
\begin{align}\label{F-10.3}
P\Phi(A,B)P=P\Phi(AP,BP)P.
\end{align}
From (i) and (ii) one has
\begin{equation}\label{F-10.4}
\begin{aligned}
\Phi(A,B)&=\Phi(AP+AP^\perp,BP+BP^\perp) \\
&\le\Phi(AP,BP)+\Phi(AP^\perp,BP^\perp) \\
&=\Phi(PAP,PBP)+\Phi(P^\perp AP^\perp,P^\perp BP^\perp) \\
&\le P\Phi(A,B)P+P^\perp\Phi(A,B)P^\perp.
\end{aligned}
\end{equation}
Multiplying $P$ from both sides of the first inequality above gives
\[
P\Phi(A,B)P\le P\Phi(AP,BP)P+P\Phi(AP^\perp,BP^\perp)P.
\]
Moreover, since $\Phi(AP^\perp,BP^\perp)\le P^\perp\Phi(A,B)P^\perp$ by (ii), one has
$P\Phi(AP^\perp,BP^\perp)P\le0$ so that
\[
P\Phi(A,B)P\le P\Phi(AP,BP)P\le P\Phi(A,B)P.
\]
Therefore, \eqref{F-10.3} follows.

Now assume that $\Phi(A,B)\in B(\cH)_\sa$; so $\Phi(AP,BP)\in B(\cH)_\sa$ as well. For any
projection $P$ commuting with $A,B$, we write
$\Phi(A,B)=\begin{bmatrix}X&Z\\Z^*&Y\end{bmatrix}$ with $X\in PB(\cH)P$ and
$Y\in P^\perp B(\cH)P^\perp$. Then \eqref{F-10.4} means that
$\begin{bmatrix}X&Z\\Z^*&Y\end{bmatrix}\le\begin{bmatrix}X&0\\0&Y\end{bmatrix}$, that is,
$\begin{bmatrix}0&-Z\\-Z^*&0\end{bmatrix}\ge0$, which implies that $Z=0$. Hence $P$
commutes with $\Phi(A,B)$ in this case. Similarly $P$ commutes with $\Phi(AP,BP)$, so that
\eqref{F-10.3} is written as
\begin{align}\label{F-10.5}
\Phi(A,B)P=\Phi(AP,BP)P.
\end{align}

In particular, when $(A,B)=(tI,I)$ for $t\in(0,\infty)$, from (iv) and the above argument it follows
that $\Phi(tI,I)$ commutes with all projections in $B(\cH)$. Therefore, for each $t\in(0,\infty)$,
$\Phi(tI,I)$ must be written as $\Phi(tI,I)=f(t)I$ for some $f(t)\in\bR$. Now, consider
$A\in B(\cH)_{++}$ whose spectral decomposition is of the form $A=\sum_{i=1}^n\lambda_iP_i$
where $\lambda_i>0$ and $\sum_{i=1}^nP_i=I$. By (i) and (ii) we have
\[
\Phi(A,I)\le\sum_{i=1}^n\Phi(\lambda_iP_i,P_i)\le\sum_{i=1}^nP_i\Phi(\lambda_iI,I)P_i,
\]
so that that $\Phi(A,I)\in B(\cH)_\sa$. Hence we find that $\Phi(A,I)$ commutes with
all $P_i$ and
\begin{equation}\label{F-10.6}
\begin{aligned}
\Phi(A,I)&=\sum_{i=1}^n\Phi(A,I)P_i=\sum_{i=1}^n\Phi(AP_i,P_i)P_i
\quad\mbox{(by \eqref{F-10.5})} \\
&=\sum_{i=1}^n\Phi(\lambda_iP_i,P_i)P_i=\sum_{i=1}^n\Phi(\lambda_iI,I)P_i
\quad\mbox{(by \eqref{F-10.5})} \\
&=\sum_{i=1}^nf(\lambda_i)P_i=f(A).
\end{aligned}
\end{equation}

By (i) we notice that the $\bR$-valued function $f$ given above is convex on $(0,\infty)$ and
hence continuous on $(0,\infty)$, so one can define continuous functional calculus $f(A)$
for all $A\in B(\cH)_{++}$. Next we show that $\Phi(A,I)=f(A)$ holds for all $A\in B(\cH)_{++}$.
For any such $A$, by approximating the spectral decomposition of $A$, one can choose an
increasing sequence $A_n$ and a decreasing sequence $A_n'$ in $B(\cH)_{++}$ such that
\[
A_n=\sum_{i=1}^{m_n}\lambda_{n,i}P_{n,i},\quad
A_n'=\sum_{i=1}^{m_n}\lambda_{n,i}'P_{n,i}\quad
\mbox{where}\ \ \sum_{i=1}^{m_n}P_{n,i}=I,
\]
\[
A_n\le A\le A_n',\quad X_n:=A-A_n\le n^{-1}I,\quad X_n':=A_n'-A\le n^{-1}I.
\]
Let $\rho\in B(\cH)_*^+$ be arbitrary and choose a $\delta\in(0,1)$ such that $A_n \geq \delta I$
for all $n$. Using (i), \eqref{F-10.2} and \eqref{F-10.6}, we have
\begin{align*}
\Phi(A,I)(\rho)&=\Phi(A_n+X_n,I)(\rho) \\
&=\Phi(A_n-\delta I+X_n+\delta I,(1-\delta)I+\delta I)(\rho) \\
&\le\Phi(A_n-\delta I,(1-\delta)I)(\rho)+\Phi(X_n+\delta I,\delta I)(\rho) \\
&=(1-\delta)\Phi\Bigl({A_n-\delta I\over1-\delta},I\Bigr)(\rho)
+\delta\Phi(I+\delta^{-1}X_n,I)(\rho) \\
&=(1-\delta)\rho\Bigl(f\Bigl({A_n-\delta I\over1-\delta}\Bigr)\Bigr)
+\delta\Phi(I+\delta^{-1}X_n,I)(\rho).
\end{align*}
Note that $X_1\ge X_2\ge\cdots$, $\|\delta^{-1}X_n\|\le\delta^{-1}n^{-1}\to0$ and
\[
\Big\|{A_n-\delta I\over1-\delta}-{A-\delta I\over1-\delta}\Big\|={\|A_n-A\|\over1-\delta}
\,\longrightarrow\,0
\]
as $n\to\infty$ with $\delta$ fixed. Hence, by (v) we find that
\[
\Phi(A,I)(\rho)\le(1-\delta)\rho\Bigl(f\Bigl({A-\delta I\over1-\delta}\Bigr)\Bigr)
+\delta\rho(\Phi(I,I)).
\]
Letting $\delta\searrow0$ gives $\Phi(A,I)(\rho)\le\rho(f(A))$ for all $\rho\in B(\cH)_*^+$, so that
$\Phi(A,I)\le f(A)$. On the other hand, we have
\begin{align*}
(1+\delta)\rho\Bigl(f\Bigl({A_n'+\delta I\over1+\delta}\Bigr)\Bigr)
&=(1+\delta)\rho\Bigl(\Phi\Bigl({A_n'+\delta I\over1+\delta},I\Bigr)\Bigr) \\
&=\rho(\Phi(A_n'+\delta I,(1+\delta)I)) \\
&=\rho(\Phi(A+X_n'+\delta I,I+\delta I)) \\
&\le\Phi(A,I)(\rho)+\Phi(X_n'+\delta I,\delta I)(\rho) \\
&=\Phi(A,I)(\rho)+\delta\Phi(I+\delta^{-1}X_n',I)(\rho).
\end{align*}
Since $X_1'\ge X_2'\ge\cdots$, $\|\delta^{-1}X_n'\|\to0$ as $n\to\infty$, we find similarly to the
above argument that
\[
(1+\delta)\rho\Bigl(f\Bigl({A+\delta I\over1+\delta}\Bigr)\Bigr)
\le\Phi(A,I)(\rho)+\delta\rho(\Phi(I,I)).
\]
Letting $\delta\searrow0$ gives $f(A)\le\Phi(A,I)$. Hence $\Phi(A,I)=f(A)$ has been
shown for all $A\in B(\cH)_{++}$. Furthermore, this implies by (i) that $f$ belongs to
$\OC(0,\infty)$.

For any $A,B\in B(\cH)_{++}$, by (ii) and the case proved above, we have
\[
\Phi(A,B)=B^{1/2}\Phi(B^{-1/2}AB^{-1/2},I)B^{1/2}=B^{1/2}f(B^{-1/2}AB^{-1/2})B^{1/2}
=\phi_f(A,B).
\]
Finally, let $A,B\in B(\cH)_+$ be arbitrary. For every $\rho\in B(\cH)_*^+$ we use (iii) to
see that
\[
\Phi(A,B)(\rho)=\lim_{\eps\searrow0}\Phi(A_\eps,B_\eps)(\rho)
=\lim_{\eps\searrow0}\phi_f(A_\eps,B_\eps)(\rho)=\phi_f(A,B)(\rho),
\]
where the last equality is due to Theorem \ref{T-7.7}. Hence $\Phi=\phi_f$ has been shown.
The uniqueness of $f$ is clear from $\Phi(tI,I)=f(t)I$, $t>0$.
\end{proof}

\begin{remark}\label{R-10.4}\rm
(1)\enspace
Theorem \ref{T-10.3} holds also when condition (v) is replaced with the following (vi) (with
keeping conditions (i)--(iv)):
\begin{itemize}
\item[(vi)] (\emph{Local boundedness})\enspace For any $\rho\in B(\cH)_*^+$,
$X\mapsto\Phi(I+X,I)(\rho)$ is bounded above on some open ball
$U_\eps:=\{X\in B(\cH)_\sa:\|X\|<\eps\}$, i.e., $\sup_{X\in U_\eps}\Phi(I+X,I)(\rho)<\infty$
for some $\eps\in(0,1)$.
\end{itemize}
Indeed, it is clear that (vi) holds when $\Phi=\phi_f$ with $f\in\OC(0,\infty)$. Conversely, assume
that $\Phi$ satisfies (i)--(iv) and (vi). For any $\rho\in B(\cH)_*^+$ it follows from (i) and (vi) that
$X\mapsto\Phi(I+X,I)(\rho)$ is convex and bounded above on some open ball $U_\eps$. This
implies (see, e.g., \cite[Proposition I.2.5]{ET}) that $X\mapsto\Phi(I+X,I)(\rho)$ is continuous at
$X=0$ in the operator norm. Hence the above proof of Theorem \ref{T-10.3} can be carried out by
using (vi) in place of (v).

(2)\enspace
Assume that $\cH$ is finite-dimensional with $n=\dim\cH$. 
We can carry out the above proof of Theorem \ref{T-10.3} without the approximation procedure
in the paragraph after \eqref{F-10.6} (hence without condition (v)).  However, in this case,
we can only conclude that $f$ is \emph{$n$-convex} on $(0,\infty)$ (i.e., the operator inequality
as in \eqref{F-3.2} holds for $n\times n$ positive definite matrices $A,B$), instead of
$f\in\OC(0,\infty)$. Conversely, when $f$ is only $n$-convex on $(0,\infty)$, it does not seem
possible to show basic properties (for instance, (ii) and (iii) above) of $\phi_f$.
\end{remark}

Theorem \ref{T-6.3} in particular says that if $f\in\OC(0,\infty)$ with $f'(\infty)<\infty$
(resp., $f(0^+)<\infty$), then $\phi_f(A,B)$ is bounded for all
$(A,B)\in(B(\cH)_+\times B(\cH)_+)_\ge$ (resp., $(A,B)\in(B(\cH)_+\times B(\cH)_+)_\le$). The
same holds true when $f$ is an operator concave function on $(0,\infty)$ with $f'(\infty)>-\infty$
(resp., $f(0^+)>-\infty$). The following proposition is a modification of Theorem \ref{T-10.3} to the
restricted domain case.

\begin{proposition}\label{P-10.5}
Let $\cB:=(B(\cH)_+\times B(\cH)_+)_\le$ (resp., $\cB:=(B(\cH)_+\times B(\cH)_+)_\ge$). A map
$\Phi:\cB\to B(\cH)_\sa$ satisfies
\begin{itemize}
\item[\rm(i$'$)] $\phi(A_1+A_2,B_1+B_2)\le\Phi(A_1,B_1)+\Phi(A_2,B_2)$ for all $(A_i,B_i)\in\cB$,
\item[\rm(ii$'$)] $\Phi(CAC,CBC)\le C\Phi(A,B)C$ for all $(A,B)\in\cB$ and $C\in B(\cH)_+$,
\item[\rm(iii$'$)] $\Phi(A_\eps,B_\eps)\to\Phi(A,B)$ in SOT as $\eps\searrow0$ for all
$(A,B)\in\cB$,
\item[\rm(iv$'$)] if $X_n\in B(\cH)_+$, $X_1\ge X_2\ge\cdots$ and $\|X_n\|\to0$, then
$\Phi(X_n,I)\to\Phi(0,I)$ (resp., $\Phi(I,X_n)\to\Phi(I,0)$) in SOT,
\end{itemize}
if and only if there exists a (unique) $f\in\OC(0,\infty)$ with $f(0^+)<\infty$ (resp., $f'(\infty)<\infty$)
such that $\Phi(A,B)=\phi_f(A,B)$ for all $(A,B)\in\cB$.
\end{proposition}

\begin{proof}
We may prove the case $\cB:=(B(\cH)_+\times B(\cH)_+)_\le$, which implies the other case by
considering $\widetilde\Phi(A,B):=\Phi(B,A)$ and the transpose $\widetilde f$. For the ``if\,'' part,
items (i$'$) and (ii$'$) are the restrictions of those in Theorem \ref{T-10.3}, (iii$'$) is contained in
Theorem \ref{T-6.3} (also \cite[Theorem 6.2]{HSW}), and (iv$'$) is obvious since
$\phi_f(X,I)=f(X)$ (where $f$ is extended to $[0,\infty)$ with $f(0)=f(0^+)$) for all $X\in B(\cH)_+$.

For the ``only if\,'' part, the proof is similar to that of Theorem \ref{T-10.3} under restricting $(A,B)$
to $\cB$. The only place where we need to modify is the paragraph after \eqref{F-10.6}; so only this
part will be explained below. Let $A_n$, $A_n'$, $X_n$ and $X_n'$ be chosen as before. For
any $\delta\in(0,1)$ fixed, we have
\begin{align*}
\Phi(A,I)&=\Phi(A_n'+X_n',I)\le\Phi(A_n',(1-\delta)I)+\Phi(X_n',\delta I) \\
&=(1-\delta)f\Bigl({A_n'\over1-\delta}\Bigr)+\delta\Phi(\delta^{-1}X_n',I)
\,\longrightarrow\,(1-\delta)f\Bigl({A\over1-\delta}\Bigr)+\delta\Phi(0,I)
\end{align*}
in SOT as $n\to\infty$, so that
\[
\Phi(A,I)\le(1-\delta)f\Bigl({A\over1-\delta}\Bigr)+\delta\Phi(0,I).
\]
Letting $\delta\searrow0$ gives $\Phi(A,I)\le f(A)$. On the other hand, we have
\begin{align*}
(1+\delta)f\Bigl({A_n\over1+\delta}\Bigr)
&=(1+\delta)\Phi\Bigl({A_n\over1+\delta},I\Bigr)=\Phi(A+X_n,(1+\delta)I) \\
&\le\Phi(A,I)+\Phi(X_n,\delta I)=\Phi(A,I)+\delta\Phi(\delta^{-1}X_n,I).
\end{align*}
Letting $n\to\infty$ and then $\delta\searrow0$ gives $f(A)\le\Phi(A,I)$. Hence $\Phi(A,I)=f(A)$
follows, so that we have $\Phi=\phi_f$ for some $f\in\OC(0,\infty)$ as before. Finally, the
additional condition $f(0^+)<\infty$ is obvious since $\phi_f(0,I)=f(0^+)I$ is bounded.
\end{proof}

\begin{remark}\label{R-10.6}\rm
(1)\enspace
Convergence in SOT in (iii$'$) and (iv$'$) can be replaced with convergence in weak operator
topology. Also, condition (iv) in Theorem \ref{T-10.3} is available in Proposition \ref{P-10.5}
in place of (iv$'$).

(2)\enspace
In view of Theorem \ref{T-6.3}, items (iii$'$) and (iv$'$) together can be replaced with the
following stronger condition:
\begin{itemize}
\item if $A,B,A_n,B_n\in B(\cH)_+$, $A_n\le\alpha B_n$ (resp., $A_n\ge\alpha B_n$) for all $n$
with some $\alpha>0$ (independent of $n$), $A_n\searrow A$ and $B_n\searrow B$, then
$\Phi(A_n,B_n)\to\Phi(A,B)$ in SOT.
\end{itemize}

(3)\enspace
The operator concavity version of Proposition \ref{P-10.5} holds too, where the inequality signs in
(i$'$) and (ii$'$) are reversed and $f\in\OC(0,\infty)$ is replaced with an operator concave function
$f$ on $(0,\infty)$. This variant is immediately seen by taking $-\Phi$ and $-f$ in
Proposition \ref{P-10.5}.
\end{remark}

The following variant of Kubo and Ando's axiomatic characterization of operator connections is
worth giving. This is seen from Remark \ref{R-10.6}(2) and (3) because a non-negative function on
$(0,\infty)$ is operator monotone if and only if it is operator concave. In (III$'$) we assume a
simple convergence $A_n\sigma B_n\to A\sigma B$ in SOT (not necessarily decreasing as in (III)),
while decreasing convergence holds in (III$'$) as a consequence.

\begin{corollary}\label{C-10.7}
A map $\sigma:B(\cH)_+\times B(\cH)_+\to B(\cH)_+$ is an operator connection if and
only if $\sigma$ satisfies the following conditions:
\begin{itemize}
\item[\rm(I$'$)] \emph{(Joint operator concavity)}\enspace
$(A_1+A_2)\sigma(B_1+B_2)\ge(A_1\sigma B_1)+(A_2\sigma B_2)$ for all $A_i,B_i\in B(\cH)_+$.
\item[\rm(II$'$)] \emph{(Transformer inequality)}\enspace
$C(A\sigma B)C\le(CAC)\sigma(CBC)$ for all $A,B,C\in B(\cH)_+$.
\item[\rm(III$'$)] \emph{(Upper continuity)}\enspace
If $A_n,B_n\in B(\cH)_+$, $A_n\searrow A$ and $B_n\searrow B$, then
$A_n\sigma B_n\to A\sigma B$ in SOT.
\end{itemize}
\end{corollary}

\subsection*{Acknowledgments}
The work of F. Hiai and Y. Ueda was supported in part by JSPS KAKENHI Grant Numbers
JP17K05266 and JP18H01122, respectively.


\addcontentsline{toc}{section}{References}

\begin{thebibliography}{99}

\bibitem{AMY}
J. Agler, J. E. McCarthy and N. J. Young, {\it Operator Analysis: Hilbert Space Methods in
Complex Analysis}, Cambridge Tracts in Mathematics, Cambridge University Press,
Cambridge, 2020.

\bibitem{AD}
W. N. Anderson, Jr. and R. J. Duffin, Series and parallel addition of matrices,
{\it J. Math. Anal. Appl.} {\bf 26} (1969), 576--594.

\bibitem{AT}
W. N. Anderson, Jr. and G. E. Trapp, Shorted operators. II,
{\it SIAM J. Appl. Math.} {\bf 28} (1975), 60--71.

\bibitem{An2}
T. Ando, Lebesgue-type decomposition of positive operators,
{\it Acta Sci. Math. (Szeged)} {\bf 38} (1976), 253--260.

\bibitem{An0}
T. Ando, {\it Topics on Operator Inequalities}, Lecture notes (mimeographed),
Hokkaido Univ., Sapporo, 1978.

\bibitem{An1}
T. Ando, Concavity of certain maps on positive definite matrices and applications to
Hadamard products, {\it Linear Algebra Appl.} {\bf 26} (1979), 203--241.

\bibitem{AH}
T. Ando and F. Hiai, Operator log-convex functions and operator means,
{\it Math. Ann.} {\bf 350} (2011), 611--630.

\bibitem{Bh}
R. Bhatia, {\it Matrix Analysis}, Graduate Texts in Mathematics, 169. Springer-Verlag,
New York, 1997.

\bibitem{Bh2}
R. Bhatia, {\it Positive Definite Matrices}, Princeton Univ. Press, Princeton, 2007.

\bibitem{Do}
M. J. Donald, On the relative entropy. 
{\it Comm. Math. Phys.} {\bf 105} (1986), 13--34. 

\bibitem{ENG}
A. Ebadian, I. Nikoufar, and M. E. Gordji, Perspectives of matrix convex functions,
{\it Proc. Natl. Acad. Sci. USA} {\bf 108} (2011), 7313--7314.

\bibitem{Ef}
E. G. Effros, A matrix convexity approach to some celebrated quantum inequalities,
{\it Proc. Natl. Acad. Sci. USA} {\bf 106} (2009), 1006--1008.

\bibitem{EH}
E. Effros and F. Hansen, Non-commutative perspectives, {\it Ann. Funct. Anal.} {\bf 5}
(2014), 74--79.

\bibitem{ET}
I. Ekeland and R. T\'emam, {\it Convex Analysis and Variational Problems},
Corrected reprint of the 1976 English edition, Classics in Applied Mathematics, 28,
Society for Industrial and Applied Mathematics (SIAM), Philadelphia, PA, 1999.

\bibitem{FW}
P. Fillmore and J. Williams, On operator ranges,
{\it Adv. Math.} {\bf 7} (1971), 254--281.

\bibitem{FHR}
U. Franz, F. Hiai and \'E. Ricard, Higher order extension of L\"owner's theory: operator
$k$-tone functions, {\it Trans. Amer. Math. Soc.} {\bf 366} (2014), 3043--3074.

\bibitem{Fu1}
J. I. Fujii, Operator-concave functions and means of positive linear functionals.
{\it Math. Japon.} {\bf 25} (1980), 453--461.

\bibitem{Fu2}
J. I. Fujii, On Izumino's view of operator means, {\it Math. Japon.} {\bf 33} (1988),
671--675.

\bibitem{FK}
J. I. Fujii and E. Kamei, Relative operator entropy in noncommutative information theory,
{\it Math. Japon.} {\bf 34} (1989), 341--348.

\bibitem{FS}
J. I. Fujii and Y. Seo, The relative operator entropy and the Karcher mean,
{\it Linear Algebra Appl.} {\bf 542} (2018), 4--34. 

\bibitem{Ha}
U. Haagerup, Operator valued weights in von Neumann algebras, I,
{\it J. Funct. Anal.} {\bf 32} (1979), 175--206.

\bibitem{HP}
F. Hansen and G. K. Pedersen, Jensen's inequality for operators and L\"owner's theorem,
{\it Math. Ann.} {\bf 258} (1982), 229--241.

\bibitem{HU}
K. Hatano and Y. Ueda, Pusz--Woronowicz's functional calculus revisited,
Preprint (2020), arXiv:2012.13072 [math.FA].

\bibitem{Hi}
F. Hiai, Matrix analysis: matrix monotone functions, matrix means, and majorization,
{\it Interdiscip. Inform. Sci.} {\bf 16} (2010), 139--248.

\bibitem{Hi1}
F. Hiai, Quantum $f$-divergences in von Neumann algebras. I. Standard $f$-divergences,
{\it J. Math. Phys.} {\bf 59} (2018), 102202, 27 pp.

\bibitem{Hi2}
F. Hiai, Quantum $f$-divergences in von Neumann algebras. II. Maximal $f$-divergences,
{\it J. Math. Phys.} {\bf 60} (2019), 012203, 30 pp.

\bibitem{Hi3}
F. Hiai, {\it Quantum $f$-Divergences in von Neumann Algebras. Reversibility of Quantum
Operations}, Mathematical Physics Studies, Springer-Verlag, Singapore, 2021.

\bibitem{HK}
F. Hiai and H. Kosaki, Connections of unbounded operators and some related topics:
von Neumann algebra case, {\it Internat. J. Math.} {\bf 32} (2021), 2150024, 88 pp.

\bibitem{HL}
F. Hiai and Y. Lim, Operator means of probability measures,
{\it Adv. Math.} {\bf 365} (2020), 107038, 40 pp.

\bibitem{HM}
F. Hiai and M. Mosonyi, Different quantum $f$-divergences and the reversibility of
quantum operations, {\it Rev. Math. Phys.} {\bf 29} (2017), 1750023, 80 pp.

\bibitem{HMPB}
F. Hiai, M. Mosonyi, D. Petz and C. B\'eny, Quantum $f$-divergences and error
correction, {\it Rev. Math. Phys.} {\bf 23} (2011) 691--747;
Erratum: Quantum $f$-divergences and error correction, {\bf 29} (2017), 1792001.

\bibitem{HSW}
F. Hiai, Y. Seo and S. Wada, Ando--Hiai-type inequalities for operator means and operator
perspectives, {\it Internat. J. Math.} {\bf 31} (2020), 2050007, 44 pp.

%
\bibitem{KV}
D. S. Kaliuzhnyi-Verbovetskyi and V. Vinnikov, {\it Foundations of Free Noncommutative
Function Theory}, Mathematical Surveys and Monographs, 199, American Mathematical Society,
Providence, RI, 2014.

%
\bibitem{Ka}
T. Kato, {\it Perturbation Theory for Linear Operators}, Reprint of the 1980 edition,
Classics in Mathematics, Springer-Verlag, Berlin, 1995.

\bibitem{Ko1}
H. Kosaki, Remarks on Lebesgue-type decomposition of positive operators,
{\it J. Operator Theory} {\bf 11} (1984), 137--143.

\bibitem{Ko3}
H. Kosaki, Relative entropy of states: a variational expression,
{\it J. Operator Theory} {\bf 16} (1986), 335--348.

%
\bibitem{Ko4}
H. Kosaki, Parallel sum of unbounded positive operators,
{\it Kyushu J. Math.} {\bf 71} (2017), 387--405.

%
\bibitem{Ku}
F. Kubo, Conditional expectations and operations derived from network connections,
{\it J. Math. Anal. Appl.} {\bf 80} (1981), 477--489.

\bibitem{KA}
F. Kubo and T. Ando, Means of positive linear operators, {\it Math. Ann.} {\bf 246} (1980),
205--224.

\bibitem{LR}
A. Lesniewski and M. B. Ruskai, Monotone Riemannian metrics and relative entropy on
noncommutative probability spaces, {\it J. Math. Phys.} {\bf 40} (1999), 5702--5724.

\bibitem{Lo}
K. L\"owner, \"{U}ber monotone Matrixfunctionen, {\it Math. Z.} {\bf 38} (1934), 177--216.

\bibitem{Ped}
G. K. Pedersen, {\it Analysis Now},
Graduate Texts in Mathematics, 118, Springer-Verlag, New York, 1989.

\bibitem{Pe1}
D. Petz, Quasi-entropies for states of a von Neumann algebra,
{\it Publ. Res. Inst. Math. Sci.} {\bf 21} (1985), 787--800.

\bibitem{Pe2}
D. Petz, Sufficient subalgebras and the relative entropy of states of a von Neumann algebra,
{\it Comm. Math. Phys.} {\bf 105} (1986), 123--131.

\bibitem{PW1}
W. Pusz and S. L. Woronowicz, Functional calculus for sesquilinear forms and the
purification map, {\it Rep. Math. Phys.} {\bf 5} (1975), 159--170.

\bibitem{PW2}
W. Pusz and S. L. Woronowicz, Form convex functions and the WYDL and other inequalities,
{\it Lett. Math. Phys.} {\bf 2} (1978), 505--512.

\bibitem{Sch}
K. Schm\"udgen, {\it Unbounded Self-adjoint Operators on Hilbert Space}, Graduate Texts in
Mathematics, 265, Springer-Verlag, Dordrecht, 2012.

\bibitem{St}
S. Str\u atil\u a, {\it Modular Theory in Operator Algebras},
Editura Academiei and Abacus Press, Tunbridge Wells, 1981.

\bibitem{Uh}
A. Uhlmann, Relative entropy and the Wigner--Yanase--Dyson--Lieb concavity in an
interpolation theory, {\it Comm. Math. Phys.} {\bf 54} (1977), 21--32.

\bibitem{Wa}
S. Wada, When does Ando--Hiai inequality hold?,
{\it Linear Algebra Appl.} {\bf 540} (2018), 234--243.

\end{thebibliography}
\end{document}